\documentclass[journal]{IEEEtran}
\IEEEoverridecommandlockouts                              
\bstctlcite{IEEEexample:BSTcontrol}                                                          
\usepackage{calrsfs}
\DeclareMathAlphabet{\pazocal}{OMS}{zplm}{m}{n}
\usepackage{tikz}
\usepackage{tkz-berge}
\usetikzlibrary{fit,shapes}                                     
\usetikzlibrary{positioning}
\usepackage{setspace}

\usepackage{filecontents}
\usepackage{makecell}
\usepackage{diagbox}
\usepackage{graphics} 
\usepackage{graphicx}
\usepackage{amsmath, amssymb}

\usepackage{amsthm}
\usepackage{algorithm}
\usepackage{algorithmicx}
\usepackage{algpascal}
\usepackage{algc}
\usepackage{algcompatible}
\usepackage{algpseudocode}
\usepackage{setspace}
\usepackage{scalefnt}

\usepackage{enumitem}
\makeatletter
\def\BState{\State\hskip-\ALG@thistlm}
\makeatother
\usepackage{mathrsfs}
\DeclareMathAlphabet{\mathscrbf}{OMS}{mdugm}{b}{n}

\newcommand{\bA}{\bar{A}}
\newcommand{\bB}{\bar{B}}
\newcommand{\bC}{\bar{C}}
\newcommand{\bK}{\bar{K}}
\newcommand{\mN}{\mathcal{N}}
\newcommand{\E}{\pazocal{E}}

\newcommand{\B}{\pazocal{B}}
\newcommand{\C}{\pazocal{C}}
\newcommand{\D}{\pazocal{D}}
\renewcommand{\P}{\pazocal{P}}
\newcommand{\K}{\pazocal{K}}
\newcommand{\Ss}{\pazocal{S}}
\newcommand{\U}{\pazocal{U}}

\newcommand{\mC}{\mathcal{C}}
\newcommand{\bfC}{\mathscrbf{C}_{\rm set}}
\newtheorem{theorem}{Theorem}
\newtheorem{lem}{Lemma}

\newtheorem{defn}{Definition}
\newtheorem{cor}{Corollary}
\newtheorem{rem}{Remark}
\newtheorem{prob}{Problem}

\newtheorem{assume}{Assumption}
\newtheorem{prop}{Proposition}
\def \10n{\!\!\!\!\!\!\!\!\!\!}
\def \20n{\!\!\!\!\!\!\!\!\!\!\!\!\!\!\!\!\!\!\!\!}
\def \*{\star}

\makeatletter 
 \def\@eqnnum{{\normalsize \normalcolor (\theequation)}} 
  \makeatother
\renewcommand{\baselinestretch}{0.98}

\usepackage{amsmath} 
\usepackage{amssymb}  
\title{\LARGE \bf
Minimum Cost Feedback Selection in Structured Systems: Hardness and Approximation Algorithm
}

\author{Aishwary~Joshi, 
        Shana~Moothedath
        and~Prasanna~Chaporkar
\thanks{The authors Aishwary Joshi and Prasanna Chaporkar are in the Department of Electrical Engineering, Indian Institute of Technology Bombay, India and Shana Moothedath in University of Washington, Seattle. Email: $\lbrace$aishwary, chaporkar$\rbrace$@ee.iitb.ac.in,  $\lbrace$sm15$\rbrace$@uw.edu.}}

\begin{document}

\maketitle
\thispagestyle{empty}
\pagestyle{empty}
\begin{abstract}
In this paper, we study output feedback selection in linear time invariant structured systems. We assume that the inputs and the outputs are {\em dedicated}, i.e., each input directly actuates a single state and each output directly senses a single state. Given a structured system with dedicated inputs and outputs and a cost matrix that denotes the cost of each feedback connection, our aim is to select an optimal set of feedback connections such that the closed-loop system satisfies arbitrary pole-placement. This problem is referred as the {\em optimal feedback selection problem for dedicated i/o}. We first prove the NP-hardness of the problem using a reduction from a well known NP-hard problem, the {\em weighted set cover} problem. In addition, we also prove that the optimal feedback selection problem for dedicated i/o is {\em inapproximable} to a constant factor of log$\,n$, where $n$ denotes the system dimension. To this end, we propose an algorithm to find an approximate solution to the optimal feedback selection problem for dedicated i/o. The proposed algorithm consists of a {\em potential function} incorporated with a {\em greedy scheme} and attains a solution with a guaranteed approximation ratio. Then we consider two special network topologies of practical importance, referred as {\em back-edge feedback structure} and {\em hierarchical networks}. For the first case, which is NP-hard and inapproximable to a multiplicative factor of log$\,n$, we provide a (log$\,n$)-approximate solution, where $n$ denotes the system dimension. For hierarchical networks, we give a dynamic programming based algorithm to obtain an optimal solution in polynomial time. 
\end{abstract}
\vspace*{-2 mm}
\vspace*{-2 mm}
\section{Introduction}\label{sec:intro}
The emergence of large-scale networks as physical models capturing the structural properties of real networks presents new challenges in design, control and optimization. Large-scale dynamical systems have applications in diverse areas, including biological networks,
transportation networks, water distribution networks, multi-agent systems and
internet. Most of the real world networks are often too complex and of large system dimension that employing conventional control theoretic tools to analyse various properties of these systems are computationally infeasible. Recently, there has been immense research advance in the area of large-scale dynamical
systems collectively using concepts from various interdisciplinary fields including control
theory, network science and statistical physics. These studies emphasise on the relationship between the topology and the dynamics of complex networks. 

This paper deals with {\em feedback selection} in linear time invariant (LTI) systems. Feedback selection problem is a classical problem in control theory which resisted much advances due to the inherent hardness of the problem. We address the feedback selection problem for a complex system whose graph pattern is known and parameter values are unknown. More specifically, this paper discusses optimal feedback selection for {\em structured} LTI systems. {\it Given a structured system with specified state, input and output structures and a cost matrix that denotes the cost of each feedback connection, our objective is to design an optimal feedback matrix that satisfies arbitrary pole-placement of the closed-loop system.} The cost associated with the feedback connections comes from installation and monitoring cost associated with the network. The motivation for this problem comes from the recent interest and developments in the control of large-scale systems modeled with a very large number of variables, where implementing control strategies that affect all or many of the variables in the system is not economical or rather not feasible.

Structural analysis of dynamical systems is a well studied area since the introduction of {\em structural controllability} (see \cite{ComDio:15,ComDioWou:02,KalBelSiv:13,Ols:14,PeqKarAgu_2:16} as representatives). The strength of this analysis lies in the fact that many structural properties are `generic' in nature, i.e., these properties hold for {\em almost all} systems with the same structure \cite{ComDioWou:02,Rei:88}. Over last few decades, various design and optimization problems  in complex networks  are addressed using structural analysis in many papers. These papers mainly use concepts of bipartite matching and graph connectivity. 
 For a detailed reading on various problems in this area see \cite{LiuBar:16} and references therein.  

Optimal feedback selection for structured systems is previously addressed in many papers \cite{UnySez:89}. Given a structured state matrix, an optimal input-output and feedback co-design problem is addressed in \cite{PeqKarAgu_2:16}. As structure of input, output and feedback matrices are {\em unconstrained}, the problem considered in \cite{PeqKarAgu_2:16} is solvable in polynomial time complexity. Paper \cite{PeqKarAgu:16} considered the input-output and feedback co-design problem for constrained input, output and feedback structures. This problem turns out be NP-hard as a subproblem, namely the  constrained minimum input selection problem, is NP-hard \cite{PeqKarAgu:16}. Due to the NP-hardness of the problem,  the class of {\em irreducible}\footnote{A structured system is said to be irreducible if there exists a directed path between any two arbitrary nodes in the state digraph $\D(\bA)$ (see Section~\ref{sec:prelim}).} systems is considered in \cite{PeqKarAgu:16}. Later in \cite{MooChaBel:17_Automatica}, an order-optimal approximation algorithm is given for the input-output and feedback selection co-design problem.

This paper deals with optimal feedback selection for structured systems with specified state, input and output matrices. Note that, here input and output matrices are specified and there is no selection of inputs and outputs. The structure of the feedback pattern is constrained and each feedback edge is associated with a cost. Our aim is to design an optimal feedback matrix that satisfies the prescribed structure and also of minimum cost. Depending on the nature of inputs and outputs, {\em dedicated}\footnote{{An input (output, resp.) 
is said to be {\em dedicated} if it can actuate (sense, resp.) a single state only.}} and {\em non-dedicated}, and the nature of the costs of the feedback connections, {\em uniform} and {\em non-uniform}, different formulations of this problem is addressed before. Table below summarizes the associated results.
\vspace*{-2 mm}
\begin{table}[h]\caption{{\scalefont{0.9}{Algorithmic complexity results of the optimal feedback selection problem}}}\label{tb:NP}
\vspace*{-3 mm}
\begin{center}
\begin{tabular}{|p{21mm}|p{1.9cm}|p{3.0cm}|}
\hline
\diaghead{\theadfont Diag ColumnmnHea}%
{Input\\ and Output}{Feedback costs} & Uniform & Non-uniform \\
\hline
Dedicated & P \cite{MooChaBel:17_Erratum} & NP-hard (this paper)\\
\hline
Non-dedicated & NP-hard \cite{MooChaBel:17_Line} & NP-hard \cite{MooChaBel:17_Line}\\
\hline
\end{tabular}
\end{center}
\vspace*{-4 mm}
\end{table}
 {\em Constrained} feedback selection with  non-dedicated inputs and  outputs  is considered in \cite{MooChaBel:17_Line}.  In \cite{MooChaBel:17_Line}, the authors show the NP-hardness of the problem for non-dedicated i/o case and later propose a polynomial time algorithm for a special graph topology so-called line graphs. Optimal feedback selection problem with dedicated inputs and outputs and uniform cost feedback edges is considered in \cite{CarPeqAguKarPap:15}, \cite{MooChaBel:17_Erratum} and a polynomial time algorithm is given in \cite{MooChaBel:17_Erratum}.
{\em  In this paper, we consider optimal feedback selection for dedicated inputs and outputs and non-uniform cost feedback edges}.
\begin{rem}
NP-hardness result for non-dedicated i/o case relies heavily on the non-dedicated nature of i/o's, and hence the NP-hardness proved in \cite{MooChaBel:17_Line} does not automatically imply NP-hardness of a special case with dedicated i/o's. Note that, for non-dedicated i/o's, Problem~\ref{prob:one} is NP-hard even when the feedback costs are uniform. However, for uniform cost setting, the dedicated i/o case is solvable in polynomial time \cite{MooChaBel:17_Erratum}.
\end{rem}
\noindent In this scenario, we make the following contributions:
 
\noindent $\bullet$ We prove that the optimal feedback selection problem with dedicated input-output set and non-uniform cost feedback edges is NP-hard (Theorem~\ref{theorem:one}).\\
\noindent $\bullet$ We prove that the optimal feedback selection problem with dedicated inputs and outputs, and feedback edges with non-uniform cost is inapproximable to a multiplicative factor of log$\, n$, where $n$ denotes the system dimension (Theorem~\ref{theorem:two}).\\
\noindent $\bullet$ We propose an approximation algorithm  with a guaranteed approximation ratio for solving the problem (Algorithm~\ref{algo:four} and Theorem~\ref{theorem:four}).\\
\noindent $\bullet$ We show that the proposed algorithm has computational  complexity  polynomial in the number of cycles in the system digraph and the system dimension (Theorem~\ref{th:comp}).\\
\noindent $\bullet$ We consider a special class of systems with a constraint on the structure of the feedback matrix, referred as back-edge feedback selection, and propose an approximation algorithm to solve the problem with a guaranteed approximation ratio of log$\,n$, where $n$ denotes the system dimension (Algorithm~\ref{algo:five} and~Theorem~\ref{th:spec_1}).\\
\noindent $\bullet$ We consider another special class of systems referred as hierarchical networks and propose a polynomial time algorithm based on dynamic programming to obtain an optimal solution to the problem (Algorithm~\ref{algo:hierarchical} and Theorem~ \ref{theorem:spcase2}).

The organization of the rest of the paper is as follows: Section~\ref{sec:prob_form} gives the formulation of the optimization problem addressed in this paper. Section~\ref{sec:prelim} describes preliminaries and few existing results used in the sequel. Section~\ref{sec:NP} analyzes the complexity of the problem and proves the NP-hardness and the inapproximability of the problem. Section~\ref{sec:Reform} reformulates the problem to a graph theoretic equivalent. Section~\ref{sec:Approx} gives an approximation algorithm to solve the problem. Section~\ref{sec:spcases} explores two special topologies of structured systems and gives an approximation algorithm and an optimal algorithm, respectively, to solve the two cases. Finally, Section~\ref{sec:conclu} gives the concluding remarks and future directions.
\section{Problem Formulation}\label{sec:prob_form}
Consider an LTI system $\dot{x}=Ax {+}Bu$, $y {=} Cx$, where $A {\, \in\,} \mathbb{R}^{n\times n}$, $B{\, \in\,} \mathbb{R}^{n\times m}$ and $C {\, \in\,} \mathbb{R}^{p\times n}$. Here the matrices  ${A},{B}$ and ${C}$ denote the state, input and output matrices respectively and $\mathbb{R}$ denotes the set of real numbers. The structured matrices  $\bar{A},\bar{B}$ and $\bar{C}$ corresponding to this system are such that
{\scalefont{0.9}{
\begin{eqnarray}
\bar{A}_{ij}=0 \text{~whenever~} {A}_{ij}=0,\nonumber\\
\bar{B}_{ij}=0 \text{~whenever~} {B}_{ij}=0,\nonumber\\
\bar{C}_{ij}=0 \text{~whenever~} {C}_{ij}=0.\label{eq:sys}
\end{eqnarray}
}}
For $(A,B,C)$ that satisfy equation~\eqref{eq:sys}, $(\bA,\bB,\bC)$ is referred as the {\it structured system} of $(A,B,C)$ and the system $(A,B,C)$ is called a {\it numerical realization} of the structured system $(\bA,\bB,\bC)$. Here $\bA\in \{0,\star\}^{n\times n}$, $\bB\in \{0,\star\}^{n\times m}$ and $\bC\in \{0,\star\}^{p\times n}$. The $0$ entries in the structured system correspond to fixed zeros and the $\*$ entries correspond to unrelated indeterminate. Let $P \in \mathbb{R}^{m\times p}$ be a cost matrix, where $P_{ij}$ denotes the cost of feeding the $j^{\rm th}$ output to the $i^{\rm th}$ input. Our objective here is output feedback selection. A feedback edge is said to be {\em infeasible} if the corresponding output can not be fed to the corresponding input. All infeasible feedback connections are assigned infinity cost. In other words,  $P_{ij}=\infty$ implies that the $j^{\rm th}$ output can not be fed to the $i^{\rm th}$ input, or the feedback edge $(y_j,u_i)$ is infeasible. We define the feedback matrix $\bar{K}\in \{0,\star\}^{m\times p}$, where $\bar{K}_{ij}=\star$ only if $P_{ij}=\infty$. Our aim is to design an optimal output feedback matrix such that the closed-loop system guarantees arbitrary pole-placement. A graph theoretic necessary and sufficient condition for checking whether arbitrary pole-placement is feasible or not in a structured system is given in \cite{PicSezSil:84}. This condition depends on the existence of {\em{structurally fixed modes}} (SFMs) in the closed-loop structured system. Hence to address the pole-placement problem in structured systems, the concept of SFMs is used in this paper. Let $[K]:= \{K:K_{ij}=0$, if $\bar{K}_{ij}=0\}$. Structured systems with no SFMs are defined as follows:
\begin{defn}\label{defn:one}
The structured system ($\bar{A},\bar{B}, \bar{C}$) and feedback matrix $\bar{K}$ is said to have no structurally fixed modes if there exists a numerical realization (A,B,C) of ($\bar{A},\bar{B}, \bar{C}$) such that
$\cap_{K\in [K]} \sigma(A+BKC) = \emptyset$, where $\sigma(T)$ denotes the set of eigenvalues of a square matrix T.
\end{defn}
Given a structured system ($\bar{A},\bar{B}, \bar{C}$) and a cost matrix $P$, our aim is to find a minimum cost set of feedback edges such that the closed-loop system denoted by ($\bar{A},\bar{B}, \bar{C}, \bar{K}$) has no SFMs. The set of all feedback matrices $\bK$ that satisfies the no-SFMs criteria is denoted by the set $\K$. In other words, $\K:= \{\bK \in \{0,\star \}^{m\times p}$: $(\bar{A},\bar{B},\bar{C},\bar{K})$ has no SFMs$\}$ is the set of all feasible solutions to the optimization problem discussed in this paper. The cost associated with the feedback matrix $\bar{K}$ is denoted by $P(\bar{K})$, where $P(\bar{K}) = \sum_{(i,j):\bar{K}_{ij} = \star} P_{ij}$. The optimization problem addressed in this paper is given below.
\begin{prob}\label{prob:one}
Given a structured system $(\bA, \bB={\mathbb{I}_m}, \bC={\mathbb{I}_p})$, find $\bK^\* ~\in~ \arg\min\limits_{ \bK \in \K} P(\bar{K})$.
\end{prob}
Here $\mathbb{I}_m$ and $\mathbb{I}_p$ denote $m$ dedicated inputs and $p$ dedicated outputs, respectively. A dedicated input is an input which {\it actuates} a single state directly and a dedicated output is an output that {\it senses} a single state directly. Thus there is exactly one $\star$ entry in each column of $\mathbb{I}_m$ and exactly one $\star$ entry in each row of $\mathbb{I}_{p}$. Problem~\ref{prob:one} is referred to as the {\em optimal feedback selection for dedicated i/o problem}. If $P(\bar{K}^\*)= \infty$, then we say that arbitrary pole-placement is not possible for $(\bA, {\mathbb{I}_m},{\mathbb{I}_p})$ and cost matrix $P$. In the section below, we give few notations and preliminaries used in the sequel.
\section{Notations, Preliminaries and Existing Results}\label{sec:prelim}
For describing various graph theoretic conditions used in the analysis of structured systems, we first elaborate on few notations and constructions. 
A digraph $\D(\bA):=(V_X,E_X)$, where $V_X=\{x_1,\dots,x_n\}$ and an edge $(x_j,x_i) \in E_X$ if $\bA_{ij} =\star$. The edge $(x_j,x_i)$ directed from $x_j$ towards $x_i$ implies that state $x_j$ can {\em influence} state $x_i$. Hence the influence of states on other states is captured in the digraph $\D(\bA)$. Similarly, we define $\D(\bA,\bB,\bC):=(V_X \cup V_Y \cup V_U, E_X \cup E_Y \cup E_U)$, where $V_U=\{u_1,\dots,u_m\}$ and $V_Y=\{y_1,\dots,y_p\}$. An edge $(u_j,x_i)\in E_U$ if $\bB_{ij}=\star$ and an edge $(x_j,y_i)\in E_Y$ if $\bC_{ij}=\star$. Next, we define $\D(\bA,\bB,\bC,\bK):=(V_X \cup V_Y \cup V_U, E_X \cup E_Y \cup E_U\cup E_K)$, where a feedback edge $(y_j,u_i) \in E_K$ if $\bK_{ij}=\star$. Thus $\D(\bA,\bB,\bC,\bK)$ captures the influence of states, inputs, outputs and feedback connections. The digraphs $\D(\bA)$ and $\D(\bA,\bB,\bC,\bK)$ are referred to as the {\em state digraph} and the {\em closed-loop system digraph}, respectively. A digraph is said to be strongly connected if there exists a path from $v_i$ to $v_k$ for each ordered pair of vertices $(v_i,v_k)$ in the digraph. A strongly connected component (SCC) is a subgraph that consists of a maximal set of strongly connected vertices. Necessary and sufficient condition for the no-SFMs criteria is described below.
\begin{prop}\cite[Theorem 4]{PicSezSil:84}:\label{prop:one} 
A structured system ($\bar{A},\bar{B}, \bar{C}$) has no SFMs with respect to an information pattern $\bar{K}$ if and only if the following conditions hold:

\begin{enumerate}
\item[(a)] in the digraph $\D(\bar{A},\bar{B}, \bar{C}, \bar{K})$, each state node $x_i$ is contained in an SCC which includes an edge from $E_K$, 
\item[(b)] there exists a finite node disjoint union of cycles $C_g=(V_g,E_g)$ in $\D(\bar{A},\bar{B}, \bar{C}, \bar{K})$, where $g$ belongs to the set of natural numbers such that $V_X \subset \cup_gV_g.$ 
\end{enumerate}
\end{prop}
The conditions given in Proposition~\ref{prop:one} thus serve as conditions for checking existence of SFMs in the closed-loop system. For verifying condition~(a), one has to find all the SCCs in the digraph $\D(\bA,\bB,\bC,\bK)$. If each SCC has atleast one feedback edge present in it, then condition~(a) is satisfied. Concerning condition~(b), an equivalent matching\footnote{A matching is a set of edges such that no two edges share the same end point. For a bipartite graph $G_{\scriptscriptstyle B}=((V_{\scriptscriptstyle B}, V'_{\scriptscriptstyle B}), \E_{\scriptscriptstyle B})$, a perfect matching is a matching whose cardinality is equal to min($|V_{\scriptscriptstyle B}|,|V'_{\scriptscriptstyle B}|$).} condition using the bipartite\footnote{ A bipartite graph $G_{\scriptscriptstyle B}=((V_{\scriptscriptstyle B}, V'_{\scriptscriptstyle B}), \E_{\scriptscriptstyle B})$ is a graph satisfying $V_{\scriptscriptstyle B}\cap V'_{\scriptscriptstyle B}=\emptyset$ and $\E_{\scriptscriptstyle B} \subseteq V_{\scriptscriptstyle B}\times V'_{\scriptscriptstyle B}$.} graph $\B(\bA,\bB,\bC,\bK)$ exists \cite{MooChaBel:17_Automatica}. The construction of bipartite graph $\B(\bA,\bB,\bC,\bK)$ is as follows. We first define {\em state bipartite graph} $\B(\bA):=((V_{X'}, V_X), \pazocal{E}_{X})$, where $V_{X'}=\{x'_1,\dots,x'_n\}$, $V_X=\{x_1,\dots,x_n\}$ and $(x'_j,x_i)\in \pazocal{E}_X \Leftrightarrow (x_i,x_j)\in {E}_X$. Now, we define  $\B(\bA,\bB,\bC,\bK):=((V_{X'}\cup V_{U'}\cup V_{Y'},V_{X}\cup V_{U}\cup V_{Y}),\E')$, where $V_{U'}=\{u'_1,\dots,u'_m\}$, $V_{Y'}=\{y'_1,\dots,y'_p\}$, $V_{U}=\{u_1,\dots,u_m\}$, $V_{Y}=\{y_1,\dots,y_p\}$ and $\E'=(\E_{X}\cup \E_{U}\cup \E_{Y} \cup \E_{K} \cup \E_{\mathbb{U}}\cup \E_{\mathbb{Y}})$. Also, $(x'_i,u_j)\in \E_U \Leftrightarrow (u_j,x_i)\in E_U$, $(y'_j,x_i)\in \E_Y \Leftrightarrow (x_i,y_j)\in E_Y$ and $(u'_i,y_j)\in \E_K \Leftrightarrow (y_j,u_i)\in E_K$. Moreover, $\E_{\mathbb{U}}$ includes edges $(u'_i,u_i)$, for $i=1,\dots,m$ and $\E_{\mathbb{Y}}$ includes edges $(y'_i,y_i)$, for $i=1,\dots,p$. 

\begin{prop}\cite[Theorem~3]{MooChaBel:17_Automatica} \label{prop:two}
Consider a closed-loop structured system $(\bA, \bB, \bC, \bK)$. The bipartite graph $\B(\bA, \bB, \bC, \bK)$ has a perfect matching if and only if all state nodes are spanned by disjoint union of cycles in $\D(\bA, \bB, \bC, \bK)$.
\end{prop}
If $\B(\bar{A})$ has a perfect matching, then $\B(\bA,\bB,\bC,\bK)$ has a perfect matching without using any feedback edge. This implies that condition~(b) is satisfied without using any feedback edge. This is because in $\B(\bA,\bB,\bC,\bK)$, $(u'_i,u_i) \in \E_{\mathbb{ U}}$, for all $i \in \{1,\dots,m\}$, and $(y'_i,y_i) \in \E_{\mathbb{Y}}$, for all $i \in \{1,\dots,p\}$. 
Thus a perfect matching in $\B(\bA)$ is a sufficient condition for satisfying condition~(b).

Since $m=O(n)$ and $p=O(n)$, finding SCCs in $\D(\bA,\bB,\bC,\bK)$ has $O(n^2)$ complexity \cite{CorLeiRivSte:01}. Verifying condition (b) has a complexity $O(n^{2.5})$ using the matching condition given in Proposition~\ref{prop:two} \cite{Die:00}. Hence, given $(\bA,\bB,\bC)$ and feedback matrix $\bK$, verifying the conditions in Proposition~1 has complexity $O(n^{2.5})$. Our objective in this paper is to obtain an optimal (in the sense of cost) set of feedback connections that guarantees arbitrary pole-placement. In other words, we need to obtain an optimal set of feedback edges that satisfies the no-SFMs criteria. Even though verifying existence of SFMs is of polynomial complexity, identifying an optimal feedback matrix may not be computationally easy. Specifically, in large scale systems of huge system dimension, an exhaustive search based technique to obtain an optimal solution to Problem~\ref{prob:one} is not computationally feasible. Before proposing a framework to solve Problem~\ref{prob:one}, we first analyze the tractability of Problem~\ref{prob:one} in the section below.
\section{Complexity of Optimal Feedback Selection Problem with Dedicated Inputs and Outputs}\label{sec:NP}
In this section, we prove the NP-hardness of Problem 1. The hardness result is obtained using a reduction of a known NP-hard problem, {\em the weighted set cover problem}, to an instance of Problem~\ref{prob:one}. The weighted set cover problem is a standard NP-hard problem with numerous applications \cite{Chv:79}. It is described here for the sake of completeness.
Given a universe $\pazocal{U}$ of $N$ elements $\pazocal{U}=\{1,\dots,N\}$, and a collection of sets $\pazocal{P}=\{\Ss_1,\Ss_2,....\Ss_r\}$, where $\Ss_i \subseteq \U$ and $\cup_{\Ss_i\in \pazocal{P}} \Ss_i = \pazocal{U}$ and a weight function $w:\pazocal{P}\rightarrow \mathbb{R}$, the objective is to find a set $\pazocal{S^\*}\subseteq \pazocal{P}$ such that $\cup_{\Ss_i\in \Ss^\*} \Ss_i  =\pazocal{U}$ and $\sum_{\Ss_i\in \Ss^\*} w(\Ss_i) \leqslant \sum_{\Ss_i\in \widetilde{\Ss}} w(\Ss_i) $, where $\cup_{\Ss_i \in \widetilde{\Ss}} = \pazocal{U}$.
\begin{algorithm}[t]
\caption{Pseudo-code for reducing the weighted set cover to an instance of Problem~\ref{prob:one}}\label{algo:one}
\textbf{Input}: A weighted set cover problem with universe $\pazocal{U}=\{1,2,\dots,N\}$, sets $\pazocal{P}=\{\Ss_1,\dots,\Ss_r\}$ and a weight function $w$ associated with each set in $\P$\\
\textbf{Output}: A structured system $(\bar{A},\bar{B} = \mathbb{I}_m,\bar{C} = \mathbb{I}_p)$ and a feedback cost matrix $P$

\begin{algorithmic}[1]
\State We define a structured system $(\bar{A},\bar{B},\bar{C})$ as follows:
{\scalefont{0.9}{
\State $\bar{A}_{ij} \leftarrow\begin{cases} \star$, for $i=j$, $\\ \star$, for $i\in \{1,\dots,N\}$ and $j = N{+}r{+}1$,  $\\ \star$, for $i\in \{N{+}1,\dots,N{+}r\}, j \in \Ss_{i-N}, \\ 0$, otherwise.$ \label{step:A}
\end{cases}$
\State $\bar{B}_{ij} \leftarrow \begin{cases} \star$, for $i \in \{N{+}1,\dots,N{+}r{+}1\}$ and $j=i-N, \\ 0$, otherwise.$\label{step:B}
\end{cases} $
\State $\bar{C}_{ij} \leftarrow \begin{cases} \star$, for $j\in \{N{+}1\dots,N{+}r\}$ and $i=j-N, \\ 0,$ otherwise$.\label{step:C}
\end{cases}$
\State $P_{ij} \leftarrow \begin{cases} w(\Ss_{j}), j\in \{1,\dots\,r\}$ and $i=r{+}1, \\ 0,$ for $i,j \in \{1,\dots,r\}$ and $i{=}j, \\ \infty,$ otherwise. $\label{step:P}
\end{cases} $
\State Let $\bK$ be a solution to Problem~1 for $(\bar{A},\bar{B},\bar{C})$ and cost matrix $P$ constructed above\label{step:K}
\State Sets selected under $\bar{K}, \Ss(\bar{K}){\leftarrow}\{\Ss_{j-N}$: $\bar{K}_{ij}=\star$ $\&$ $i\neq j\}$\label{step:S}
\State Weight of the set $w(\Ss(\bar{K}))\leftarrow \sum_{\Ss_i\in \Ss(\bar{K})} w(\Ss_i)$\label{step:w}
}}
\end{algorithmic}
\end{algorithm}

The pseudo-code showing a polynomial time reduction of the weighted set cover problem to an instance of Problem~\ref{prob:one} is presented in Algorithm~\ref{algo:one}. From a general instance of the weighted set cover problem, we construct an instance of Problem~\ref{prob:one}. The structured system $(\bar{A},\bar{B},\bar{C})$ has states $x_1,\dots,x_{N+r+1}$, inputs $u_{1},\dots,u_{r+1}$ and outputs $y_{1},\dots,y_{r}$. The structured state matrix $\bar{A} \in \{0,\star\}^{(N{+}r{+}1)\times(N{+}r{+}1)}$ is constructed as follows. For ease of understanding, we refer $x_1,\dots,x_{N}$, as the {\em element nodes} and $x_{N+1},\dots,x_{N+r}$, as the {\em set nodes}. The {\em element nodes} correspond to the elements of the universe $\pazocal{U}$ and every node in $\{x_1,\dots,x_N\}$ has an edge from the node $x_{N+r+1}$. The {\em set nodes} correspond to the sets of the weighted set cover problem. A set node $x_{N+k}$ has an edge from element node $x_j$ if element $j \in \U$ belongs to set $\Ss_{k} \in \pazocal{P}$. This completes the construction of $\bA$ (Step~\ref{step:A}).

 The structured matrix $\bar{B} \in \{0,\star\}^{(N+r+1)\times (r+1)}$ corresponds to the $(r+1)$ dedicated input nodes which are fed to set nodes $x_{N+1},\dots,x_{N+r}$ and $x_{N+r+1}$ (Step~\ref{step:B}). The structured matrix $\bar{C} \in \{0,\star\}^{r\times(N+r+1)}$ corresponds to the $r$ dedicated output nodes which come out from the $r$ set nodes $x_{N+1},\dots,x_{N+r}$ respectively (Step~\ref{step:C}). Thus for the constructed structured system, $n=N+r+1$, $m=r+1$ and $p=r$. Corresponding to the $(r+1)$ inputs and $r$ outputs, the feedback cost matrix $P \in \mathbb{R_+}^{(r+1) \times r}$ is defined as follows. We assign $P_{ij}=0$, for $i,j\in \{1,\dots,r\}$ and $i=j$. For $i=r+1$ and $j\in \{1,\dots,r\}$,  $P_{ij}$ is assigned the weight of the set $\Ss_{j}$ (Step~\ref{step:P}). The motive for defining such a feedback cost structure is the following. In a solution to Problem~\ref{prob:one}, if we select a feedback edge connecting output connected to the set node $x_{N+k}$ to $u_{r+1}$, it is analogous to selecting the set $\Ss_{k}$ in the weighted set cover problem. The zero cost feedback edges take into account the set nodes $x_{N+j}$ for which the feedback edge going from output of  $x_{N+j}$ to $u_{r+1}$ is not selected. Given a solution $\bK$ to Problem~\ref{prob:one}, the sets selected under $\bK$ is defined as $\Ss(\bK)$. Here $\Ss(\bK)$ consists of all those sets whose corresponding set node has its dedicated output connected to the input $u_{r+1}$ in $\bK$ (Step~\ref{step:S}). Further, the weight $w(\Ss(\bK))$ is defined as shown in Step~\ref{step:w}. An illustrative example demonstrating the construction given in Algorithm~\ref{algo:one} is given in Figure~\ref{fig:sys1}.
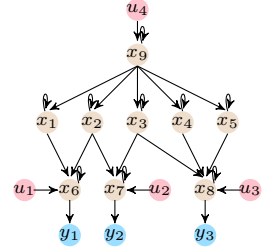
\begin{figure}[t]
\centering
\begin{minipage}{.22\textwidth}
\begin{tikzpicture}[scale=0.60, ->,>=stealth',shorten >=1pt,auto,node distance=1.65cm, main node/.style={circle,draw,font=\scriptsize\bfseries}]
\definecolor{myblue}{RGB}{80,80,160}
\definecolor{almond}{rgb}{0.94, 0.87, 0.8}
\definecolor{bubblegum}{rgb}{0.99, 0.76, 0.8}
\definecolor{columbiablue}{rgb}{0.61, 0.87, 1.0}

  \fill[almond] (-1,-1.5) circle (7.0 pt);
  \fill[almond] (-2,-1.5) circle (7.0 pt);
  \fill[almond] (0,0) circle (7.0 pt);
  \fill[almond] (0,-1.5) circle (7.0 pt);
  \fill[almond] (1,-1.5) circle (7.0 pt);
  \fill[almond] (2,-1.5) circle (7.0 pt);
  \node at (-2,-1.5) {\scriptsize $x_1$};
  \node at (-1,-1.5) {\scriptsize $x_2$};
  \node at (0,0) {\scriptsize $x_9$};
  \node at (0,-1.5) {\scriptsize $x_3$};
  \node at (1,-1.5) {\scriptsize $x_4$};
  \node at (2,-1.5) {\scriptsize $x_5$};

  \fill[bubblegum] (0,1) circle (7.0 pt);
  \fill[bubblegum] (-2.5,-3) circle (7.0 pt);
  \fill[bubblegum] (0.5,-3) circle (7.0 pt);
  \fill[bubblegum] (2.5,-3) circle (7.0 pt);
   \node at (0,1) {\scriptsize $u_4$};
   \node at (-2.5,-3) {\scriptsize $u_1$};
   \node at (0.5,-3) {\scriptsize $u_2$};
   \node at (2.5,-3) {\scriptsize $u_3$};
   
  \fill[almond] (-1.5,-3) circle (7.0 pt);
  \fill[almond] (-0.5,-3) circle (7.0 pt);
  \fill[almond] (1.5,-3) circle (7.0 pt);
  \fill[columbiablue] (-1.5,-4) circle (7.0 pt);
  \fill[columbiablue] (-0.5,-4) circle (7.0 pt);
  \fill[columbiablue] (1.5,-4) circle (7.0 pt);
   
   \node at (-1.5,-3.0) {\scriptsize $x_6$};
   \node at (-0.5,-3.0) {\scriptsize $x_7$};
   \node at (1.5,-3.0) {\scriptsize $x_8$};
   
   \node at (-1.5,-4.0) {\scriptsize $y_1$};
   \node at (-0.5,-4.0) {\scriptsize $y_2$};
   \node at (1.5,-4.0) {\scriptsize $y_3$};
   
  \draw (-2.3,-3)  ->   (-1.7,-3);
  \draw (0.3,-3)  ->   (-0.3,-3);
  \draw (2.3,-3)  ->   (1.7,-3);   
  

  \draw (-1.5,-3.22)  ->   (-1.5,-3.8);
  \draw (-0.5,-3.22)  ->   (-0.5,-3.8);
  \draw (1.5,-3.22)  ->   (1.5,-3.8);

  \draw (0,0.75)  ->   (0,0.25);
  \draw (0,-0.25)  ->   (-2,-1.25);
  \draw (0,-0.25)  ->   (-1,-1.25);
  \draw (0,-0.25)  ->   (0,-1.25);
  \draw (0,-0.25)  ->   (1,-1.25);
  \draw (0,-0.25)  ->   (2,-1.25);

  \draw (-2,-1.75)  ->   (-1.5,-2.75);
  \draw (-1,-1.75)  ->   (-1.5,-2.75);
  \draw (-1,-1.75)  ->   (-0.5,-2.75);
  \draw (0,-1.75)  ->   (-0.5,-2.75);
  \draw (1,-1.75)  ->   (1.5,-2.75);
  \draw (2,-1.75)  ->   (1.5,-2.75);
  \draw (0,-1.75)  ->   (1.5,-2.75);

\path[every node/.style={font=\sffamily\small}]
%

(-1.3,-2.85) edge[loop above] (-1.3,-2.85)
(-0.3,-2.85) edge[loop above] (-0.3,-2.85)
(1.7,-2.85) edge[loop above] (1.7,-2.85)
 
(-2.05,-1.25) edge[loop above] (-2,-1)
(-1.05,-1.25)edge[loop above] (-1,-1)
(0.1,0.25) edge[loop above]  (0,0)
(0.1,-1.25) edge[loop above]  (0,-1)
(1.0,-1.25) edge[loop above]  (1,-1)
(2.0,-1.25) edge[loop above]  (2,-1);
\end{tikzpicture}
\end{minipage}\hspace{0.1 cm}
\begin{minipage}{.2\textwidth}
\caption{\small Digraph $\D(\bA, \bB, \bC)$ constructed using Algorithm~\ref{algo:one} for a weighted set cover problem with $\U = \{1,\ldots,5\}$, $\P= \{\Ss_1, \Ss_2, \Ss_3\}$, where $\Ss_1 = \{1,2 \}$, $\Ss_2 = \{2,3\}$ and $\Ss_3 = \{3,4,5\}$.}
\label{fig:sys1}
\end{minipage}
\vspace*{-5 mm}
\end{figure} 

\begin{lem}\label{lem:one}
Consider the weighted set cover problem with $\pazocal{U}=\{1,\dots,N\}$, sets  $\pazocal{P}=\{\Ss_1,\dots,\Ss_r\}$ and weight function $w$. Let $(\bA, \bB={\mathbb{I}_m}, \bC={\mathbb{I}_p})$ and $P$ be the structured system and feedback cost matrix constructed using Algorithm~\ref{algo:one}. If $\bar{K}$ is a solution to Problem~\ref{prob:one}, then $\Ss(\bar{K})$ covers $\pazocal{U}$.
\end{lem}
\begin{proof}
Here we assume that a $\bK$ is a solution to Problem~\ref{prob:one} and then show that $\pazocal{S}(\bK)$ is a solution to the weighted set cover problem. Consider an arbitrary element $j \in \pazocal{U}$. We show that $\Ss(\bar{K})$ covers the element $j$. Consider node $x_j$. Since $\bar{K}$ is a solution to Problem~\ref{prob:one}, it follows that $x_j$ must lie in an SCC with atleast one feedback edge in it. Notice that node $x_j$ does not have an input or output connected directly to it. Thus the only way for node $x_j$ to satisfy condition~(a) in Proposition~\ref{prop:one} is using a feedback edge connecting the output of some set node $x_{k}$, where $k \in \{N+1,\dots,N+r\}$, to the input node $u_{r+1}$ such that $(x_j,x_k)\in E_X$, i.e., $j\in \Ss_{k-N}$. Using Step~\ref{step:S} of Algorithm~\ref{algo:one}, this implies that set $\Ss_{k-N} \in \pazocal{S}(\bK)$. Thus element $j$ is covered by $\pazocal{S}(\bK)$. Since element $j$ is arbitrary, the proof follows.
\end{proof}
\begin{theorem}\label{theorem:one}
Consider a structured system $(\bA,\bB={\mathbb{I}_m},\bC={\mathbb{I}_p})$ and a feedback cost matrix $P$. Then, Problem~\ref{prob:one} is NP-hard.
\end{theorem}
\begin{proof}
The reduction of the weighted set cover problem given in Algorithm~\ref{algo:one} is used for proving the NP-hardness. Let $\bK$ be a solution to Problem~\ref{prob:one}. Now we show that $\Ss(\bK^\*)$ is an optimal solution to the weighted set cover problem, where $\bK^\*$ is an optimal solution to Problem~\ref{prob:one}. By Lemma~1, $\pazocal{S}(\bK^\*)$ is a solution to the weighted set cover problem. Hence feasibility holds.
To prove optimality, assume that $\bK^\*$ denotes an optimal solution to Problem~\ref{prob:one}. The proof follows if $\Ss(\bK^\*)$ is an optimal solution to the weighted set cover problem. We prove this using a contradiction argument. Let the set $\Ss'$ be a cover to the weighted set cover problem, i.e., $\cup_{\Ss_i\in \Ss'} \Ss_i = \U$, such that $w(\Ss') < w(\Ss(\bK^\*))$. Corresponding to the set $\Ss'$, we construct $\bK' \in \{0, \* \}^{(r+1) \times r}$ as follows:

{\scalefont{0.9}{
$\bK'_{ij} = \begin{cases} \star$, for $i=r+1$ and $j:\Ss_j\in \Ss',\\
\star$, for $i=j,\\ 0,$ otherwise.$ 
\end{cases}$
}}

Notice that the cost $P(\bK')=w(\Ss')$ because the feedback edges selected in $\bK'$ of the form $(y_k,u_{r+1})$ have cost $w(\Ss_{k})$ and other feedback edges of the form $(y_k,u_k)$ have zero cost. 
To show that $\bK'\in \K$, for an arbitrary node $x_j$ consider the following three cases: 1)~$j \in\{1,\dots,N\}$, 2)~$j \in \{N+1,\dots,N+r\}$, and 3)~$j = N+r+1$. 

For case~1), consider node $x_j$. Since $\Ss'$ is a solution to the weighted set cover problem, there exists a set $\Ss_{k} \in \Ss'$ such that $j \in \Ss_k$. Corresponding to the set $\Ss_k$, $\bK'_{(r+1)k}=\star$. Hence, $x_j$ lies in an SCC with the feedback edge $(y_k,u_{r+1})$. For case~2), notice that $\bK'_{ii}=\star$ for all $i$. Hence, $x_{N+k}$, for $k=1,\dots,r,$ lies in an SCC with the zero cost feedback edge $(y_k,u_k)$. For case~3), since we have shown that element nodes are part of SCC with feedback edges connected to $u_{r+1}$ which is connected to node $x_{N+r+1}$, node $x_{N+r+1}$ also belongs to an SCC with a feedback edge. Thus all nodes lie in an SCC with a feedback edge and condition~(a) in Proposition~\ref{prop:one} is satisfied. Since $\B(\bA)$ has a perfect matching, condition~(b) in Proposition~\ref{prop:one} is satisfied. Hence $\bK'\in \K$.  By Steps~\ref{step:S} and~\ref{step:w} of Algorithm~\ref{algo:one}, $P(\bK^\*)=w(\Ss(\bK^\*))$. Further, we know that $P(\bK')=w(\Ss')$ and by assumption $w(\Ss') < w(\Ss(\bK^\*))$. Thus $P(\bK') < P(\bK^\*)$, which is a contradiction to the optimality of $\bK^\*$. As a result, given an optimal solution $\bK^\*$, an optimal solution to the weighted set cover problem $\Ss(\bK^\*)$ can be obtained. Hence, Problem~\ref{prob:one} is NP-hard.  
\end{proof}

\begin{rem} \label{rem:one}
Problem~1 is NP-hard even when the cost of the feedback edges are restricted to 1, 0, and $\infty$. For this case, one can reduce the minimum set cover problem to an instance of Problem~\ref{prob:one} in polynomial time using Algorithm~\ref{algo:one}. In this reduction all the feedback edges from $\{y_1,\ldots,y_r\}$ to $u_{r+1}$ are of uniform cost.
\end{rem}

Notice that in the reduction given in Algorithm~1, $\bA$ has all diagonal entries as $\*$'s. Hence $\B(\bA)$ has a perfect matching. Thus, even without using any feedback edges, condition~(b) is satisfied and hence the optimization in Problem~\ref{prob:one} is now to satisfy condition~(a) optimally. The following result holds.
\begin{cor}
Consider the structured system $(\bA,\bB={\mathbb{I}_m},\bC={\mathbb{I}_p})$ and feedback cost matrix $P$. Then, finding a minimum cost feedback matrix that satisfies condition~(a) in Proposition~1 is NP-hard.
\end{cor}

By Theorem~1, Problem 1 is atleast as hard as the weighted set cover problem. Hence there does not exist a polynomial time algorithm to solve Problem~\ref{prob:one}, unless P=NP. However, approximation algorithms may exist. Before investigating this, the inapproximability of Problem~\ref{prob:one} is analyzed in Theorem~\ref{theorem:two}.

\begin{theorem} \label{theorem:two}
Consider a general instance of the weighted set cover problem and a structured system $(\bA,\bB,\bC)$ and feedback cost matrix $P$ constructed using Algorithm~\ref{algo:one}. Let $\Ss^\*$ and $\bK^\*$ be optimal solutions to the weighted set cover problem and Problem~1, respectively. For $\epsilon \geqslant 1$, if $\bK'$ is an $\epsilon$-optimal solution to Problem~\ref{prob:one}, then $\Ss(\bK')$ is an $\epsilon$-optimal solution to the weighted set cover problem, i.e., $P(\bK')\leqslant \epsilon\, P(\bK^\*)$ implies $w(\Ss(\bK'))\leqslant \epsilon\, w(\Ss^\*)$. Moreover, Problem~\ref{prob:one} is inapproximable to a multiplicative factor of ${\rm log\,}n$, where $n$ denotes the number of state nodes.
\end{theorem}
\begin{proof} Given $\bK'$ is an $\epsilon$-optimal solution to problem~1, i.e.,  $P(\bK') \leqslant \epsilon\, P(\bK^\*)$. From Steps~\ref{step:S} and~\ref{step:w} of Algorithm~1, we have $w(\Ss(\bK'))=P(\bK')$ and $w(\Ss(\bK^\*))=P(\bK^\*)$. Also, by Theorem~\ref{theorem:one}, $\Ss(\bK^\*)$ is an optimal solution to the weighted set cover problem. Therefore, $w(\Ss^\*)=w(\Ss(\bK^\*))= P(\bK^\*)$. Hence $w(\Ss(\bK')) \leqslant \epsilon\, w(\Ss^\*)$. Thus an $\epsilon$-optimal solution to Problem~\ref{prob:one} gives an $\epsilon$-optimal solution to the weighted set cover problem. The weighted set cover problem is inapproximable to a factor of $(1-o(1))\,\mbox{log\,}{N}$ \cite{Fei:98}, where $N$ denotes the cardinality of the universe. Thus Problem~\ref{prob:one} is inapproximable to a multiplicative factor of $\mbox{log\,}{n}$.
\end{proof}
In the following Sections (Section~\ref{sec:Reform} and \ref{sec:Approx}), we explore approximation algorithm to solve Problem~\ref{prob:one}. Later in Section~\ref{sec:spcases}, we consider Problem~\ref{prob:one} on two special graph topologies, which are of practical importance, and propose polynomial time algorithms to obtain a solution.
\section{Reformulating Optimal Feedback Selection Problem to Optimal Cycle Selection Problem}\label{sec:Reform}
In this section, we reformulate Problem~\ref{prob:one} to a graph theoretic equivalent. The following assumption holds.
\begin{assume}\label{assume:one}
The structured system $(\bA,\bB={\mathbb{I}_m},\bC={\mathbb{I}_p})$ satisfies the following condition: $\B(\bar{A})$ has a perfect matching.
\end{assume}
The motivation to make this assumption comes from the fact that there exists a wide class of systems called as {\em self-damped} systems that have a perfect matching in $\B(\bA)$, for example consensus dynamics in multi-agent systems and epidemic equations \cite{ChaMes:13}.
Self-damped systems are the systems with all diagonal entries of $\bA$ as nonzero. All systems with non-singular state matrix also satisfy Assumption~\ref{assume:one}. Consider a structured system $(\bA,\bB={\mathbb{I}_m},\bC={\mathbb{I}_p})$ that satisfies Assumption~\ref{assume:one} and a cost matrix $P$. Recall that under Assumption~1, condition~(b) in Proposition~\ref{prop:one} is satisfied without using any feedback edge. Hence, for solving Problem~\ref{prob:one} we need to satisfy only condition~(a) in Proposition~\ref{prop:one}.
\begin{algorithm}
\caption{Pseudo-code reducing Problem~\ref{prob:one} to a cycle formulation}\label{algo:two}
\textbf{Input}: Structured system $(\bA,\bB={\mathbb{I}_m},\bC={\mathbb{I}_p})$ and feedback cost matrix $P$\\
\textbf{Output}: Cycles $\C = \{\C_1,\dots,\C_t\}$ of digraph $\D_R$
\begin{algorithmic}[1]
\State Construct $\D(\bA)$ and find SCCs in $\D(\bA)$, say $\mN = \{\mN_{1},\dots, \mN_{\ell}\}$\label{step:digraph}
\State Condense each SCC into a single node, say node set $\mN = \{\mN_{1},\dots, \mN_{\ell}\}$\label{step:condense}
\State Define $E_{\mN}:= \{(\mN_a,\mN_b)$ : $x_i \in \mN_a$ , $x_j\in \mN_b$ and $(x_i,x_j) \in E_{X}\}$\label{step:En}
\State Define $E'_U:= \{(u_j,\mN_k)$ : $x_i \in \mN_k$ and $(u_j,x_i) \in E_U\}$\label{step:E'u}
\State Define $E'_Y:= \{(\mN_k,y_j)$ : $x_i \in \mN_k$ and $(x_i,y_j) \in E_Y\}$\label{step:E'y}
\State Construct $\D_F\leftarrow(\mN \cup V_U \cup V_Y, E_{\mN} \cup E'_U \cup E'_Y \cup E_K)$\label{step:Df}
\State $E_{ab} \leftarrow$ $\{(y_i,u_j)$ : $(u_j,\mN_a) \in E'_U$ and $(\mN_b,y_i) \in E'_Y\}$\label{step:E_ab}
\State $e_{ab} \leftarrow$ $\{(y_{i'},u_{j'})$ : $(i',j') \in \arg\min_{(y_i,u_j) \in E_{ab}} P_{ji}\}$\label{step:e_ab}
\State $E_{\rm min}\leftarrow\{ {e_{ab}:a,b \in \{1,\dots,\ell\}}\}$\label{step:E_min}
\State Construct $\D_R\leftarrow(\mN \cup V_U \cup V_Y, E_\mN \cup E'_U \cup E'_Y \cup E_{\rm min})$\label{step:Dr}
\State Find all the cycles in $\D_R$, $\C = \{\C_1,\dots,\C_t\}$\label{step:cycle}
\State Each cycle $\C_i\in \C$ has the following structure: $\C_i\leftarrow (\{N_i \subseteq \mN\}:[E_i \subseteq E_{\rm min}])$\label{step:cycle_struc}
\end{algorithmic}
\end{algorithm}
The approximation algorithm given in this paper is based on {\em cycle} formulation of Problem~\ref{prob:one}. Given $(\bA,\bar{B}=\mathbb{I}_m, \bar{C}=\mathbb{I}_p)$ and cost matrix $P$, the pseudo-code showing reformulation of Problem~\ref{prob:one} to a cycle based problem is presented in Algorithm~\ref{algo:two}. Algorithm~\ref{algo:two} constructs digraph $\D_F$ and the reduced digraph $\D_R$ as defined below and gives as output the cycles in digraph $\D_R$. The cycles in a directed graph can be found using the algorithm in \cite{Joh:75}.

Consider the directed graph $\D(\bA)$. We first find the set of all SCCs, $\mN = \{\mN_1,\dots,\mN_{\ell}\}$, in $\D(\bA)$ (Step~\ref{step:digraph}). Each SCC is now condensed to a node. With a slight abuse of notation, $\mN=\{\mN_1,\dots,\mN_{\ell}\}$ is used to denote the set of condensed nodes (Step~\ref{step:condense}). The construction of the digraph $\D_F=(\mN \cup V_U \cup V_Y, E_\mN\cup E'_U \cup E'_Y \cup E_K)$ is as follows. In $\D_F$,  an edge $(\mN_a,\mN_b) \in E_\mN$ if there exists an $x_i \in \mN_a$ and $x_j \in \mN_b$ and $\bA_{ji} = \star$ (Step~\ref{step:En}). Given the input edge set $E_U$, the edge set $E'_U$ is constructed in such a way that $(u_i,\mN_a) \in E'_U \Leftrightarrow x_j \in \mN_a$ and $(u_i,x_j) \in E_U$ (Step~\ref{step:E'u}). Similarly, the edge set $E'_Y$ is constructed such that $(\mN_a,y_i) \in E'_Y \Leftrightarrow x_j \in \mN_a$ and $(x_j,y_i) \in E_Y$ (Step~\ref{step:E'y}). Thus $E'_U$ consists of edges from an input to SCCs in $\mN$ and $E'_Y$ consists of edges from SCCs in $\mN$ to an output. Recall that $E_K$ is the set of all feedback edges for which $P_{ij}$ is finite. Thus, $E_K$ consists of all feasible feedback edges.
 
Next we construct the reduced edge set $E_{\rm min}$ and the directed graph $\D_R=(\mN \cup V_U \cup V_Y, E_\mN\cup E'_U \cup E'_Y \cup E_{\rm min})$ from $\D_F$. Corresponding to each SCC node in $\mN$, there are possibly multiple input and output nodes. Thus for an arbitrary node pair $\mN_a,\mN_b \in \mN$ there are numerous feedback edges possible between them. In such a situation, we only consider a least cost feedback edge between these nodes and ignore others. Corresponding to an arbitrary node pair $\mN_a, \mN_b \in \mN$, we define the set $E_{ab}$ as the set of all feasible feedback edges from $\mN_b$ to $\mN_a$ (Step~\ref{step:E_ab}). For all $\mN_a,\mN_b\in \mN$, if a feedback edge exists between $(\mN_b,\mN_a)$, select a minimum cost edge from edge set $E_{ab}$ and include it in edge set $E_{\rm min}$. This simplification results in a digraph $\D_R:=(\mN \cup V_U \cup V_Y, E_\mN \cup E'_U \cup E'_Y \cup E_{\rm min})$ (Step~\ref{step:Dr}).
Next, the directed cycle set $\C = \{\C_1,\dots,\C_t\} $ in $\D_R$ is obtained. A cycle consists of two sets: node set $N_i\subseteq \mN$ and feedback edge set $E_i \subseteq E_{\min}$. Also, the cost of an edge set $\hat{E} \subseteq E_{\min}$, denoted by $c(\hat{E})$ is the sum of the costs of the individual edges present in it, i.e., $c(\hat{E})=\sum_{e_i\in \hat{E}}c(e_i)$, where $c(e_i)$ denotes the cost of the feedback edge $e_i$ as defined by the feedback cost matrix $P$. Below we define Problem~\ref{prob:two}, which is an optimization problem on $\D_R$ and later show that this formulation indeed solves Problem~\ref{prob:one}. 

\begin{prob}[Optimal cycle selection problem]\label{prob:two}
Consider a structured system $(\bA, \bB=\mathbb{I}_m, \bC=\mathbb{I}_p)$ and feedback cost matrix $P$. Let $E_{\rm min}$ denotes the set of feedback edges constructed using Algorithm~\ref{algo:two}. Then, find $E_{opt} ~\in~ \arg\min\limits_{\ \hat{E} \subseteq E_{\rm min}} c(\hat{E}),$ such that each node, $\mN_i \in \mN$, lies in atleast one cycle in the digraph $\D_{opt} = (\mN\cup V_U\cup V_Y, E_{\mN}\cup E'_U\cup E'_Y\cup E_{opt})$.
\end{prob}
We show that Problem~\ref{prob:two} is equivalent to optimal feedback selection problem for dedicated i/o.
\begin{theorem}\label{theorem:three}
Consider a structured system $(\bA, \bB, \bC)$ and feedback cost matrix $P$. Let $\D_R$ be the digraph constructed using Algorithm~\ref{algo:two}. Then, $E'$ is a solution to Problem~\ref{prob:two} if and only if $\bK':=\{\bK'_{ij}= \star : (y_j,u_i)\in E'\}$ is a solution to Problem~\ref{prob:one}. Moreover, for $\epsilon \geqslant 1$, if $E'$ is an $\epsilon$-optimal solution to Problem~\ref{prob:two}, then $\bK'$ is an $\epsilon$-optimal solution to Problem~\ref{prob:one}, i.e., $c(E') \leqslant \epsilon\, c(E_{opt})$ implies $P(\bK') \leqslant \epsilon\,P(\bK^\*)$.
\end{theorem}

\begin{proof}
\textbf{Only-if part}: We assume that $E'$ is a solution to Problem~\ref{prob:two} and then show that $\bK'$ is a solution to Problem~\ref{prob:one}. Since $E'$ is a solution to Problem~\ref{prob:two}, each $\mN_i\in \mN$ lies in a cycle with some feedback edge, say $(y_b,u_a)\in E'$. Consider an arbitrary node $x_j\in \mN_i$. Since $x_j$ lies in the SCC $\mN_i$ and $\mN_i$ lies in a cycle with some feedback edge $(y_b,u_a)$, $x_j$ lies in an SCC in $\D(\bA, \bB, \bC, \bK')$ with feedback edge $(y_b,u_a)$. Since $x_j$ is arbitrary, all nodes lie in an SCC with a feedback edge. Hence $\bK'$ is a solution to Problem~\ref{prob:one}. 

\noindent\textbf{If-part}: We assume that $\bK'$ is a solution to Problem~\ref{prob:one} and show that $E':=\{(y_j,u_i)$: $\bK'_{ij}=\star\}$ is a solution to Problem~\ref{prob:two}. Let $x_j\in \mN_i$ be an arbitrary state in SCC $\mN_i$ of $\D(\bA)$. Since $\bK'$ is a solution to Problem~\ref{prob:one}, $x_j$ lies in an SCC in $\D(\bA,\bB,\bC,\bK')$with some feedback edge, say $(y_b,u_a)$. Hence there exists a directed path $L$ in $\D(\bA,\bB,\bC,\bK')$ from $x_j$ to itself, with node repetitions allowed, which includes feedback edge $(y_b,u_a)$. Let the set of state nodes that lie in this path be denoted by $N_{\scriptscriptstyle L}$. Consider the digraph $\D'=(\mN\cup V_U\cup V_Y, E_{\mN}\cup E'_U\cup E'_Y)$. If $N_{\scriptscriptstyle L}\subseteq \mN_i$, then $\mN_i$ lies in a cycle with feedback edge $(y_b,u_a)$. If $N_{\scriptscriptstyle L}\nsubseteq \mN_i$, then since all the state nodes in $L$ lie in some SCC in $\D(\bA,\bB,\bC,\bK')$ there exists a path that originates at SCC $\mN_i$ and returns to $\mN_i$ including the feedback edge $(y_b,u_a)$. For this path, the SCC node repetitions are not allowed because $\D'$ is a DAG. Thus the directed path $L$ along with feedback edge $(y_b,u_a)$ forms a cycle in $\D'$ and hence $\mN_i$ lies in the cycle formed by path $L$ which includes the feedback edge $(y_b,u_a)$. This concludes the if-part of the proof.   

 Next, we show the $\epsilon$-optimality. Given $E'$ is a solution to Problem~\ref{prob:two} and  $c(E') \leqslant \epsilon\, c(E_{opt})$. By only-if part of Theorem~\ref{theorem:three}, $\bK':=\{\bK'_{ij}= \star : (y_j,u_i)\in E'\}$ is a feasible solution to Problem~\ref{prob:one}. Also, by definition of $\bK'$, $c(E')=P(\bK')$. Similarly, $\bK^{opt}:=\{\bK^{opt}_{ij}=\*$: $(y_j,u_i)\in E_{opt}\}$ is a feasible solution to Problem~\ref{prob:one} and $P(\bK^{opt})=c(E_{opt})$. Thus, $P(\bK') \leqslant \epsilon\,P(\bK^{opt})$.  Now we show that $\bK^{opt}$ is an optimal solution to Problem~\ref{prob:one}. Suppose not, i.e., $P(\bK^\*) <P(\bK^{opt})$. Then, by if-part of Theorem~\ref{theorem:three}, $E^\*:=\{(y_j,u_i) \in E^\*$: $\bK^\*_{ij}=\*\}$ is a feasible solution to Problem~\ref{prob:two}. Also, $c(E^\*) = P(\bK^\*)$. Thus $c(E^\*) < c(E_{opt})$. This contradicts the optimality of $E_{opt}$. Hence $P(\bK^{opt}) = P(\bK^\*)$. Now, since $P(\bK') \leqslant \epsilon\, P(\bK^{opt})$ and $P(\bK^{opt})=P(\bK^\*)$, $P(\bK') \leqslant \epsilon\,P(\bK^\*)$. This completes the proof.
\end{proof}
Theorem~\ref{theorem:three} thus concludes that an $\epsilon$-optimal solution to Problem~\ref{prob:two} gives an $\epsilon$-optimal solution to Problem~\ref{prob:one}. We elaborate our approach to solve Problem~\ref{prob:two} below.
\section{Approximation Algorithm for the Optimal Feedback Selection Problem}\label{sec:Approx}
This section discusses a {\em greedy algorithm} and later an {\em approximation algorithm} to find an approximate solution to Problem~\ref{prob:two}, which in turn gives an approximate solution to Problem~\ref{prob:one} (Theorem~\ref{theorem:three}).
Recall $\C$ as the set of cycles in $\D_R$.

\begin{defn}\label{defn:two}
Consider the set of cycles in $\D_R$, $\C=\{\C_1,\ldots,\C_t\}$. Given a set of cycles $\C' \subseteq \C$ in $\D_R$, the node set $N'$ \underline{covered} by $\C'$ is defined as $ N' := \cup_{\C_i \in \C'} N_i$, where $\C_i =  (\{N_i\} : [E_i])$. Here $N' \subseteq \mN$, where $\mN$ is the set of SCCs in $\D(\bA)$. In other words, we say $\C'$ \underline {covers} $N'$.
Further, the cost of the cover of cycle set $\C'$ is defined as $c(\cup_{\C_i \in \C'} E_i)$. Also, $\C'$ is said to be an optimal cycle cover if $N'=\mN$ and the cost of the cover $\C'$ is equal to $c(E_{opt})$, where $E_{opt}$ is an optimal solution to Problem~\ref{prob:two}.
\end{defn}
Our approach to solve Problem~\ref{prob:two} incorporates a greedy algorithm presented in Algorithm~\ref{algo:three} with a {\em potential function} presented in Algorithm~\ref{algo:four}. Algorithm~\ref{algo:three} is described below.   
\begin{algorithm}
\caption{Pseudo-code for subroutine \textsc{Greedy}$(\cdot,\cdot)$}\label{algo:three}
\textbf{Input}: Cycle set $\C_{\rm inp} \subseteq \C$ in $\D_R$, where $\C_i\in \C_{\rm inp} :=(\{N_i\} : [E_i])$, and an edge set $E_{\rm inp}\subseteq E_K$\\
\textbf{Output}: Set of feedback edges $H$
\begin{algorithmic}[1]
\State \textsc{Greedy}$(\cup_{\C_i \in \C_{\rm inp}} N_i,E_{\rm inp})$:\label{step:Greedy}
\State Initialize the set of covered nodes, $I\leftarrow emptyset$\label{step:I}
\State Initialize the set of selected edges, $H\leftarrow \emptyset$  
\State $N_{\rm inp}\leftarrow \cup_{\C_i \in \C_{\rm inp}} N_i$\label{step:Ninp}
\State $E_{i}\leftarrow E_i \setminus E_{\rm inp}$, for all $\C_i \in \C_{\rm inp}$\label{step:Einp}
\While{$I\neq N_{\rm inp}$}\label{step:While}
\State Calculate $\rho(\C_k) \leftarrow c(E_k)/|N_k|$, for all $\C_k \in \C_{\rm inp}$\label{step:p(C)}
\State Select $\C_j\in \arg\min_{\C_i\in \C_{\rm inp}} \rho(\C_i)$\label{step:C_j}
\State Update $I\leftarrow I\cup N_j$, $H\leftarrow H\cup E_j$\label{step:I_update}
\State $N_k\leftarrow N_k\setminus I$, $E_k\leftarrow E_k\setminus H$, for all $\C_k \in \C_{\rm inp}$\label{step:N_update}
 \For{$\C_k \in \C_{\rm inp}$}
   \If {$N_k=\{ \}$}  
\State $\C_{\rm inp}\leftarrow \C_{\rm inp}\setminus \C_k$\label{step:C_update}
   \EndIf
 \EndFor
\EndWhile
\Return $H$
\end{algorithmic}
\end{algorithm}
The pseudo-code to find a greedy solution to Problem~\ref{prob:two} is presented in Algorithm~\ref{algo:three}. Consider a structured system $(\bA,\bB,\bC)$ and feedback cost matrix $P$. Let $\D_R$ denotes the digraph corresponding to the structured system constructed using Algorithm~\ref{algo:two}. Given a set of cycles $\C_{\rm inp}$ and an edge set $E_{\rm inp}$ as input, Algorithm~\ref{algo:three} outputs a set of feedback edges $H$ such that $H \subseteq
 E_{\min}$, $H \cap E_{\rm inp} = \emptyset$ and all nodes $\mN_i\in N_{\rm inp}$, where $N_{\rm inp}=\cup_{\C_i \in \C_{\rm inp}} N_i$ (Step~\ref{step:Ninp}), lie in atleast one cycle in the digraph $({\mN}\cup V_U\cup V_Y,E_{\mN}\cup E'_U\cup E'_Y\cup H)$. At each step of the \textbf{while} loop (Step~\ref{step:While}), the sets $I$ and $H$ are defined as the set of nodes covered and the set of feedback edges selected, respectively (Steps~\ref{step:I} and~\ref{step:Ninp}). Our purpose is to make $I=N_{\rm inp}$. In other words, given a set of cycles $\C_{\rm inp}$ in $\D_R$, our aim is to choose a set of cycles $\C_{sol} \subseteq \C_{\rm inp}$ such that $\C_{sol}$ is a cover (Definition~\ref{defn:two}) of $N_{\rm inp}$. For each cycle $\C_i\in \C_{\rm inp}$, we define price of a cycle as the average cost per node, i.e., $\rho(\C_i)=c(E_i)/|N_i|$ (Step~\ref{step:p(C)}). A cycle which has a minimum price, say $\C_j$, is selected (Step~\ref{step:C_j}). We call this selection as a greedy selection of the cycle $\C_j$. If there are multiple cycles with minimum price, select any one of them. Based on this selection, the sets $I$ and $H$ are updated by including the nodes and the edges of $\C_j$, respectively (Step~\ref{step:I_update}). Further, all the covered nodes $(I)$ and all the selected edges $(H)$ are removed from the node set and the edge set of each cycle, respectively  (Step~\ref{step:N_update}). The set of cycles $\C_{\rm inp}$ is now updated by removing all the cycles with empty node set (Step~\ref{step:C_update}).  These set of operations are performed until we cover all the nodes in $N_{\rm inp}$, i.e., $I=N_{\rm inp}$. The cost of this greedy approach is denoted by $c(H)$, where $H$ is the set of feedback edges selected by the greedy algorithm satisfying $H \cap E_{\rm inp} = \emptyset$.
Let $\mC_{\rm arb}$ be an arbitrary set of cycles. Then, for each edge $e_i \in E_{\min}$, we define multiplicity $m_i(\mC_{\rm arb})$ as $m_i(\mC_{\rm arb}) = |\{\C_j:\C_j \in \mC_{\rm arb} {~\rm and ~} e_i \in E_j \}|$. In other words, $m_i(\mC_{\rm arb})$ is the number of cycles in $\mC_{\rm arb}$ in which the feedback edge $e_i$ is present. Now we define $k_1(\mC_{\rm arb}) := \max_{e_i\in E_{\min}} m_i(\mC_{\rm arb})$  and is referred as the {\em first highest multiplicity} of an edge in cycle set $\mC_{\rm arb}$. Also, for every cycle $\C_j \in \mC_{\rm arb}$, $k^j_1 (\mC_{\rm arb}) :=  \max_{e_i \in E_{\min}\setminus E_j}m_i(\mC_{\rm arb})$. Then, $k_2(\mC_{\rm arb}) := \min_{\C_j\in \mC_{\rm arb}} k^j_1$ and is referred as the {\em second highest multiplicity} of an edge in cycle set $\mC_{\rm arb}$. Next, let $\bfC$ denotes the set that consists of all possible optimal solutions to Problem~\ref{prob:two}. Note that, $\mC(j) \in \bfC$ is a set of cycles in $\D_R$. Then, we define $\tilde{k}_1 = \min_{\mC(j)\in \bfC} k_1(\mC(j))$ and a corresponding cycle set $\mC^1_{opt} \in \arg\min_{\mC(j)\in \bfC} k_1(\mC(j))$. Similarly, $\tilde{k}_2 = \min_{\mC(j)\in \bfC} k_2(\mC(j))$ and a corresponding cycle set $\mC^2_{opt} \in \arg\min_{\mC(j)\in \bfC} k_2(\mC(j))$. Further, $E^1_{opt}$ and $E^2_{opt}$ denote the set of feedback edges present in set of cycles $\mC^1_{opt}$ and $\mC^2_{opt}$, respectively. Note that, $\tilde{k}_1$ and $\tilde{k}_2$ may not necessarily be from the same cycle set in $\bfC$. Also, since $\mC^1_{opt} \in \bfC$ and $\mC^2_{opt} \in \bfC$,  $c(E^1_{opt}) = c(E^2_{opt}) = c(E_{opt})$.

We describe an example using Figure~\ref{fig:k1tilda} to demonstrate the values of variables $k_1$, $k_2$ for a cycle set and $\tilde{k}_1$, $\tilde{k}_2$ for the structured system illustrated.
Consider the following cycles: 
{\scalefont{0.9}{
\begin{eqnarray}\label{eqn:cycles}
\C_1&:&(\{\mN_1,\mN_2,\mN_3\}:[(y_3,u_2),(y_2,u_1)])\nonumber\\ 
\C_2&:&(\{\mN_1,\mN_2,\mN_4\}:[(y_4,u_2),(y_2,u_1)])\nonumber\\
\C_3&:&(\{\mN_1,\mN_2,\mN_5\}:[(y_1,u_5),(y_2,u_1)])\nonumber\\
\C_4&:&(\{\mN_5,\mN_6,\mN_8\}:[(y_5,u_6),(y_6,u_8)])\nonumber\\
\C_5&:&(\{\mN_5,\mN_6,\mN_7\}:[(y_5,u_6),(y_7,u_5)])\nonumber\\
\C_6&:&(\{\mN_3\}:[(y_3,u_3)])\nonumber\\
\C_7&:&(\{\mN_6\}:[(y_6,u_6)])\nonumber\\
\C_8&:&(\{\mN_7\}:[(y_7,u_7)])\nonumber\\
\C_9&:&(\{\mN_8\}:[(y_8,u_8)]).
\end{eqnarray}
}}
Let the feedback cost matrix $P$ associated with the structured system given in Figure~\ref{fig:k1tilda} be
{\scalefont{0.9}{
$$ P=\left[
\begin{smallmatrix}
10& 1  & 10 & 10 & 10 & 10 & 10 & 10 \\
10& 10 & 1 & 1  & 10 & 10 & 10 & 10 \\
10& 10 & 1  & 10 & 10 & 10 & 10 & 10 \\
10& 10 & 10 & 10 & 10 & 10 & 10 & 10 \\
1&  10 & 10 & 10 & 10 & 10 & 1 & 10 \\
10& 10 & 10 & 10 & 1  & 1  & 10 & 10\\
10& 10 & 10 & 10 & 10 & 10 & 1  & 10\\
10& 10 & 10 & 10 & 10 & 1  & 10 & 1  
\end{smallmatrix}
\right]
$$
}}
For the structured system given in Figure~\ref{fig:k1tilda}, the set of all possible optimal solutions to Problem~\ref{prob:two} $\bfC=\{\mC(1),\mC(2),\mC(3),\mC(4)\}$. Here, $\mC(1)=\{\C_1,\C_2,\C_3,\C_7,\C_8,\C_9\}$, $\mC(2)=\{\C_2,\C_3,\C_4,\C_5,\C_6\}$, $\mC(3)=\{\C_1,\C_2,\C_3,\C_4,\C_5\}$ and  $\mC(4)=\{\C_2,\C_3,\C_6,\C_7,\C_8,\C_9\}$. In the cycle set $\mC(1)$, the feedback edge $(y_2,u_1)$ is present in 3 cycles, which is the first highest multiplicity of an edge in cycle set $\mC(1)$. The second highest multiplicity of an edge in $\mC(1)$ is 1 because all the other feedback edges are present in only one cycle. Hence $k_1(\mC(1))=3$ and $k_2(\mC(1))=1$. In $\mC(2)$, the feedback edges $(y_2,u_1)$ and $(y_5,u_6)$ are both present in 2 cycles each.  Therefore, the first highest multiplicity of an edge in $\mC(2)$ is 2. Also, the second highest multiplicity of an edge in $\mC(2)$ is also 2. In $\mC(3)$, the edge $(y_2,u_1)$ is present in 3 cycles, which is the first highest multiplicity of an edge in cycle set $\mC(3)$, and the edge $(y_5,u_6)$ is present in 2 cycles.  Therefore, the first highest multiplicity of an edge in $\mC(3)$ is 3 and the second highest multiplicity of an edge in $\mC(3)$ is 2. In $\mC(4)$, the feedback edges $(y_2,u_1)$ is present in 2 cycles and all the other feedback edges are present in one cycle.  Thus, the first highest multiplicity of an edge in $\mC(4)$ is 2 and the second highest multiplicity of an edge in $\mC(4)$ is 1. Therefore, $\tilde{k}_1 = \min\{k_1(\mC(1)),k_1(\mC(2)),k_1(\mC(3)),k_1(\mC(4))\}$ = $\min\{3,2,3,2\}=2$ and $\tilde{k}_2 =\min\{k_2(\mC(1)),k_2(\mC(2))$, $k_2(\mC(3)), k_2(\mC(4))\}$ = $\min\{1,2,2,1\}=1$. Note that, only the feedback edges which lie in the cycles present in the sets $\mC(1),\mC(2),\mC(3)$ and $\mC(4)$ are shown in Figure~\ref{fig:k1tilda}.
\begin{figure}[t]
\begin{minipage}{.25\textwidth}
\begin{tikzpicture}[scale=0.42, ->,>=stealth',shorten >=1pt,auto,node distance=1.65cm, main node/.style={circle,draw,font=\scriptsize\bfseries}]
\definecolor{myblue}{RGB}{80,80,160}
\definecolor{almond}{rgb}{0.94, 0.87, 0.8}
\definecolor{bubblegum}{rgb}{0.99, 0.76, 0.8}
\definecolor{columbiablue}{rgb}{0.61, 0.87, 1.0}


	\fill[almond] (0,0) circle (12.0 pt);
    \node at (0,0) {\scriptsize $\mN_1$};
	\fill[almond] (0,-6) circle (12.0 pt);  
    \node at (0,-6) {\scriptsize $\mN_2$};
    \fill[almond] (1.5,-3) circle (12.0 pt);  
    \node at (1.5,-3) {\scriptsize $\mN_5$};
    \fill[almond] (-3,-3) circle (12.0 pt);  
    \node at (-3,-3) {\scriptsize $\mN_3$};
    \fill[almond] (-1,-3) circle (12.0 pt);  
    \node at (-1,-3) {\scriptsize $\mN_4$};
    \fill[almond] (5,-3) circle (12.0 pt);  
    \node at (5,-3) {\scriptsize $\mN_6$};
    \fill[almond] (3.5,0) circle (12.0 pt);  
    \node at (3.5,0) {\scriptsize $\mN_7$};
    \fill[almond] (3.5,-6) circle (12.0 pt);  
    \node at (3.5,-6) {\scriptsize $\mN_8$};
    
    \fill[bubblegum] (0,-1.5) circle (9.0 pt);  
    \node at (0,-1.5) {\scriptsize $u_1$};
	\fill[columbiablue] (1.5,0) circle (9.0 pt);  
    \node at (1.5,0) {\scriptsize $y_1$};
    \fill[bubblegum] (-1.3,-6) circle (9.0 pt);  
    \node at (-1.3,-6) {\scriptsize $u_2$};
	\fill[columbiablue] (0,-4.5) circle (9.0 pt);  
    \node at (0,-4.5) {\scriptsize $y_2$};
   \fill[bubblegum] (-3,-1.5) circle (9.0 pt);  
    \node at (-3,-1.5) {\scriptsize $u_3$};
	\fill[columbiablue] (-3,-4.5) circle (9.0 pt);  
    \node at (-3,-4.5) {\scriptsize $y_3$};
	\fill[columbiablue] (0-1,-4.3) circle (9.0 pt);  
    \node at (-1,-4.3) {\scriptsize $y_4$}; 
    \fill[bubblegum] (1.5,-1.7) circle (9.0 pt);  
    \node at (1.5,-1.7) {\scriptsize $u_5$};
	\fill[columbiablue] (2.7,-3) circle (9.0 pt);  
    \node at (2.7,-3) {\scriptsize $y_5$};
	\fill[bubblegum] (3.8,-3) circle (9.0 pt);  
    \node at (3.8,-3) {\scriptsize $u_6$};
	\fill[columbiablue] (5,-4.3) circle (9.0 pt);  
    \node at (5,-4.3) {\scriptsize $y_6$};
    \fill[bubblegum] (2.2,0) circle (9.0 pt);  
    \node at (2.2,0) {\scriptsize $u_7$};
	\fill[columbiablue] (3.5,-1.5) circle (9.0 pt);  
    \node at (3.5,-1.5) {\scriptsize $y_7$};
    \fill[columbiablue] (2,-6) circle (9.0 pt);  
    \node at (2,-6) {\scriptsize $y_8$};
	\fill[bubblegum] (5,-6) circle (9.0 pt);  
    \node at (5,-6) {\scriptsize $u_8$};
    
    \draw (-0.1,-0.3)->(-3,-2.7);
    \draw (-0.1,-0.3)->(-0.8,-2.8);
    \draw (1.5,-3.3)->(0,-5.7);
    \draw (3.5,-5.6)->(1.5,-3.3);
    \draw (5,-2.6)->(3.5,-0.3);
    
    \draw(0,-1.2)->(0,-0.3);
    \draw(0.4,0)->(1.2,0);
    
    \draw(-1.1,-6)->(-0.3,-6);
    \draw(0,-5.6)->(0,-4.7);

	\draw(-3,-1.8)->(-3,-2.7);
	\draw (-3,-3.3)->(-3,-4.2);
	
	\draw (-1,-3.3)->(-1,-4);      

	\draw (1.5,-2)->(1.5,-2.7);
	\draw (1.95,-3)->(2.5,-3);
	
	\draw (4.1,-3)->(4.7,-3);
	\draw (5,-3.3)->(5,-4);
	
	\draw (2.5,0)->(3.2,0);
	\draw (3.5,-0.4)->(3.5,-1.3);
	
	\draw (3.1,-6)->(2.3,-6);
	\draw (4.7,-6)->(3.8,-6);
	
	\draw[red] (0,-4.2)->(0,-1.8);
	\draw[red] (1.5,-0.3)->(1.5,-1.5);
	\draw[red] (-3,-4.8)->(-1.6,-6);
	\draw[red] (-1,-4.6)->(-1.3,-5.8);
	\draw[red] (3.2,-1.5)->(2.2,-0.3);
	\draw[red] (3.2,-1.5)->(1.8,-1.7);
	\draw[red] (3,-3)->(3.5,-3);
	\draw[red] (4.7,-4.3)->(3.8,-3.3);
	\draw[red] (5,-4.6)->(5,-5.7);

;
\path[every node/.style={font=\sffamily\small}]
(-3.3,-4.3)  edge[red, bend left = 92] node [left] {} (-3.2,-1.5)
(2,-6.3) edge[red, bend right = 50] node [left] {} (5,-6.3);
\end{tikzpicture}
\end{minipage}~\hspace*{2 mm}
\begin{minipage}{.22\textwidth}
\caption{\small {\scalefont{0.9}{Illustrative figure demonstrating the variables $\tilde{k}_1$ and $\tilde{k}_2$. $\bfC=\{\mC(1),\mC(2),\mC(3)$ and $\mC(4)\}$. Then, $k_1(\mC(1))=3$, $k_1(\mC(2))=2$,  $k_1(\mC(3))=3$, $k_1(\mC(4))=2$. Similarly $k_2(\mC(1))=1$, $k_2(\mC(2))=2$, $k_2(\mC(3))=2$, $k_2(\mC(4))=1$. Thus, $\tilde{k}_1=2$ and $\tilde{k}_2=1$.}}} 
\label{fig:k1tilda}
\end{minipage}
\vspace*{-7 mm}
\end{figure}
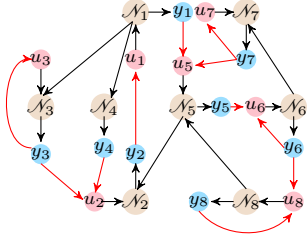
\begin{lem}\label{lem:two} 
Consider a structured system $(\bA, \bB, \bC)$ and cost matrix $P$. Let $\mC^1_{opt}$ be an optimal cycle cover and $H\subseteq E_K$ be the output of Algorithm~\ref{algo:three}, which takes as input a set of cycles and a set of feedback edges. Then, $c(H) \leqslant \tilde{k}_1\, (1 + \log\,|\mN|)\, c(E_{opt})$, where $\tilde{k}_1$ is the highest multiplicity of an edge in the cycle set $\mC^1_{opt}$ and $E_{opt}$ is an optimal solution to Problem~\ref{prob:two}.
\end{lem}
\begin{proof}
Given $\mC^1_{opt}$ is an optimal solution to Problem~\ref{prob:two}. We define the total cost of cycles $c_{tot}$ as
\begin{eqnarray}
c_{tot} & =& \sum_{\C_i\in \mC^1_{opt}} c(E_i).
\label{eq:ctot1}
\end{eqnarray}
Since $\tilde{k}_1$ is the highest multiplicity edge, in the edge set $E^1_{opt}:=\cup_{\C_i \in \mC^1_{opt}} E_i$, corresponding to $\mC^1_{opt}$, from \eqref{eq:ctot1} 
\begin{equation}\label{eq:ctot2}
c_{tot}  \leqslant  \sum_{\tilde{e}_i\in E^1_{opt}}(\tilde{k}_1\, c(\tilde{e}_i)) = \tilde{k}_1(\sum_{\tilde{e}_i\in E^1_{opt}} c(\tilde{e}_i)) = \tilde{k}_1\times c(E^1_{opt}).
\end{equation}
Let in $v^{\rm th}$ iteration of the \textbf{while} loop (Steps~\ref{step:p(C)}-\ref{step:C_update}), $\widetilde{\C}_{ns}(v)=\{\widetilde{\C}^1_{ns}(v),\dots,\widetilde{\C}^z_{ns}(v)\}\subseteq \mC^1_{opt}$, where $\widetilde{\C}^i_{ns}(v) =(\{\widetilde{N}^i_{ns}(v)\}:[\widetilde{E}^i_{ns}(v)])$, be the set of cycles not yet selected by the greedy scheme described in Algorithm~3. Since $\widetilde{\C}_{ns}(v)\subseteq \mC^1_{opt}$,
\begin{eqnarray}
c_{tot} & \geqslant & \sum_{\widetilde{\C}^i_{ns}(v)\in \widetilde{\C}_{ns}(v)} c(\widetilde{E}^i_{ns}(v)).\label{eq:ctot3}
\end{eqnarray}
From \eqref{eq:ctot2} and~\eqref{eq:ctot3}, we get
\begin{eqnarray}
\tilde{k}_1\times c(E^1_{opt})\hspace{-2 mm}&\geqslant &\hspace*{-2 mm}\sum_{{\widetilde{\C}^i}_{ns}(v)\in \widetilde{\C}_{ns}(v)}c(\widetilde{E}^i_{ns}(v)), \nonumber\\
 & = &\hspace*{-2 mm} c(\widetilde{E}^1_{ns}(v)) + \dots + c(\widetilde{E}^z_{ns}(v)), \nonumber\\
& = &\hspace*{-2 mm}  {\scalebox{1.0}{\mbox{$|\widetilde{N}^1_{ns}(v)|\, \frac{c(\widetilde{E}^1_{ns}(v))}{|\widetilde{N}^1_{ns}(v)|} + \dots + |\widetilde{N}^z_{ns}(v)|\, \frac{c(\widetilde{E}^z_{ns}(v))}{|\widetilde{N}^z_{ns}(v)|}$}}}\nonumber
\end{eqnarray}
The ratio of the cost of each cycle $\C_i$ to the number of nodes it will cover is denoted by $\rho({\C_i})$ (Step~\ref{step:p(C)} of Algorithm~\ref{algo:three}), i.e., $c({E}_i)/|{N}_i| = \rho({\C_i})$. Let the cycle $\C_j$ with minimum price is selected greedily in the current iteration. Then, $\rho(\C_j)\leqslant \rho(\widetilde{\C}^i_{ns}(v))$, for $i = 1,\dots,z$. So,
\begin{eqnarray}
\tilde{k}_1\times c(E^1_{opt})  & \geqslant & \sum_{\widetilde{\C}^i_{ns}(v)\in \widetilde{\C}_{ns}(v)} \rho(\C_j)\times |\widetilde{N}^i_{ns}(v)|,\nonumber\\
  & = & \rho(\C_j)\times (\sum_{\widetilde{\C}^i_{ns}(v)\in \widetilde{\C}_{ns}(v)}|\widetilde{N}^i_{ns}(v)|),\nonumber\\
  & \geqslant & \rho(\C_j)\times (|\cup_{\widetilde{\C}^i_{ns}(v)\in \widetilde{\C}_{ns}(v)}\widetilde{N}^i_{ns}(v)|).\nonumber
\end{eqnarray}
Notice that $\widetilde{\C}_{ns}(v)$ covers nodes $\mN \setminus I$, where $I$ is the set of nodes in $\mN$ covered till the $v^{\rm th}$ iteration of the \textbf{while} loop. Let $\mN \setminus I = N_{ns}(v)$. Thus $|N_{ns}(v)| = |\cup_{\widetilde{\C}^i_{ns}(v)\in \widetilde{\C}_{ns}(v)}\widetilde{N}^i_{ns}(v)|$.  
\begin{eqnarray}
\tilde{k}_1\times c(E^1_{opt}) & \geqslant & \rho(\C_j) \times |N_{ns}(v)| ,\nonumber\\
\rho(\C_j) & \leqslant & \tilde{k}_1\times \frac{c(E^1_{opt})}{|N_{ns}(v)|}.\label{eq:cycle_cost}
\end{eqnarray}
Let the sequence of cycles selected by Algorithm~\ref{algo:three} be $\hat{\C}=\{\hat{\C}_1,\dots,\hat{\C}_d\}$. In $v^{\rm th}$ iteration, let the number of nodes covered by cycle $\hat{\C}_v$ be given by $\hat{n}_v$. Here ${|N_{ns}(v)|}$ is the number of nodes yet to be covered after $(v-1)$ iterations. Thus $N_{ns}(1)=\mN$. Also, by \eqref{eq:cycle_cost}, $\rho(\hat{\C}_v)\leqslant \tilde{k}_1\, \frac{c(E^1_{opt})}{|N_{ns}(v)|}$. The cost incurred when selecting cycle $\hat{\C}_v$ is $\rho(\hat{\C}_v)\times\,\hat{n}_v$. So, the total cost incurred
\begin{eqnarray*}
c(H)&=& \sum_{\hat{\C}_v\in \hat{\C}} {\scalefont{0.9}{\rho(\hat{\C}_v)\times \hat{n}_v,}} \\
&{\scalefont{0.9}{\leqslant}}& {\scalefont{0.9}{\tilde{k}_1 \, c(E^1_{opt}) \Big( \frac{\hat{n}_1}{|N_{ns}(1)|} + \ldots + \frac{\hat{n}_d}{|N_{ns}(d)|}\Big),}} \nonumber \\
&=& \tilde{k}_1 \, c(E^1_{opt}) \Big(\frac{\hat{n}_1}{|\mN|} + \ldots + \frac{\hat{n}_d}{|N_{ns}(d)|}\Big),\\
\small &=& \tilde{k}_1 \, c(E^1_{opt})  \Big( {\scalebox{0.7}{\mbox{$ \underbrace{\scriptsize{\frac{1}{|\mN|}}+\dots+\frac{1}{|\mN|}}_{\hat{n}_1 ~{\rm times}}+ \underbrace{\frac{1}{|\mN|- \hat{n}_1}+\dots+ \frac{1}{|\mN|-\hat{n}_1}}_{{\scalefont{1.2}{\hat{n}_2~{\rm times}}}}$}}} \\ 
&& {\scalebox{0.7}{\mbox{$+ \ldots+ \underbrace{\frac{1}{|\mN|-\sum\limits_{i=1}^{d-1}\hat{n}_i}+\ldots+\frac{1}{|\mN|-\sum\limits_{i=1}^{d-1}\hat{n}_i}}_{\hat{n}_d~{\rm times}}\Big)$}}},\\
&{\scalefont{0.9}{\leqslant }}&{\scalefont{0.9}{ \tilde{k}_1 \, c(E^1_{opt}) \, \Big(1 + \mbox{log}(|\mN|) \Big),}}\\
&{\scalefont{0.9}{= }}& {\scalefont{0.9}{\tilde{k}_1 \, c(E_{opt}) \Big(1 + \mbox{log}(|\mN|) \Big).}}
\end{eqnarray*}
Thus $c(H)\leqslant\ \tilde{k}_1 \, c(E_{opt}) (1+\log|\mN|)$.
\end{proof}

\begin{rem}
Let $\mC^1_{opt}$ be an optimal cycle set that solves Problem~\ref{prob:two} and the highest multiplicity of a feedback edge in $\mC^1_{opt}$ be $\tilde{k}_1$. Notice that $|\mC^1_{opt}| \leqslant |\mN|$ because in optimal solution each cycle covers atleast one different node. Hence, $\tilde{k}_1 \leq |\mN|$. 
\end{rem}
\begin{algorithm}
\caption{Pseudo-code to find an approximate solution to Problem~\ref{prob:two}}\label{algo:four}
\textbf{Input}: Cycle set $\C = \{\C_1,\dots,\C_t\}$, where $\C_i :=(\{N_i\} : [E_i])$\\
\textbf{Output}: Set of feedback edges $H_A$
\begin{algorithmic}[1]
\State {\scalefont{0.9}{Initialize the set of covered nodes as $I_A\leftarrow\emptyset$}}\label{step:IA}
\State {\scalefont{0.9}{Initialize the set of selected edges as $H_A\leftarrow\emptyset$}}\label{step:HA} 
\State {\scalefont{0.9}{Define $H_A(\C_i) \leftarrow\textsc{Greedy}\Big(\cup_{j=1}^t N_j\setminus N_i,E_i\Big) $}}\label{step:rungreedy}
\State {\scalefont{0.9}{Define $\textsc{pot}(\C_i)\leftarrow c(E_i) + c(H_A(\C_i))$}}\label{step:potdef}
\While {\scalefont{0.9}{ {$I\neq \mN$}}}
\State {\scalefont{0.9}{Calculate $\textsc{pot}(\C_k)$, for $k=1,\dots,|\C|$}}\label{step:potcal}
\State {\scalefont{0.9}{Select $\C_j\in \arg\min_{\C_i\in \C} \textsc{pot}(\C_i)$}}\label{step:potmin}
\State {\scalefont{0.9}{$I_A\leftarrow I_A\cup N_j$, $H_A\leftarrow H_A\cup E_j$}}\label{step:IAupadte}
\State {\scalefont{0.9}{$N_k\leftarrow N_k/I_A$, $E_k\leftarrow E_k/H_A$, for $k=1,\dots,|\C|$}}\label{step:NKupdate}
\EndWhile
\State {\scalefont{0.9}{Return $H_A$}}
\end{algorithmic}
\end{algorithm}

The pseudo-code for finding an approximate solution to Problem~\ref{prob:two} is presented in Algorithm~\ref{algo:four}. This algorithm incorporates the greedy algorithm given in Algorithm~\ref{algo:three} with a potential function. Here, $I_A$ and $H_A$ are defined as the set of nodes covered and the set of feedback edges selected, respectively. Our purpose is to make $I_A=\mN$. Consider a cycle $\C_i \in \C$. The potential of a cycle is defined in the following way. We apply the greedy scheme discussed in Algorithm~\ref{algo:three} with input $(\cup_{j=1}^t N_j/N_i,E_i)$ and let the solution obtained be the edge set $H_A(\C_i)$ (Step~\ref{step:rungreedy}). Notice that $H_A(\C_i)\cap E_i =\emptyset$ because we removed the edge set $E_i$ from all $E_j$'s before applying the greedy scheme (see Algorithm~\ref{algo:three}). The potential of cycle $\C_i$ is then defined as the sum of $c(E_i)$ and $c(H_A(\C_i))$ (Step~\ref{step:potdef}). Also, the edge set $E_i\cup H_A(\C_i)$ is a feasible solution to Problem~\ref{prob:two}, as $E_i$ covers $N_i$ and $H_A(\C_i)$ covers $(\cup_{j=1}^t N_j/N_i)$. After calculating the potential for each $\C_i \in \C$, we select a cycle with minimum potential value, say $\C_j$ (Step~\ref{step:potmin}). The node set covered and the edge set selected till current iteration is updated as in Step~\ref{step:IAupadte}. Also, the edge set $E_j$ is removed from remaining edge sets for all $\C_k\in \C\setminus \C_j$ (Step~\ref{step:NKupdate}).
In Theorem~\ref{theorem:four}, we prove that Algorithm~\ref{algo:four} gives an approximate solution to Problem~\ref{prob:two} with approximation ratio $\tilde{k}_2 (1+\log|\mN|)$.
\begin{theorem}\label{theorem:four}
Algorithm~\ref{algo:four} which takes as input a cycle set $\C=\{\C_1,\dots,\C_t\}$  outputs a solution $H_A$ to Problem~\ref{prob:two} such that $c(H_A)\leqslant\tilde{k}_2\,(1 + {\rm log}|{\mN}|)\, c(E_{opt})$, where $E_{opt}$ is an optimal solution to Problem~\ref{prob:two}. In other words, output of Algorithm~\ref{algo:four} is a $\tilde{k}_2\,(1 + {\rm log}|{\mN}|)$-optimal solution to Problem~\ref{prob:two}. 
\end{theorem}
\begin{proof}
Let $\mC^2_{opt}$ be an optimal solution of Problem~\ref{prob:two}. Recall the definition of $\tilde{k}_2$. Let the highest multiplicity of a feedback edge in $\mC^2_{opt}$ be $k'$ and the corresponding edge be ${e'}$. Consider the cycle $\widetilde{\C}_1 \in \mC^2_{opt}$, where $\widetilde{\C}_1 = (\{\widetilde{N}_1\}:[\widetilde{E}_1])$ such that ${e'}\in \widetilde{E}_1$. Let $H_A(\widetilde{\C}_1) := \textsc{greedy}(\cup_{j=1}^t N_j/\widetilde{N}_1,\widetilde{E}_1)$. The potential of cycle $\widetilde{\C}_1$ is given by
$\textsc{pot}(\widetilde{\C}_1) = c(\widetilde{E}_1) + c(H_A(\widetilde{\C}_1)).
$

Let $E^2_{opt}$ be the set of feedback edges corresponding to $\mC^2_{opt}$. Note that, an optimal edge set to cover the nodes $\mN{\setminus} \widetilde{N}_1$ is $E^2_{opt}{\setminus} \widetilde{E}_1$ and the optimal cost is $c(E^2_{opt}){-}c(\widetilde{E}_1)$. Also, since $\widetilde{\C}_1 \in \mC^2_{opt}$, where $\mC^2_{opt}$ is an optimal cycle cover, the highest multiplicity of an edge in $\mC^2_{opt}{\setminus} \widetilde{\C}_1$ is $\tilde{k}_2$. Hence by Lemma~\ref{lem:two}, we have
\begin{eqnarray*}
{\scalefont{0.9}{c(H_A(\widetilde{\C}_1))}} &{\scalefont{0.9}{ \leqslant}} & {\scalefont{0.9}{\tilde{k}_2 \, (1 + \mbox{log}{|\mN\setminus \widetilde{N}_1|}) \, (c(E^2_{opt})-c(\tilde{E}_1)),}}\\
 & {\scalefont{0.9}{\leqslant}} &{\scalefont{0.9}{ \tilde{k}_2\, (1 + \mbox{log}{|\mN|})\, (c(E^2_{opt})-c(\widetilde{E}_1)).}}
\end{eqnarray*}
Algorithm~\ref{algo:four} greedily selects a cycle, say $\C_k$, with minimum potential. Then, $\textsc{pot}(\C_k)\leqslant \textsc{pot}({\widetilde{\C}_1})$. Hence, 
\begin{eqnarray}\label{eqn:potent}
{\scalefont{0.9}{\textsc{pot}(\C_k)}} &{\scalefont{0.9}{ \leqslant}} & {\scalefont{0.9}{c(\widetilde{E}_1) + \tilde{k}_2\, (1 + \mbox{log}{|\mN|})\,\Big(c(E^2_{opt})-c(\widetilde{E}_1)\Big),}}\nonumber\\
& {\scalefont{0.9}{\leqslant}} &{\scalefont{0.9}{ \tilde{k}_2\, (1 + \mbox{log}{|\mN|})\, c(E^2_{opt}).}}
\end{eqnarray}
Equation~\eqref{eqn:potent} holds since $\tilde{k}_2\,(1+\mbox{log}|\mN|) \geqslant 1$.
Notice that $\textsc{pot}(\C_k)$ is the cost of the edge set obtained by selecting cycle $\C_k$ and then applying greedy scheme on the remaining $\mN\setminus
N_k$ nodes. Hence, edge set $E_k \cup H_A(\C_k)$ is a solution to Problem~\ref{prob:two}.
Therefore, after the first iteration of the \textbf{while} loop of Algorithm~\ref{algo:four}, we obtain a solution to Problem~\ref{prob:two}, the cost of which is bounded by $\tilde{k}_2\, (1 + \mbox{log}{|\mN|})\, c(E^2_{opt}) = \tilde{k}_2\, (1 + \mbox{log}{|\mN|})\, c(E_{opt})$. Thus Algorithm~\ref{algo:four} gives an approximate solution to Problem~\ref{prob:two} with approximation ratio $\tilde{k}_2 (1+\mbox{log}{|\mN|}))$. This completes the proof.
\end{proof}

The result below gives the computational complexity of Algorithm~\ref{algo:four}.
\begin{theorem}\label{th:comp}
Consider a structured system $(\bA,\bB=\mathbb{I}_m,\bC=\mathbb{I}_p)$ and feedback cost matrix $P$. Algorithm~4, which takes as input a set of cycles $\C$ and gives as output the feedback edge set $H_A$, has complexity $O(n^2\,|\C|^2)$, where $n$ denotes the system dimension and $|\C|$ is the number of cycles in $\D_R$. 
\end{theorem}
\begin{proof}
Finding all cycles in the digraph $\D_R$ has complexity $O(n^2  |\C|)$ \cite{Joh:75} as the number of SCCs in $\D(\bA)$ are in $O(n)$, where $n$ is the number of state nodes in the structured system $(\bA,\bB,\bC)$. Algorithm~\ref{algo:three} finds the price for $|\C|$ cycles in each iteration and the number of iterations are $O(n)$. Hence, Algorithm~\ref{algo:three} has complexity $O(n \,|\C|)$. In Algorithm~\ref{algo:four}, Algorithm~\ref{algo:three} is called as a subroutine $O(n\,|\C|)$ times. All the other steps in Algorithm~\ref{algo:four} are of linear complexity. Hence, the complexity of Algorithm~\ref{algo:four} is $(n^2\,|\C|^2)$. 
\end{proof}

\begin{rem}
{\it Cycle merging}: A cycle merging operation can be performed on the cycle set $\C$ in $\D_R$ before applying Algorithm~\ref{algo:four}. For all $\C_a,\C_b \in \C$, if $E_a \subset E_b$, then we merge the cycle $\C_a$ with cycle $\C_b$, i.e., $\C_b = (\{N_a \cup N_b\}:[E_b])$. 
\end{rem}
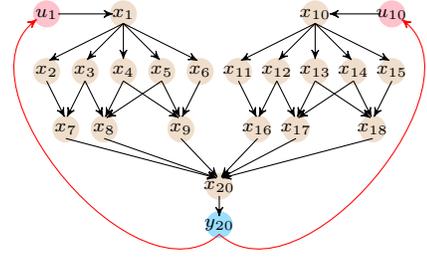
\begin{figure}[t]
\begin{center}
\begin{tikzpicture}[scale=0.51, ->,>=stealth',shorten >=1pt,auto,node distance=1.65cm, main node/.style={circle,draw,font=\scriptsize\bfseries}]
\definecolor{myblue}{RGB}{80,80,160}
\definecolor{almond}{rgb}{0.94, 0.87, 0.8}
\definecolor{bubblegum}{rgb}{0.99, 0.76, 0.8}
\definecolor{columbiablue}{rgb}{0.61, 0.87, 1.0}

  \fill[almond] (-1,-1.5) circle (10.0 pt);
  \fill[almond] (-2,-1.5) circle (10.0 pt);
  \fill[almond] (0,0) circle (10.0 pt);
  \fill[almond] (0,-1.5) circle (10.0 pt);
  \fill[almond] (1,-1.5) circle (10.0 pt);
  \fill[almond] (2,-1.5) circle (10.0 pt);
  \node at (-2,-1.5) {\scriptsize $x_2$};
  \node at (-1,-1.5) {\scriptsize $x_3$};
  \node at (0,0) {\scriptsize $x_1$};
  \node at (0,-1.5) {\scriptsize $x_4$};
  \node at (1,-1.5) {\scriptsize $x_5$};
  \node at (2,-1.5) {\scriptsize $x_6$};

  \fill[bubblegum] (-2,0) circle (10.0 pt);
   \node at (-2,0) {\scriptsize $u_1$};
   
   \fill[bubblegum] (7,0) circle (10.0 pt);
   \node at (7,0) {\scriptsize $u_{10}$};

  \fill[almond] (-1.5,-3) circle (10.0 pt);
  \fill[almond] (-0.5,-3) circle (10.0 pt);
  \fill[almond] (1.5,-3) circle (10.0 pt);

   \node at (-1.5,-3.0) {\scriptsize $x_7$};
   \node at (-0.5,-3.0) {\scriptsize $x_8$};
   \node at (1.5,-3.0) {\scriptsize $x_9$};
   
;
   
  \draw (-1.7,0)  ->   (-0.2,0);
  \draw (6.7,0)  ->   (5.3,0);   
  


  \draw (0,-0.25)  ->   (-2,-1.25);
  \draw (0,-0.25)  ->   (-1,-1.25);
  \draw (0,-0.25)  ->   (0,-1.25);
  \draw (0,-0.25)  ->   (1,-1.25);
  \draw (0,-0.25)  ->   (2,-1.25);

  \draw (-2,-1.75)  ->   (-1.5,-2.75);
  \draw (-1,-1.75)  ->   (-1.5,-2.75);
  \draw (-1,-1.75)  ->   (-0.5,-2.75);
  \draw (0,-1.75)  ->   (-0.5,-2.75);
  \draw (1,-1.75)  ->   (1.5,-2.75);
  \draw (1,-1.75)  ->   (-0.5,-2.75);
  \draw (2,-1.75)  ->   (1.5,-2.75);
  \draw (0,-1.75)  ->   (1.5,-2.75); 
  
  \fill[almond] (3,-1.5) circle (10.0 pt);
  \fill[almond] (4,-1.5) circle (10.0 pt);
  \fill[almond] (5,0) circle (10.0 pt);
  \fill[almond] (5,-1.5) circle (10.0 pt);
  \fill[almond] (6,-1.5) circle (10.0 pt);
  \fill[almond] (7,-1.5) circle (10.0 pt);
  \node at (3,-1.5) {\scriptsize $x_{11}$};
  \node at (4,-1.5) {\scriptsize $x_{12}$};
  \node at (5,0) {\scriptsize $x_{10}$};
  \node at (5,-1.5) {\scriptsize $x_{13}$};
  \node at (6,-1.5) {\scriptsize $x_{14}$};
  \node at (7,-1.5) {\scriptsize $x_{15}$};

  \fill[almond] (3.5,-3) circle (10.0 pt);
  \fill[almond] (4.5,-3) circle (10.0 pt);
  \fill[almond] (6.5,-3) circle (10.0 pt);
  \fill[almond] (2.5,-4.5) circle (10.0 pt);
  \fill[columbiablue] (2.5,-5.5) circle (10.0 pt);

   \node at (3.5,-3.0) {\scriptsize $x_{16}$};
   \node at (4.5,-3.0) {\scriptsize $x_{17}$};
   \node at (6.5,-3.0) {\scriptsize $x_{18}$};
   
   \node at (2.5,-4.5) {\scriptsize $x_{20}$};
   
   \node at (2.5,-5.5) {\scriptsize $y_{20}$};
;
   
  \draw (-1.5,-3.25)  ->   (2.5,-4.3);
  \draw (-0.5,-3.25)  ->   (2.5,-4.3);
  \draw (1.5,-3.25)  ->   (2.5,-4.3);   
  \draw (3.5,-3.25)  ->   (2.5,-4.3);
  \draw (4.5,-3.25)  ->   (2.5,-4.3);
  \draw (6.5,-3.25)  ->   (2.5,-4.3);

  \draw (5,-0.25)  ->   (3,-1.25);
  \draw (5,-0.25)  ->   (4,-1.25);
  \draw (5,-0.25)  ->   (5,-1.25);
  \draw (5,-0.25)  ->   (6,-1.25);
  \draw (5,-0.25)  ->   (7,-1.25);

  \draw (3,-1.75)  ->   (3.5,-2.75);
  \draw (4,-1.75)  ->   (3.5,-2.75);
  \draw (4,-1.75)  ->   (4.5,-2.75);
  \draw (5,-1.75)  ->   (4.5,-2.75);
  \draw (6,-1.75)  ->   (6.5,-2.75);
  \draw (6,-1.75)  ->   (4.5,-2.75);
  \draw (7,-1.75)  ->   (6.5,-2.75);
  \draw (5,-1.75)  ->   (6.5,-2.75);
  
  \draw (2.5,-4.75)  ->   (2.5,-5.35);

\path[every node/.style={font=\sffamily\small}]
(2.5,-5.75)  edge[red, bend left = 92] node [left] {} (-2.2,-0.1)
(2.5,-5.75) edge[red, bend right = 92] node [left] {} (7.2,-0.1
);
%
\end{tikzpicture}
\vspace*{-9 mm}
\caption{\small Illustrative figure demonstrating the merging operation. Each state vertex $x_k$ has input $u_k$ and output $y_k$ connected which are omitted for many $x_k$'s for the sake of clarity, i.e, feedback edges $(y_k,u_k)$ for all $k=1,\ldots,20$ are present in the system. }
\label{fig:merge}
\end{center}
\vspace*{-7 mm}
\end{figure} 
Notice that after the merging operation, the cost $c(E_b)$ of selecting the cycle $\C_b$ does not change, but the number of nodes covered can increase resulting in a better ratio of cost to nodes covered, $\rho(\C_b)$. The bound achieved in Algorithm~\ref{algo:four} has a factor of $\tilde{k}_2$. As a result of this merging operation, the optimal edge set does not change, but the multiplicity $\tilde{k}_2$ can decrease resulting in a better approximation and lower complexity of Algorithm~\ref{algo:four}. An illustrative example showing merging operation is shown in Figure~\ref{fig:merge}. Assume that an optimal solution to the given system is the set of edges $(y_{20},u_1)$ and $(y_{20},u_{10})$. Then both $\tilde{k}_1$ and $\tilde{k}_2$ are $8$ and can possibly be very high as the number of nodes increases. If we perform the merging operation as mentioned above, $\tilde{k}_2$ becomes $1$. Broadly, the merging operation simplifies the proposed algorithm and requires more detailed analysis.

\begin{rem}
Notice that in Algorithm~\ref{algo:four}, only the first iteration of the \textbf{while} loop is used to prove an approximation ratio of $\tilde{k}_2\, ({1+\rm log(|\mN|)})$. The cost of the final edge set obtained when Algorithm~\ref{algo:four} terminates will be atmost $\tilde{k}_2\, ({1+\rm log(|\mN|)})\,(c(E_{opt}))$, i.e., lesser cost than $\tilde{k}_2\,({1+ \rm log(|\mN|)})\,(c(E_{opt}))$.
\end{rem}
The following section considers two special cases of Problem~1 of practical importance and we propose polynomial time algorithms to obtain approximate and optimal solutions to the two cases, respectively. 
\vspace*{-1 mm}
\section{Special cases}\label{sec:spcases}
In this section, we consider two special graph topologies: (i)~structured systems with back-edge feedback structure and (ii)~hierarchical network.
\subsection{Structured systems with back-edge feedback structure}
In this subsection, we consider a special class of structured systems with a constraint on the structure of the feedback matrix. We assume that the only feasible feedback edges $(y_j,u_i)'$s are those edges where there exists a directed path from input $u_i$ to output $y_j$ in $\D(\bA,\bB,\bC)$. In other words, the assumption states that an output from a state is fed back to an input which can directly or indirectly influence the state associated with that output. A feedback structure that satisfies this constraint is referred as a {\em back-edge} feedback structure. Note that, inputs and outputs are dedicated here. For this class of systems we propose a polynomial time algorithm to find an approximate solution to Problem~\ref{prob:one} with an optimal approximation ratio. We describe below the graph topology considered in this subsection.


\begin{defn}\label{defn:ance}
Consider a digraph $\D_G:=(V_G,E_G)$. Let the nodes $v_i, v_j \in V_G$ be such that there exists a directed path from $v_i$ to $v_j$. Then, $v_i$ is referred as an \underline{ancestor} of $v_j$. Also, node $v_j$ is referred as a \underline{descendant} of node $v_i$.
\end{defn}

\begin{assume}\label{assume:spcase1}
Consider a structured system $(\bA,\bB={\mathbb{I}_m},\bC={\mathbb{I}_p})$ and a feedback cost matrix $P \in \mathbb{R}^{m\times p}$, where $P_{ij}$ denotes the cost of feeding the $j^{\rm th}$ output to the $i^{\rm th}$ input. Then, $P_{ij} =\infty$, if the input node $u_i$ is \underline{not} an ancestor of the output node $y_j$ in  $\D(\bA,\bB,\bC)$.
\end{assume}
Recall that if $P_{ij}=\infty$, then the feedback edge $\bK_{ij}$ is infeasible. Thus Assumption~\ref{assume:spcase1} concludes that an output $y_j$ can be fed to an input $u_i$ only if $u_i$ is an ancestor of $y_j$ in $\D(\bA,\bB,\bC)$. If $u_i$ is not an ancestor of $y_j$, then $(y_j,u_i)$ is an infeasible feedback link. An illustrative example showing feasible and infeasible feedback connections in a structured system is presented in Figure~\ref{fig:hierar1}.
\begin{figure}[t]
\begin{minipage}{.22\textwidth}
\begin{tikzpicture}[scale=0.40, ->,>=stealth',shorten >=1pt,auto,node distance=1.65cm, main node/.style={circle,draw,font=\scriptsize\bfseries}]
\definecolor{myblue}{RGB}{80,80,160}
\definecolor{almond}{rgb}{0.94, 0.87, 0.8}
\definecolor{bubblegum}{rgb}{0.99, 0.76, 0.8}
\definecolor{columbiablue}{rgb}{0.61, 0.87, 1.0}


\fill[almond] (0,0) circle (12.0 pt);  
  \node at (0,0) {\scriptsize $x_1$};

  \fill[almond] (-4,-2) circle (12.0 pt);
  \fill[almond] (0,-2) circle (12.0 pt);
  \fill[almond] (4,-2) circle (12.0 pt);
  
  \node at (-4,-2) {\scriptsize $x_2$};  
  \node at (0,-2) {\scriptsize $x_3$};  
  \node at (4,-2) {\scriptsize $x_4$};
   
  \fill[almond] (-4.8,-4) circle (12.0 pt);
  \fill[almond] (-3.2,-4) circle (12.0 pt);
   \node at (-4.8,-4) {\scriptsize $x_5$};
   \node at (-3.2,-4) {\scriptsize $x_6$};
   
   \fill[almond] (-0.8,-4) circle (12.0 pt);
   \fill[almond] (0.8,-4) circle (12.0 pt);
   \node at (-0.8,-4) {\scriptsize $x_{7}$};
   \node at (0.8,-4) {\scriptsize $x_{8}$};

  \fill[almond] (4,-4) circle (12.0 pt);
   \node at (4,-4) {\scriptsize $x_{9}$}; 
   \fill[almond] (0.8,-6) circle (12.0 pt);
   \node at (0.8,-6) {\scriptsize $x_{13}$};
   
   \fill[almond] (-3.2,-6) circle (12.0 pt);
   \fill[almond] (-2.2,-6) circle (12.0 pt);
  \node at (-3.2,-6) {\scriptsize $x_{11}$};
  \node at (-2.2,-6) {\scriptsize $x_{12}$};
  
  \fill[almond] (4,-6) circle (12.0 pt);
  \node at (4,-6) {\scriptsize $x_{14}$};

   \fill[almond] (0.1,-8) circle (12.0 pt);
   \fill[almond] (1.7,-8) circle (12.0 pt);
   \node at (0.1,-8) {\scriptsize $x_{17}$};
   \node at (1.7,-8) {\scriptsize $x_{18}$};
   
   \fill[almond] (-4.0,-8) circle (12.0 pt);
   \fill[almond] (-2.4,-8) circle (12.0 pt);
   \node at (-4.0,-8) {\scriptsize $x_{15}$};
   \node at (-2.4,-8) {\scriptsize $x_{16}$};
   
   \fill[bubblegum] (-4.8,-2.6) circle (8.0 pt);
   \node at (-4.8,-2.6) {\scriptsize $u_1$};
   \fill[columbiablue] (-4.6,-6) circle (8.0 pt);
   \node at (-4.6,-6) {\scriptsize $y_1$};
  
\fill[columbiablue] (3,-8) circle (8.0 pt);
   \node at (3,-8) {\scriptsize $y_2$};
   \fill[bubblegum] (4,-0.5) circle (8.0 pt);
   \node at (4,-0.5) {\scriptsize $u_2$};  
  
  \draw (4,-0.8)->(4,-1.7);
  \draw (2.1,-8)->(2.8,-8);

  \draw (-3.5,-6)->(-4.4,-6);
	\draw (-4.8,-2.9)->(-4.8,-3.7);

	\draw (0,-0.3)  ->   (0,-1.7);
	\draw (0,-0.3)  ->   (-4,-1.7);
	\draw (0,-0.3)  ->   (4,-1.7);
	
	\draw (-4.8,-3.6)-> (-4,-2.3);
	\draw (-4,-2.3)  ->   (-3.2,-3.7);
	
	\draw (0,-2.3)  ->   (0.8,-3.7);
	
	\draw (4,-2.3)  ->   (4,-3.7);
	
	\draw (0.8,-4.3)  ->   (0.8,-5.7);
	
	\draw (0.8,-6.3)  ->   (0.1,-7.7);
	\draw  (1.7,-7.7) -> (0.8,-6.3);
	
	\draw (-3.2,-4.3) -> (-3.2,-5.7);
	\draw (-3.2,-4.3) -> (-2.2,-5.7);
	
	\draw (4,-4.3) -> (4,-5.7);
	
	\draw (-3.2,-6.3) -> (-4.0,-7.7);
	\draw (-3.2,-6.3) -> (-2.4,-7.7);
	\draw (-2.2,-5.7) -> (-0.8,-4.2);
	\draw (0.8,-3.7) -> (4,-2.2);
	\draw (1.7,-7.7) - > (4,-4.2)
%
 
%

;
\path[every node/.style={font=\sffamily\small}]
(-4.9,-6) edge[red,thick, bend left = 50] node [left] {} (-5.05,-2.6)
(2.8,-7.7)  edge[blue,thick, dashed, bend left = 50] node [left] {} (3.8,-0.5);
\end{tikzpicture}
\end{minipage}\hspace{0.6 cm}
\begin{minipage}{.22\textwidth}
\caption{\small Illustrative figure demonstrating feasible feedback connections. Under Assumption~\ref{assume:spcase1}, feedback edge $(y_1,u_1)$ is feasible while $(y_2,u_2)$ is infeasible.}
\label{fig:hierar1}
\end{minipage}
\vspace*{-5 mm}
\end{figure}
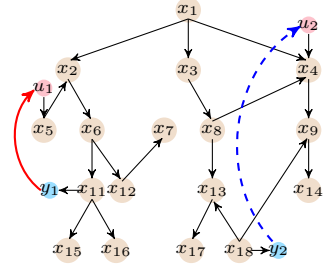


\begin{cor}\label{cor:spcase1}
Consider a structured system $(\bA,\bB={\mathbb{I}_m},\bC={\mathbb{I}_p})$ and a feedback cost matrix $P$ that satisfies Assumption~\ref{assume:spcase1}. For this structured system the following hold: 
\begin{enumerate}
\item[(i)] Problem~\ref{prob:one} is NP-hard,
\item[(ii)] Problem~\ref{prob:one} is inapproximable to a multiplicative factor of ${\rm log\,}n$, where $n$ is the number of states in the system.
\end{enumerate}
\end{cor}

The above corollary is a consequence of the fact that the structured system and the feedback cost matrix obtained in the reduction given in Algorithm~\ref{algo:two} and the NP-hardness proof given in Theorem~\ref{theorem:one} satisfy Assumption~\ref{assume:spcase1}. 

In this subsection, we present a polynomial time approximation algorithm that finds a $(\log\,n)$-approximate solution to Problem~\ref{prob:one}. This algorithm is based on a reduction of Problem~\ref{prob:one} to an instance of the weighted set cover problem. We reduce a general instance of Problem~\ref{prob:one} satisfying Assumption~\ref{assume:spcase1} to an instance of the weighted set cover problem in such a way that an approximation algorithm of the weighted set cover problem will serve as an approximation algorithm for Problem~\ref{prob:one}. To achieve this, we reduce the weighted set cover problem to Problem~\ref{prob:one} and prove in Theorem~\ref{th:spec_1_eps} that any $\epsilon$-optimal solution of the weighted set cover problem is an $\epsilon$-optimal solution to Problem~\ref{prob:one}. 

\begin{algorithm}
\caption{Pseudo-code for reducing a general instance of Problem~\ref{prob:one} following 
Assumption~\ref{assume:spcase1}  to an instance of the weighted set cover problem denoted by $(\pazocal{U}_s, \pazocal{P}_s, w_s)$.\label{algo:five}}
\begin{algorithmic}
\State \textit {\bf Input:} Structured system $(\bA,\bB={\mathbb{I}_m},\bC={\mathbb{I}_p})$ and feedback cost matrix $P$
\State \textit{\bf Output:} Weighted set cover problem $(\pazocal{U}_s, \pazocal{P}_s, w_s)$ 
\end{algorithmic}
\begin{algorithmic}[1]
\State Define $\bK^P:=\{\bK^P_{ij}=\*$ if $P_{ij}\neq\infty\}$\label{step:bkP}
\State Define an instance of the weighted set cover problem as: 
\State Universe $\pazocal{U}_s \leftarrow \{x_1,\ldots, x_n\}$\label{step:Univ_2}
\State Set $\pazocal{P}_s=\{\Ss_1,\ldots,\Ss_{|E_{K^P}|}\}$ \label{step:Setlist_2}
\For {$e_d = (y_j,u_i) \in E_{K^P}$}
    \State $\Ss_d := \{x_a$: $x_a$ lies in an SCC in the digraph formed by adding the feedback edge $e_d = (y_j, u_i)$ to $\D(\bA,\bB,\bC)\}$\label{step:Set_2}
    \State Weight $w_s(\Ss_d) = P_{ij}$\label{step:weight_2}
\EndFor
\State Let $\Ss'$ be a solution to the weighted set cover problem $(\pazocal{U}_s,\pazocal{P}_s,w_s)$\label{step:cover_2}
\State The feedback matrix $\bK(\Ss')$ selected under $\Ss'$, $\bK(\Ss')\leftarrow \{\bK(\Ss')_{ij} =\*: \Ss_d\in \Ss'$ and $e_d = (y_j,u_i)\}$ \label{step:edge_2}
\State Cost of the edge set $\bK(\Ss')$, $P(\bK(\Ss')) = \sum_{ (i,j):\bK(\Ss')_{ij}=\*} P_{ij}$\label{step:cost_2}
\end{algorithmic}
\end{algorithm}

Algorithm~\ref{algo:five} gives the pseudo-code for reducing a general instance of Problem~\ref{prob:one} to an instance of the weighted set cover problem denoted by $(\pazocal{U}_s, \pazocal{P}_s, w_s)$. We define a feedback matrix $\bK^P$, such that $\bK^P$ consists of all feasible feedback edges (Step~\ref{step:bkP}). The universe $\pazocal{U}_s$ of the weighted set cover problem consists of all states $\{x_1, \ldots, x_n\}$ of the system (Step~\ref{step:Univ_2}). The set $\pazocal{P}_s$ is defined in such a way that a set $\Ss_d \in \pazocal{P}_s$ corresponds to a feedback edge $(y_j,u_i)=e_d$ (Step~\ref{step:Setlist_2}). Thus $|\pazocal{P}_s|=|E_{K^P}|$ and each set $\Ss_d$  consists of state nodes in $\D(\bA)$ that lie in an SCC in the digraph formed by adding the feedback edge $(y_j,u_i)$ to $\D(\bA,\bB,\bC)$ (Step~\ref{step:Set_2}). The weight of the set $\Ss_d$ is assigned the cost of the feedback edge $(y_j,u_i)$ (Step~\ref{step:weight_2}). We denote a solution to the weighted set cover problem $(\U_s, \P_s, w_s)$ by $\Ss'$ (Step~\ref{step:cover_2}). With respect to $\Ss'$ the feedback matrix selected is denoted by $\bK(\Ss')$ (Step~\ref{step:edge_2}) and its cost is denoted by $P(\bK(\Ss'))$ (Step~\ref{step:cost_2}). The result below proves that  $\bK(\Ss')$ is a solution to Problem~\ref{prob:one}.


\begin{theorem}\label{th:spec_1_eps}
Consider a structured system $(\bA,\bB={\mathbb{I}_m},\bC={\mathbb{I}_p})$ and cost matrix $P$ such that Assumption~\ref{assume:spcase1} holds. Also, let $\B(\bA)$ has a perfect matching. Then, 
\begin{enumerate}
\item[(i)] $\Ss'$ is a solution to the weighted set cover problem $(\pazocal{U}_s,\pazocal{P}_s,w_s)$ constructed using Algorithm~\ref{algo:five} if and only if $\bK(\Ss')$ is a solution to Problem~\ref{prob:one}. 
\item[(ii)] $\Ss^\*$ is an optimal solution to the weighted set cover problem $(\U_s, \P_s, w_s)$ implies $\bK(\Ss^\*)$ is an optimal solution to Problem~\ref{prob:one}, i.e., $P(\bK(\Ss^\*)) = P(\bK^\*)$.
\item[(iii)] For $\epsilon \geqslant 1$, if $\Ss'$ is an $\epsilon$-optimal solution to the weighted set cover problem, then $\bK(\Ss')$ is an $\epsilon$-optimal solution to Problem~\ref{prob:one}, i.e., $w_s(\Ss')\leqslant \epsilon\, w_s(\Ss^\*)$ implies $P(\bK(\Ss'))\leqslant \epsilon\,P(\bK^\*)$.
\end{enumerate}
\end{theorem}
\begin{proof}
\textbf{(i)~Only-if part}: Here we assume that $\Ss'$ is a solution to the weighted set cover problem and then show that $\bK(\Ss')$ is a solution to Problem~\ref{prob:one}.  Note that in $\B(\bA)$ there exists a perfect matching, and hence condition~(b) in Proposition~\ref{prop:one} is satisfied without using any feedback edge. As a result, only condition~(a) has to be satisfied. Since $\Ss'$ is a solution to the weighted set cover problem, $\cup_{\Ss_d \in \Ss'}\Ss_d = \U_s = \{x_1, \ldots, x_n\}$. Consider an arbitrary state $x_i$ such that $x_i \in \Ss_j$ for some $\Ss_j \in \Ss'$. We now show that $x_i$ lies in an SCC in $\D(\bA,\bB,\bC,\bK(\Ss'))$. Note that $x_i \in \Ss_j$ implies that $x_i$ lies in an SCC in a digraph obtained by adding feedback edge $e_j = (y_b, u_a)$ to $\D(\bA,\bB,\bC,\bK^P)$   (see Step~\ref{step:Set_2}). By construction of $\bK(\Ss')$ (see Step~\ref{step:edge_2}), $\bK(\Ss')_{ab}=\*$. This concludes that $x_i$ lies in an SCC with a feedback edge in $\bK(\Ss')$. As $x_i$ is arbitrary the only-if part follows.

\noindent\textbf{(i)~If part}: Here we assume that $\widetilde{\bK}$ is a solution to Problem~\ref{prob:one} and then show that $\widetilde{\Ss}$, where $\widetilde{\Ss}:= \{\Ss_j \in \P_s$: $e_j =(y_b,u_a)$ and $\widetilde{\bK}_{ab} = \*\}$, is a solution to the weighted set cover problem. Consider an arbitrary element $x_i \in \U_s$. Since $\widetilde{\bK}$ is a solution to Problem~\ref{prob:one}, there exists some $e_j =(y_b,u_a)$ such that $\widetilde{\bK}_{ab}=\*$ and $x_i$ lies in an SCC in $\D(\bA,\bB,\bC,\widetilde{\bK})$ with feedback~edge~$e_j$. By Step~\ref{step:Set_2} of Algorithm~\ref{algo:five}, this implies that $x_i \in \Ss_j$. Since $\widetilde{\bK}_{ab}=\* $ and  $e_j=(y_b,u_a)$, by definition of $\widetilde{\Ss}$, $\Ss_j \in \widetilde{\Ss}$. Hence $\widetilde{\Ss}$ covers the element $x_i \in \U_s$. Since $x_i$ is arbitrary, the  if-part follows. This completes the proof of~$(i)$.

\noindent\textbf{(ii)}: Given $\Ss^\*$ is an optimal solution to $(\U_s, \P_s, w_s)$. By Theorem~\ref{th:spec_1_eps}~(i), $\bK(\Ss^\*)$ is a solution to Problem~\ref{prob:one}. We need to show that $\bK(\Ss^\*)$ is an optimal solution to Problem~\ref{prob:one}. Suppose not. Then there exists $\bK' \in \K$, i.e.,  a solution to Problem~\ref{prob:one},  and $P(\bK') < P(\bK(\Ss^\*))$. From if-part of Theorem~\ref{th:spec_1_eps}~(i), corresponding to $\bK'$ there exists $\widetilde{\Ss}:=\{\Ss_j: e_j=(y_b,u_a)$ and $\bK'_{ab}=\*\}$ which a solution to $(\U_s, \P_s, w_s)$.  Using Steps~\ref{step:weight_2},~\ref{step:edge_2} and~\ref{step:cost_2}, $w_s({\Ss^\*})=P(\bK(\Ss^\*))$ and $w_s(\widetilde{\Ss})=P(\bK')$.  As $P(\bK') < P(\bK(\Ss^\*))$, this implies $w_s(\widetilde{\Ss}) < w_s({\Ss^\*})$. This contradicts the fact that $\Ss^\*$ is an optimal solution to $(\U_s, \P_s, w_s)$. Thus $\bK(\Ss^\*)$ is an optimal solution to Problem~\ref{prob:one}.

\noindent\textbf{(iii)}:  Let $\Ss^\*$ and $\bK^\*$ be  optimal solutions of the weighted set cover problem $(\pazocal{U}_s,\pazocal{P}_s,w_s)$ and Problem~\ref{prob:one}, respectively. Given $w_s(\Ss') \leqslant \epsilon\, w_s(\Ss^\*)$. Now we need to show that $P(\bK(\Ss')) \leqslant \epsilon\,P(\bK^\*)$.
Since $\Ss'$ and $\Ss^\*$ are feasible solutions to the weighted set cover problem, by Theorem~\ref{th:spec_1_eps}~(i),  $\bK(\Ss')$ and $\bK(\Ss^\*)$ are feasible solutions to Problem~\ref{prob:one}. By Steps~\ref{step:weight_2},~\ref{step:edge_2}~and~\ref{step:cost_2} of Algorithm~\ref{algo:five},  $w_s(\Ss') = P(\bK(\Ss'))$ and $w_s(\Ss^\*) = P(\bK(\Ss^\*))$. Hence $P(\bK(\Ss')) \leqslant \epsilon\,P(\bK(\Ss^\*))$. From Theorem~\ref{th:spec_1_eps}~(ii), $P(\bK(\Ss^\*)) = P(\bK^\*)$. Thus $P(\bK(\Ss')) \leqslant \epsilon\,P(\bK^\*)$.
 This completes the proof.
\end{proof}

\begin{theorem}\label{th:spec_1}
Consider a structured system $(\bA,\bB=\mathbb{I}_m,\bC=\mathbb{I}_p)$ and feedback cost matrix $P$ such that Assumption~\ref{assume:spcase1} holds. Then,
\begin{enumerate}
\item[(i)] There exists an algorithm that approximates  Problem~\ref{prob:one} to factor $\mbox{log\,}n$, where $n$ is the system dimension.
\item[(ii)] Further, the log$\,n$ approximation ratio is optimal.
\end{enumerate}
\end{theorem}

\begin{proof}
\noindent\textbf{(i)}: Using Algorithm~\ref{algo:five}, any general instance of Problem~\ref{prob:one} satisfying Assumption~\ref{assume:spcase1} can be reduced to an instance of  the weighted set cover problem. Notice that Algorithm~\ref{algo:five} iterates over all the feasible feedback edges and each iteration has $O(n)$ complexity. Since $m=O(n)$ and $p=O(n)$, number of feedback edges in the system are $O(n^2)$. The remaining steps of Algorithm~\ref{algo:five} are of linear complexity. Hence the complexity of Algorithm~\ref{algo:five} is $O(n^3)$. This concludes that the reduction given in Algorithm~\ref{algo:five} is a polynomial time reduction. From Theorem~\ref{th:spec_1_eps}~(iii), an $\epsilon$-optimal solution to the weighted set cover problem gives an $\epsilon$-optimal solution to Problem~\ref{prob:one}. For solving weighted set cover problem there exists a polynomial time greedy algorithm which gives a $(\log\,N)$-optimal solution, where $N$ denotes the cardinality of the universe \cite{Chv:79}. Thus Problem~\ref{prob:one} is approximable to factor $\log\,n$, using Algorithm~\ref{algo:five} and the greedy algorithm given in \cite{Chv:79}, in polynomial time.  

\noindent\textbf{(ii)}: For a structured system satisfying Assumption~\ref{assume:spcase1}, Problem~\ref{prob:one} is inapproximable to \underline{multiplicative factor} of $\log\,n$ (Theorem~\ref{theorem:two}). Theorem~\ref{th:spec_1}~(i) proves that one can find $(\log\,n)$-optimal solution to Problem~\ref{prob:one}. Thus, the above approximation bound is optimal bound for Problem~\ref{prob:one}.
\end{proof}
We explain Algorithm~\ref{algo:five} using an illustrative example below.
\subsubsection*{Illustrative example for structured systems with back-edge feedback structure}
\begin{figure}[t]
\begin{minipage}{.22\textwidth}
\begin{tikzpicture}[scale=0.40, ->,>=stealth',shorten >=1pt,auto,node distance=1.65cm, main node/.style={circle,draw,font=\scriptsize\bfseries}]
\definecolor{myblue}{RGB}{80,80,160}
\definecolor{almond}{rgb}{0.94, 0.87, 0.8}
\definecolor{bubblegum}{rgb}{0.99, 0.76, 0.8}
\definecolor{columbiablue}{rgb}{0.61, 0.87, 1.0}


	\fill[almond] (0,0) circle (10.0 pt);  
    \node at (0,0) {\scriptsize $x_1$};
	\fill[almond] (0,-3) circle (10.0 pt);  
    \node at (0,-3) {\scriptsize $x_2$};
    \fill[almond] (2,-1.5) circle (10.0 pt);  
    \node at (2,-1.5) {\scriptsize $x_4$};
    \fill[almond] (-2,-1.5) circle (10.0 pt);  
    \node at (-2,-1.5) {\scriptsize $x_3$};
    \fill[almond] (4,-1.5) circle (10.0 pt);  
    \node at (4,-1.5) {\scriptsize $x_5$};
    
    \fill[bubblegum] (-1.5,1) circle (9.0 pt);  
    \node at (-1.5,1) {\scriptsize $u_1$};
	\fill[columbiablue] (1.5,1) circle (9.0 pt);  
    \node at (1.5,1) {\scriptsize $y_1$};
    
    \fill[bubblegum] (1.5,-4) circle (9.0 pt);      
    \node at (1.5,-4) {\scriptsize $u_2$};
    \fill[columbiablue] (-1.5,-4) circle (9.0 pt);  
    \node at (-1.5,-4) {\scriptsize $y_2$};
    
    \fill[bubblegum] (-3,0) circle (9.0 pt);  
    \node at (-3,0) {\scriptsize $u_3$};
    \fill[columbiablue] (-3,-3) circle (9.0 pt);  
    \node at (-3,-3) {\scriptsize $y_3$};
    
	\fill[bubblegum] (3,0) circle (9.0 pt);  
    \node at (3,0) {\scriptsize $u_4$};
    \fill[columbiablue] (3,-3) circle (9.0 pt);  
    \node at (3,-3) {\scriptsize $y_4$};
    
    \fill[bubblegum] (5,0) circle (9.0 pt);  
    \node at (5,0) {\scriptsize $u_5$};
    \fill[columbiablue] (5,-3) circle (9.0 pt);  
    \node at (5,-3) {\scriptsize $y_5$};
    
    \draw (0,-0.3)->(-2,-1.2);
    \draw (0,-0.3)->(2,-1.2);
    
    \draw (-2,-1.8)->(0,-2.7);
    \draw (2,-1.8)->(0,-2.7);
    
    \draw (2.3,-1.5)->(3.7,-1.5);
    
    \draw (-1.2,1)->(-0.3,0);
    \draw (0.3,0)->(1.2,1);
    
    \draw (-1.2,-4)->(-0.3,-3);
    \draw (0.3,-3)->(1.2,-4);
    
    \draw (-3,-0.3)->(-2,-1.2);
    \draw (-2,-1.8)->(-3,-2.7);
    \draw (3,-0.3)->(2,-1.2);
    \draw (2,-1.8)->(3,-2.7);
    \draw (5,-0.3)->(4,-1.2);
    \draw (4,-1.8)->(5,-2.7);     
  
;
\end{tikzpicture}
\end{minipage}\hspace{0.1 cm}
\begin{minipage}{.2\textwidth}
\caption{\small Illustrative figure of a structured system with dedicated inputs and outputs to demonstrate Algorithm~\ref{algo:five}.}
\label{fig:example1}
\end{minipage}
\vspace*{-5 mm}
\end{figure}
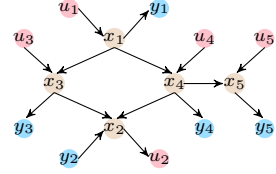

In this section, we describe Algorithm~\ref{algo:five} using the example given in Figure~\ref{fig:example1}. Let the feedback cost matrix $P$ associated with the structured system given in Figure~\ref{fig:example1} be
{\scalefont{0.9}{
$$ P=\left[
\begin{smallmatrix}
1      & 10     & 10     & 2      & 10  \\
\infty & 3      & \infty & \infty & \infty \\
\infty & 10     & 4      & \infty & \infty  \\
\infty & 10     & \infty & 2      & 8  \\
\infty & \infty & \infty & \infty & 5 \\ 
\end{smallmatrix}
\right].
$$
}}
Notice that an output $y_j$ can be given as feedback to an input $u_i$ if there exists a directed path from $u_i$ to $y_j$ in $\D(\bA,\bB,\bC)$. We reduce this instance of Problem~\ref{prob:one} to an instance of the weighted set cover problem  $(\U_s,\P_s,w_s)$ as follows. Here, the universe {\scalefont{0.9}{$\U_s=\{x_1,\ldots,x_5\}$.}} As per $P$, there are 12 feasible feedback edges. Corresponding to these edges, the sets of the weighted set cover problem {\scalefont{0.9}{$\P_s=\{\Ss_1,\ldots,\Ss_{12}\}$}} are constructed as follows: {\scalefont{0.9}{ $\Ss_1=\{x_1\}$, $\Ss_2=\{x_1,x_2,x_3,x_4\}$, $\Ss_3=\{x_1,x_3\}$, $\Ss_4=\{x_1,x_4\}$,
 $\Ss_5=\{x_1,x_4,x_5\}$,
 $\Ss_6=\{x_2\}$,
 $\Ss_7=\{x_2,x_3\}$,
 $\Ss_8=\{x_3\}$,
 $\Ss_9=\{x_2,x_4\}$,
 $\Ss_{10}=\{x_4\}$,
 $\Ss_{11}=\{x_4,x_5\}$, and 
 $\Ss_{12}=\{x_5\}$.}} The respective weights for the sets defined by matrix $P$ given above are given by, {\scalefont{0.9}{$w_s=\{1,10,10,2,10,3,10,4,10,2,8,5\}$.}} Solving weighted set cover problem for $(\U_s,\P_s,w_s)$ given above using approximation algorithm  given in \cite{Chv:79} gives a $(\log n)$-optimal solution to Problem~\ref{prob:one} (Theorem~\ref{th:spec_1}). The next section discusses the second graph topology.
\vspace*{-3 mm}
\subsection{Hierarchical Network}\label{subsec:hierar}
In this subsection, we consider a special graph topology referred as {\em layered graphs} in the literature \cite{LiuYanSlo:12}. Many real-world systems such as power grids, drinking water networks, biological cell regulation networks, online social networks, and road traffic control  can be described and modeled using a layered network structure where the states in the system interact with each other in a layered fashion \cite{FegPer:77}. Each layer in the layered structure is influenced\footnote{In a directed graph a node $v_i$ is said to be influenced by node $v_j$, if there exists a directed path from $v_j$ to $v_i$.} by the nodes in the previous layer and hence the network follows a directed tree structure called as {\em arborescence}. A directed graph following a tree structure such that every node except the root node has exactly one incoming edge is referred as a {\em hierarchical network}. Here, we aim to solve the minimum cost feedback selection problem for dedicated i/o satisfying Assumption~\ref{assume:spcase1} for structured systems whose DAG of SCCs  is a hierarchical network. 

Hierarchical network structure is common in real-life networks \cite{LiuYanSlo:12}. A power distribution system follows a hierarchical network structure and finding an optimal control strategy aims towards designing a least cost feedback pattern to  maintain the system parameters such as voltages and frequency at different layers of the network at specified levels \cite{FegPer:77}, \cite{Mar:07}. In a water distribution network, optimization techniques in controlling the network contribute towards developing a smart management strategy for implementing drinking water networks \cite{MarBarPui:12}. In case of road traffic control, a hierarchical network is a natural choice to structure the control problems \cite{VraSchSan:09}. Next, we discuss few notations and constructions required to describe a hierarchical network. 

\begin{defn}\label{defn:parent}
Consider a directed graph $\D_G:=(V_G,E_G)$. Let nodes $v_i,v_j \in V_G$ be such that there exists an edge $(v_i,v_j) \in E_G$ from $v_i$ to $v_j$. Then, $v_i$ is referred as a \underline{parent} of $v_j$.
\end{defn}
Let the DAG of SCCs in $\D(\bA)$ be denoted by $\D_{A}:=(V_{A},E_{A})$. Here the node set $V_A = \{\mN_1, \ldots, \mN_\ell\}$ is the set of all SCCs  in $\D(\bA)$ and $(\mN_i, \mN_j) \in E_A$ if there exists a directed edge in $\D(\bA)$ from a state in $\mN_i$ to a state in $\mN_j$. Then we have the following assumption on the digraph $\D_{A}$.     

\begin{assume}\label{assume:spcase2}
Consider the DAG $\D_A=(V_A,E_A)$ which consists of SCCs in $\D(\bA)$. Then, each node $\mN_i\in V_{A}$ except the root node has a unique parent, where root node is a vertex which has no incoming edge.
\end{assume}
Under Assumption~\ref{assume:spcase2}, the DAG $\D_{A}$ is a hierarchical network. For a hierarchical network, we define the notion of {\em layer} which corresponds to the position of a set of nodes in the network arrangement.
\begin{defn}\label{def:layer}
Consider  $\mN_i,\mN_j \in V_{A}$ such that there exists a directed path from $\mN_i$ to $\mN_j$ in $\D_A$. The distance from $\mN_i$ to $\mN_j$ in $\D_A$ is the number of edges in the shortest directed path from $\mN_i$ to $\mN_j$. Then, a layer $L_i$ is defined as the set of all nodes which are at a distance $i-1$ from the root node in $\D_A$. Note that $L_i\subseteq \mN$. The node set $L_i$ is represented as $L_i=\{\mN^i_1,\ldots,\mN^i_{h_i}\}$, where a node $\mN^i_j \in V_A$ denotes the $j^{\rm th}$ node in $L_i$ and $h_i$ denotes the number of nodes in $L_i$. 
\end{defn} 
An illustrative example of a hierarchical network is presented in Figure~\ref{fig:hierar}.
\begin{figure}[t]
\begin{minipage}{.25\textwidth}
\begin{tikzpicture}[scale=0.42, ->,>=stealth',shorten >=1pt,auto,node distance=1.65cm, main node/.style={circle,draw,font=\scriptsize\bfseries}]
\definecolor{myblue}{RGB}{80,80,160}
\definecolor{almond}{rgb}{0.94, 0.87, 0.8}
\definecolor{bubblegum}{rgb}{0.99, 0.76, 0.8}
\definecolor{columbiablue}{rgb}{0.61, 0.87, 1.0}

\draw[blue] (-5.5,-0.5) rectangle (5,0.5);\draw[red] (5.5,0) -> (5,0);\node at (5.8,0) {\scriptsize $L_1$};
\draw[blue] (-5.5,-2.5) rectangle (5,-1.5);\draw[red] (5.5,-2) -> (5,-2);\node at (5.8,-2) {\scriptsize $L_2$};
\draw[blue] (-5.5,-4.5) rectangle (5,-3.5);\draw[red] (5.5,-4) -> (5,-4);\node at (5.8,-4) {\scriptsize $L_3$};
\draw[blue] (-5.5,-6.5) rectangle (5,-5.5);\draw[red] (5.5,-6) -> (5,-6);\node at (5.8,-6) {\scriptsize $L_4$};
\draw[blue] (-5.5,-8.5) rectangle (5,-7.5);\draw[red] (5.5,-8) -> (5,-8);\node at (5.8,-8) {\scriptsize $L_5$};

\fill[almond] (0,0) circle (15.0 pt);  
  \node at (0,0) {\scriptsize $\mN^1_1$};

  \fill[almond] (-4,-2) circle (15.0 pt);
  \fill[almond] (0,-2) circle (15.0 pt);
  \fill[almond] (4,-2) circle (15.0 pt);
  
  \node at (-4,-2) {\scriptsize $\mN^2_1$};  
  \node at (0,-2) {\scriptsize $\mN^2_2$};  
  \node at (4,-2) {\scriptsize $\mN^2_3$};
   
  \fill[almond] (-4.8,-4) circle (15.0 pt);
  \fill[almond] (-3.2,-4) circle (15.0 pt);
   \node at (-4.8,-4) {\scriptsize $\mN^3_1$};
   \node at (-3.2,-4) {\scriptsize $\mN^3_2$};
   
   \fill[almond] (-0.8,-4) circle (15.0 pt);
   \fill[almond] (0.8,-4) circle (15.0 pt);
   \node at (-0.8,-4) {\scriptsize $\mN^3_3$};
   \node at (0.8,-4) {\scriptsize $\mN^3_4$};

  \fill[almond] (4,-4) circle (15.0 pt);
   \node at (4,-4) {\scriptsize $\mN^3_5$}; 
   \fill[almond] (0.8,-6) circle (15.0 pt);
   \node at (0.8,-6) {\scriptsize $\mN^4_4$};
   
   \fill[almond] (-3.2,-6) circle (15.0 pt);
   \fill[almond] (-4.5,-6) circle (15.0 pt);
   \fill[almond] (-1.9,-6) circle (15.0 pt);
  \node at (-3.2,-6) {\scriptsize $\mN^4_2$};
  \node at (-4.5,-6) {\scriptsize $\mN^4_1$};  
  \node at (-1.9,-6) {\scriptsize $\mN^4_3$};
  
  \fill[almond] (4,-6) circle (15.0 pt);
  \node at (4,-6) {\scriptsize $\mN^4_5$};

   \fill[almond] (0.1,-8) circle (15.0 pt);
   \fill[almond] (1.7,-8) circle (15.0 pt);
   \node at (0.1,-8) {\scriptsize $\mN^5_3$};
   \node at (1.7,-8) {\scriptsize $\mN^5_4$};
   
   \fill[almond] (-4.0,-8) circle (15.0 pt);
   \fill[almond] (-2.4,-8) circle (15.0 pt);
   \node at (-4.0,-8) {\scriptsize $\mN^5_1$};
   \node at (-2.4,-8) {\scriptsize $\mN^5_2$};

	\draw (0,-0.5)  ->   (0,-1.5);
	\draw (0,-0.5)  ->   (-4,-1.5);
	\draw (0,-0.5)  ->   (4,-1.5);
	
	\draw (-4,-2.5)  ->   (-4.8,-3.5);
	\draw (-4,-2.5)  ->   (-3.2,-3.5);
	
	\draw (0,-2.5)  ->   (-0.8,-3.5);
	\draw (0,-2.5)  ->   (0.8,-3.5);
	
	\draw (4,-2.5)  ->   (4,-3.5);
	
	\draw (0.8,-4.5)  ->   (0.8,-5.5);
	
	\draw (0.8,-6.5)  ->   (0.1,-7.5);
	\draw (0.8,-6.5)  ->   (1.7,-7.5);
	
	\draw (-3.2,-4.5) -> (-3.2,-5.5);
	\draw (-3.2,-4.5) -> (-4.4,-5.5);
	\draw (-3.2,-4.5) -> (-2.0,-5.5);
	
	\draw (4,-4.5) -> (4,-5.5);
	
	\draw (-3.2,-6.5) -> (-4.0,-7.5);
	\draw (-3.2,-6.5) -> (-2.4,-7.5);
	
	\draw [dashed,red] (-1.4,-1.2) rectangle (2.3,-8.8);
  
;

\end{tikzpicture}
\end{minipage}~\hspace*{7 mm}
\begin{minipage}{.18\textwidth}
\caption{\small A structured system whose DAG of SCCs forms a hierarchical network. Each vertex $\mN_i^j$ in the figure corresponds to an SCC of $\D(\bA)$. The subgraph enclosed in the dashed box illustrates a subtree rooted at node $\mN^2_2$ denoted by $Tree(\mN^2_2)$.}
\label{fig:hierar}
\end{minipage}
\vspace*{-4 mm}
\end{figure}
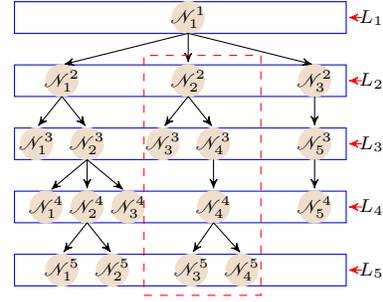
Under Definition~\ref{def:layer}, the root node of the hierarchical network is denoted by $\mN^1_1$ and it is the only node present in the top layer. Next, we define a {\em subtree} of a hierarchical network which is a subgraph of the system digraph. For $\D_A=(V_A,E_A)$, $\D_S:=(V_S,E_S)$  denotes a subgraph of $\D_A$ whose vertex set $V_S \subseteq V_A$ and edge set $E_S \subseteq E_A$, such that  endpoints of $E_S$ are nodes from $V_S$.

\begin{defn}\label{def:tree}
Consider a node $\mN^f_k\in V_A$ in the layer $L_f$. Then, a \underline{subtree} rooted at node $\mN^f_k$, denoted by $Tree(\mN^f_k)$, is defined as the subgraph in the hierarchical network which consists of the node $\mN^f_k$ and all of it's descendants.
\end{defn}
Note that $Tree(\mN^1_1)$ denotes the entire hierarchical network, where $\mN^1_1$ is the top node in the network. An illustrative example of a subtree is shown enclosed in the dashed box with respect to the hierarchical network in Figure~\ref{fig:hierar}. In this paper, we propose a dynamic programming based algorithm to solve the minimum cost feedback selection problem for dedicated i/o when the structured system is a hierarchical network. The approach is based on dividing the network into smaller subtrees (Definition~\ref{def:tree}) and finding an optimal solution for the subtrees in a bottom up fashion. Eventually we merge the solutions obtained for the smaller subtrees and find an optimal solution to the bigger network.


Consider a hierarchical network $\D_{A}$. Our aim is to find a set of minimum cost feedback edges such that the hierarchical network along with these feedback edges satisfies condition~(a) in Proposition~\ref{prop:one}. Consider a node $\mN^f_k$, where $\mN^f_k$ denotes the $k^{\rm th}$ node in layer $L_f$. Recall that $\mN^f_k$ lies in $\D_A$ which is a DAG. For $\mN^f_k$, let $A^f_k$ denotes the set of all feedback edges such that each edge in $A^f_k$ makes $\mN^f_k$ lie in a cycle. For a feedback edge $(y_b,u_a)$ to be in $A^f_k$, $(y_b,u_a)$ has to be directed from an output $y_b$ which is a descendant of $\mN^f_k$ to an input $u_a$ which is an ancestor of $\mN^f_k$. To characterize all the edges in $A^f_k$, we give the following definition.
 
\begin{defn}
Consider $\D_A$ and $\mN^f_k \in V_{A}$. The set ${\cal{A}}^f_k$ denotes the set of all state nodes that lie in the SCCs of $\D(\bA)$ which are ancestors of  $\mN^f_k$. Similarly, the set ${\cal{D}}^f_k$ denotes the set of all state nodes which lie in some SCC of $\D(\bA)$ which are descendants of $\mN^f_k$. We denote $U^f_k$ as the set of input nodes $u_i$'s which are connected to the state nodes in ${\cal{A}}^f_k$. Similarly, $Y^f_k$ denotes  the set of output nodes $y_j$'s which are connected from the state nodes in ${\cal{D}}^f_k$. Then, with respect to $\mN^f_k$, a feedback edge $(y_j,u_i)$ belongs to the edge set $A^f_k$  if $y_j \in Y^f_k$ and $u_i\in U^f_k$. A feedback edge $(y_b,u_a)$ is said to \underline{cover} $\mN^f_k$ if $(y_b,u_a)\in A^f_k$.
\end{defn}

We need to find an optimal solution to the minimum cost feedback selection problem for hierarchical networks, i.e., we need to find a set of feedback edges which cover the entire network represented by $Tree(\mN^1_1)$. The proposed algorithm is based on dynamic programming where we find solutions to the subproblems and merge them to obtain a solution for the original problem. The subproblem is to find an optimal feedback edge set to cover a general subtree $Tree(\mN^f_k)$ in the network. Next, we describe the procedure to cover a subtree $Tree(\mN^f_k)$ optimally. Consider $Tree(\mN^f_k)$ and $(y_b,u_a) \in A^f_k$. Since $Tree(\mN^f_k)$ includes $\mN^f_k$, an edge in $A^f_k$ is essential to cover the nodes in $Tree(\mN^f_k)$. Suppose we select $(y_b,u_a)$ that covers $\mN^f_k$. Note that there might be a set of nodes in $Tree(\mN^f_k)$ other than $\mN^f_k$ which are covered by the edge $(y_b,u_a)$. We need to cover the rest of the nodes in $Tree(\mN^f_k)$ which are not covered by the edge $(y_b,u_a)$. These nodes lie in a subgraph of $Tree(\mN^f_k)$ and form a set of disjoint subtrees denoted by $Forest(\mN^f_k,(y_b,u_a))$. 
\begin{defn}\label{def:forest}
Consider node $\mN^f_k$ and a feedback edge $(y_b,u_a)\in A^f_k$. Then, $Forest(\mN^f_k,(y_b,u_a))$ is defined as the subgraph of $Tree(\mN^f_k)$ which consists of the nodes in $Tree(\mN^f_k)$ which are not covered by the feedback edge $(y_b,u_a)$. The $Forest(\mN^f_k,(y_b,u_a))$ is composed of disjoint subtrees in $Tree(\mN^f_k)$.
\end{defn}
Consider an example of forest presented in Figure~\ref{fig:forest}. With respect to the node $\mN^2_2$ and the feedback edge $(y_b,u_a)$ covering the node $\mN^2_2$, the $Forest(\mN^2_2,(y_b,u_a))$ is represented by subtrees consisting of the node set $\{\mN^3_3,\mN^4_4,\mN^5_3,\mN^5_4\}$ (highlighted in green colour). Here, there are two subtrees, namely $Tree(\mN^3_3)$ and $Tree(\mN^4_4)$, in $Forest(\mN^2_2,(y_b,u_a))$.
Consider $Forest(\mN^f_k,(y_b,u_a))$, where $(y_b,u_a)\in A^f_k$. The cost to cover the disjoint subtrees in $Forest(\mN^f_k,(y_b,u_a))$ is the sum of the cost to cover the subtrees individually (Corollary~\ref{cor:spcase2-2}) and is denoted by $c(F(\mN^f_k,(y_b,u_a)))$, where $F(\mN^f_k,(y_b,u_a))$ is an optimal set of feedback edges to cover all the individual subtrees in $Forest(\mN^f_k,(y_b,u_a))$.  
\begin{figure}[t]
\begin{minipage}{.25\textwidth}
\begin{tikzpicture}[scale=0.42, ->,>=stealth',shorten >=1pt,auto,node distance=1.65cm, main node/.style={circle,draw,font=\scriptsize\bfseries}]
\definecolor{mygreen}{RGB}{80,160,80}
\definecolor{myblue}{RGB}{80,80,160}
\definecolor{almond}{rgb}{0.94, 0.87, 0.8}
\definecolor{bubblegum}{rgb}{0.99, 0.76, 0.8}
\definecolor{columbiablue}{rgb}{0.61, 0.87, 1.0}

\draw[blue] (-5.5,-0.5) rectangle (5,0.5);\draw[red] (5.5,0) -> (5,0);\node at (5.8,0) {\scriptsize $L_1$};
\draw[blue] (-5.5,-2.5) rectangle (5,-1.5);\draw[red] (5.5,-2) -> (5,-2);\node at (5.8,-2) {\scriptsize $L_2$};
\draw[blue] (-5.5,-4.5) rectangle (5,-3.5);\draw[red] (5.5,-4) -> (5,-4);\node at (5.8,-4) {\scriptsize $L_3$};
\draw[blue] (-5.5,-6.5) rectangle (5,-5.5);\draw[red] (5.5,-6) -> (5,-6);\node at (5.8,-6) {\scriptsize $L_4$};
\draw[blue] (-5.5,-8.5) rectangle (5,-7.5);\draw[red] (5.5,-8) -> (5,-8);\node at (5.8,-8) {\scriptsize $L_5$};

\fill[almond] (0,0) circle (15.0 pt);  
  \node at (0,0) {\scriptsize $\mN^1_1$};

  \fill[almond] (-4,-2) circle (15.0 pt);
  \fill[almond] (0,-2) circle (15.0 pt);
  \fill[almond] (4,-2) circle (15.0 pt);
  
  \node at (-4,-2) {\scriptsize $\mN^2_1$};  
  \node at (0,-2) {\scriptsize $\mN^2_2$};  
  \node at (4,-2) {\scriptsize $\mN^2_3$};
   
  \fill[almond] (-4.8,-4) circle (15.0 pt);
  \fill[almond] (-3.2,-4) circle (15.0 pt);
   \node at (-4.8,-4) {\scriptsize $\mN^3_1$};
   \node at (-3.2,-4) {\scriptsize $\mN^3_2$};
   
   \fill[mygreen] (-0.8,-4) circle (15.0 pt);
   \fill[almond] (0.8,-4) circle (15.0 pt);
   \node at (-0.8,-4) {\scriptsize $\mN^3_3$};
   \node at (0.8,-4) {\scriptsize $\mN^3_4$};

  \fill[almond] (4,-4) circle (15.0 pt);
   \node at (4,-4) {\scriptsize $\mN^3_5$}; 
   \fill[mygreen] (0.8,-6) circle (15.0 pt);
   \node at (0.8,-6) {\scriptsize $\mN^4_4$};
   
   \fill[almond] (-3.2,-6) circle (15.0 pt);
   \fill[almond] (-4.5,-6) circle (15.0 pt);
   \fill[almond] (-1.9,-6) circle (15.0 pt);
  \node at (-3.2,-6) {\scriptsize $\mN^4_2$};
  \node at (-4.5,-6) {\scriptsize $\mN^4_1$};  
  \node at (-1.9,-6) {\scriptsize $\mN^4_3$};
  
  \fill[almond] (4,-6) circle (15.0 pt);
  \node at (4,-6) {\scriptsize $\mN^4_5$};

   \fill[mygreen] (0.1,-8) circle (15.0 pt);
   \fill[mygreen] (1.7,-8) circle (15.0 pt);
   \node at (0.1,-8) {\scriptsize $\mN^5_3$};
   \node at (1.7,-8) {\scriptsize $\mN^5_4$};
   
   \fill[almond] (-4.0,-8) circle (15.0 pt);
   \fill[almond] (-2.4,-8) circle (15.0 pt);
   \node at (-4.0,-8) {\scriptsize $\mN^5_1$};
   \node at (-2.4,-8) {\scriptsize $\mN^5_2$};

	\draw (0,-0.5)  ->   (0,-1.5);
	\draw (0,-0.5)  ->   (-4,-1.5);
	\draw (0,-0.5)  ->   (4,-1.5);
	
	\draw (-4,-2.5)  ->   (-4.8,-3.5);
	\draw (-4,-2.5)  ->   (-3.2,-3.5);
	
	\draw (0,-2.5)  ->   (-0.8,-3.5);
	\draw (0,-2.5)  ->   (0.8,-3.5);
	
	\draw (4,-2.5)  ->   (4,-3.5);
	
	\draw (0.8,-4.5)  ->   (0.8,-5.5);
	
	\draw[mygreen, thick] (0.8,-6.5)  ->   (0.1,-7.5);
	\draw[mygreen, thick] (0.8,-6.5)  ->   (1.7,-7.5);
	
	\draw (-3.2,-4.5) -> (-3.2,-5.5);
	\draw (-3.2,-4.5) -> (-4.4,-5.5);
	\draw (-3.2,-4.5) -> (-2.0,-5.5);
	
	\draw (4,-4.5) -> (4,-5.5);
	
	\draw (-3.2,-6.5) -> (-4.0,-7.5);
	\draw (-3.2,-6.5) -> (-2.4,-7.5);
	
	\fill[bubblegum] (0,1.5) circle (10.0 pt);
  \node at (0,1.5) {\scriptsize $u_a$};
  \fill[columbiablue] (2,-5) circle (10.0 pt);
  \node at (2,-5) {\scriptsize $y_b$};
	
	\draw (0,1.2)->(0,0.4);
	\draw (1.3,-4.2)->(2,-4.7);
  
;
\path[every node/.style={font=\sffamily\small}]
(2.3,-5)  edge[red, bend right = 50] node [left] {} (0.3,1.5);
\end{tikzpicture}
\end{minipage}~\hspace*{7 mm}
\begin{minipage}{.18\textwidth}
\caption{\small Illustrative figure demonstrating forest corresponding to a node and a feedback edge in the hierarchical network given in Figure~\ref{fig:hierar}. Figure shows $Forest(\mN^2_2,(y_b,u_a))$  whose node set is  $\{\mN^3_3,\mN^4_4,\mN^5_3,\mN^5_4\}$.}
\label{fig:forest}
\end{minipage}
\vspace*{-2 mm}
\end{figure}
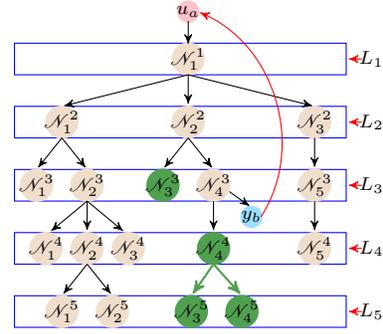
Next, we give a dynamic programming algorithm to find an optimal solution to Problem~1 under Assumptions~\ref{assume:spcase1} and~\ref{assume:spcase2}.
\begin{algorithm}
\caption{Pseudo-code to solve Problem~\ref{prob:one} for structured systems satisfying Assumptions~\ref{assume:spcase1} and~\ref{assume:spcase2}}\label{algo:hierarchical}
\textbf{Input}: Structured system $(\bA,\bB=\mathbb{I}_m,\bC=\mathbb{I}_p)$ and cost matrix $P$ satisfying Assumptions~\ref{assume:spcase1} and~\ref{assume:spcase2}\\
\textbf{Output}: Set of optimal feedback edges $H_{opt}$
\begin{algorithmic}[1]
\State Find SCCs in $\D(\bA)$, $\mN=\{\mN_1,\ldots,\mN_\ell\}$
\State Define set $L_f \leftarrow$ nodes in $\D_A$ which are at distance $f-1$ from the root node\label{step:Lf}
\State Define $\mN^f_k\leftarrow$ $k^{\rm th}$ node in layer $L_f$\label{step:SCC} 
\State Define $U^f_k \leftarrow \{u_i: \bB_{ri}=\*$ and $x_r\in {\cal{A}}^f_k\}$\label{step:U}
\State Define $Y^f_k \leftarrow \{y_j: \bC_{jr}=\*$ and $x_r\in {\cal{D}}^f_k\}$\label{step:Y}
\For {$f = \{\Delta,\dots,1\}$}\label{step:loop1}
\For {$k \in \{1,\dots,|L_f|\}$}\label{step:loop2}
\State $F(\mN^f_k,(y_j,u_i)) \leftarrow$ minimum cost edge set to keep the nodes in $Forest(\mN^f_k,(y_j,u_i))$ in cycles\label{step:W}
\State $c(F(\mN^f_k,(y_j,u_i))) \leftarrow$ cost of the edge set $F(\mN^f_k,(y_j,u_i))$\label{step:cW}
\State $A^f_k \leftarrow \{(y_j,u_i): y_j \in Y^f_k$ and $u_i\in U^f_k\}$\label{step:A1}
\State $c(Z(\mN^f_k))\leftarrow \min\limits_{(y_j,u_i)\in A^f_k} \{P_{ij}{+} c(F(\mN^f_k,(y_j,u_i)))\}$\label{step:cZ}
\State If $c(Z(\mN^f_k)) = P_{ab} + c(F(\mN^f_k,(y_b,u_a)))$, then $Z(\mN^f_k)\leftarrow (y_b,u_a) \cup F(\mN^f_k,(y_b,u_a))$, where $a \in\{1,\ldots,m\},b\in \{1,\ldots,p\}$\label{step:Z}
\EndFor
\EndFor
\State $H_{opt} = Z(\mN^1_1) $\label{step:Hopt}
\Return $H_{opt}$ and $c(Z(\mN^1_1)$
\end{algorithmic}
\end{algorithm}
The pseudo-code to find an optimal solution to Problem~\ref{prob:one} for hierarchical networks satisfying Assumption~\ref{assume:spcase1} is presented in Algorithm~\ref{algo:hierarchical}. 
Here $L_f$ (Step~\ref{step:Lf}) denotes the $f^{\rm th}$ layer in the network and $\mN^f_k$ (Step~\ref{step:SCC}) denotes the $k^{\rm th}$ node in layer $L_f$. We denote $U^f_k$ (Step~\ref{step:U}) as the set of input nodes from which there exists a directed path to the states in SCC $\mN^f_k$, and we denote $Y^f_k$ (Step~\ref{step:Y}) as the set of output nodes which have a directed path from the states in the SCC $\mN^f_k$. The algorithm iterates over two nested \textbf{for}-loops, where the first loop (Step~\ref{step:loop1}) iterates over the layers in the network and the second loop (Step~\ref{step:loop2}) iterates over the nodes in a particular layer. We start with the bottom most layer $L_\Delta$ and find the optimal cost to cover each node in layer $L_\Delta$. At layer $L_f$,  consider a particular node $\mN^f_k$. For an edge $(y_j,u_i)\in A^f_k$ (Step~\ref{step:A1}), the algorithm finds the cost to cover $Tree(\mN^f_k)$ using $(y_j,u_i)$ (Step~\ref{step:cZ}). The cost is computed as the sum of the cost of the feedback edge $(y_j,u_i)\in A^f_k$ and the cost of the edge set to cover the $Forest(\mN^f_k,(y_j,u_i))$. The feedback edge set $F(\mN^f_k,(y_j,u_i))$ denotes an optimal feedback edge set to cover $Forest(\mN^f_k,(y_j,u_i))$ (Step~\ref{step:W}) and $c(F(\mN^f_k,(y_j,u_i)))$ denotes the corresponding cost of the edge set $F(\mN^f_k,(y_j,u_i))$ (Step~\ref{step:cW}). The cost to cover $Forest(\mN^f_k,(y_j,u_i))$ is already found as the subtrees in $Forest(\mN^f_k,(y_j,u_i))$ are rooted at some descendants of the node $\mN^f_k$ and the costs to cover these subtrees individually are already computed. Next, we perform a minimization over all the feedback edges present in $A^f_k$ and select the feedback edge $(y_b,u_a)$ which results in the minimum cost to cover $Tree(\mN^f_k)$. The set of feedback edges to cover $Tree(\mN^f_k)$ is then obtained by taking the union of the optimal edge $(y_b,u_a)$ and an optimal edge set to cover the $Forest(\mN^f_k,(y_b,u_a))$ (Step~\ref{step:Z}). Eventually, the algorithm reaches the top most layer where we find the optimal cost to cover $Tree(\mN^1_1)$ (Step~\ref{step:Hopt}), which is in fact the cost to cover the entire hierarchical network.    
Next, we give the main result regarding the optimality of Algorithm~\ref{algo:hierarchical}.
\begin{theorem}\label{theorem:spcase2}
Consider a structured system $(\bA,\bB={\mathbb{I}_m},\bC={\mathbb{I}_p})$ and a feedback cost matrix $P$ satisfying Assumptions~\ref{assume:spcase1} and~\ref{assume:spcase2}.  Let $\B(\bA)$ has a perfect matching. Then, output of Algorithm~\ref{algo:hierarchical} is an optimal solution to Problem~\ref{prob:one}.
\end{theorem}

To prove Theorem~\ref{theorem:spcase2}, we state and prove the following lemma. Further, we state two corollaries extending the result of Lemma~\ref{lem:spcase2-1}. Finally, we give a proof for Theorem~\ref{theorem:spcase2}.

\begin{lem}\label{lem:spcase2-1}
Consider the nodes $\mN^f_i,\mN^g_j \in V_A$ such that there does not exist a path directed from $\mN^f_i$ to $\mN^g_j$. Let the set of feedback edges which cover the nodes $\mN^f_i$ and $\mN^g_j$ be $A^f_i$ and $A^g_j$, respectively. Then $A^f_i \cap A^g_j = \emptyset$.
\end{lem}

\begin{proof}
We prove by contradiction. Let $(y_b,u_a)$ be a feedback edge such that $(y_b,u_a) \in A^f_i \cap A^g_j$. The feedback edge $(y_b,u_a)$ is directed from output node $y_b$ to an input node $u_a$ and covers the nodes $\mN^f_i$ and $\mN^g_j$. Note that, since $y_b$ and $u_a$ are dedicated inputs and outputs and they belong to a hierarchical network, there exists atmost one directed path between $u_a$ and $y_b$. Since $(y_b,u_a)$ covers $\mN^f_i$, there exists a directed path from node $u_a$ towards node $y_b$ through the node $\mN^f_i$. Similarly, there exists a path directed from node $u_a$ towards node $y_b$ through the node $\mN^g_j$. Since there exists exactly one directed path from node $u_a$ towards node $y_b$, the nodes $\mN^f_i$ and $\mN^g_j$ must lie in single path directed from node $u_a$ to node $y_b$. This is a contradiction to our assumption that the nodes $\mN^f_i$ and $\mN^g_j$ do not lie in a directed path. Hence  $A^f_i \cap A^g_j =\emptyset$. 
\end{proof}
\vspace*{-3 mm}
\begin{cor}\label{cor:spcase2-cor1}
Consider a hierarchical network corresponding to the structured system $(\bA,\bB=\mathbb{I}_m,\bC=\mathbb{I}_p)$ and the feedback cost matrix $P$. Consider the nodes $\mN^f_i$ and $\mN^f_k$ in a layer $L_f$. Let $A^f_i$ and $A^f_k$ be the set of all feedback edges which cover the nodes $\mN^f_i$ and $\mN^f_k$, respectively. Then, $A^f_i \cap A^f_k = \emptyset$. 
\end{cor}
\begin{proof}
Note that since the nodes $\mN^f_i$ and $\mN^f_k$ belong to the same layer, there does not exist a directed path between them. Hence the proof follows from Lemma~\ref{lem:spcase2-1}.
\end{proof}
The following corollary states that an optimal feedback edge set to cover a forest composed of disjoint subtrees is the union of the optimal edge sets to cover the subtrees individually. Moreover, these edge sets are disjoint and hence their cost is equal to the sum of the costs of edge sets to cover the subtrees individually.
\begin{cor}\label{cor:spcase2-2}
Consider nodes ${\mN^f_i},{\mN^g_j} \in V_A$, such that there does not exist a path directed from node ${\mN^f_i}$ to node ${\mN^g_j}$. Let $Z'(\mN^f_i)$ and $Z'(\mN^g_j)$ be some arbitrary maximal feedback edge sets which cover $Tree({\mN^f_i})$ and $Tree({\mN^g_j})$, respectively. Then, $Z'(\mN^f_i) \cap Z'(\mN^g_j) = \emptyset$. Also, the optimal cost to cover $Tree({\mN^f_i})$ and $Tree({\mN^g_j})$ together is equal to the sum of the cost of covering $Tree({\mN^f_i})$ and $Tree({\mN^g_j})$ optimally, i.e., $c(Z(\mN^f_i)) + c(Z(\mN^g_j))$.
\end{cor}

\begin{proof}
Given there exists no directed path from node $\mN^f_i$ to $\mN^g_j$. Therefore, there exists no directed path between any node in $Tree(\mN^f_i)$ to any node in $Tree(\mN^g_j)$. Since the edge set $Z'(\mN^f_i)$ covers the nodes in $Tree(\mN^f_i)$ and $Z'(\mN^g_j)$ covers the nodes in $Tree(\mN^g_j)$, from Lemma~\ref{lem:spcase2-1}, it follows that $Z'(\mN^f_i) \cap Z'(\mN^g_j) = \emptyset$. Therefore, the cost to cover $Tree({\mN^f_i})$ and $Tree({\mN^g_j})$ is equal to the sum of the costs of the feedback edge sets $Z'(\mN^f_i)$ and $Z'(\mN^g_j)$ separately. Let $Z({\mN^f_i})$ and $Z({\mN^g_j})$ be optimal edge sets to cover $Tree({\mN^f_i})$ and $Tree({\mN^g_j})$, respectively. Since $Z({\mN^f_i}) \cap Z({\mN^g_j}) =\emptyset$, the optimal cost to cover $Tree({\mN^f_i})$ and $Tree({\mN^g_j})$ is $c(Z({\mN^f_i})) + c(Z({\mN^g_j}))$. This completes the proof.                                                                                                                                                                                                                                                                                                                                                                                                                                                                                                                                                                                                                                                                                                                                                                                                                                                                                                                                                                                                                                                                                                                                                                                                                                                                                                                                                                                                                                                                                                                                                                                                                                                                                                                                                                                                                                                                                                                                                                                                                                                               
\end{proof}

Next, we prove Theorem~\ref{theorem:spcase2} to show the optimality of Algorithm~\ref{algo:hierarchical} and give complexity of Algorithm~\ref{algo:hierarchical} in Theorem~\ref{theorem:algo_hierar_complex}. 

\noindent{\em Proof~of~Theorem~\ref{theorem:spcase2}:}
 We prove Theorem~\ref{theorem:spcase2} using an induction argument. The induction hypothesis states that $Z(\mN^f_i)$ is an optimal set of feedback edges such that the nodes in $Tree(\mN^f_i)$ lie in cycles with feedback edges in $Z(\mN^f_i)$.
 
\noindent\textbf{Base Step:} We consider $k=\Delta$ as the base case. Consider node $\mN^{\Delta}_j$ in layer $L_{\Delta}$ and the feedback edge set $A^{\Delta}_j$. Note that $A^{\Delta}_j$ consists of all feedback edges that can make the node $\mN^{\Delta}_j$ lie in a cycle with a feedback edge. For  $k=\Delta$, we find the minimum cost to cover the subtree rooted at $\mN^{\Delta}_j$. Since $L_\Delta$ is the lowest layer in the hierarchical network, $\mN^{\Delta}_j$ is a leaf node in $\D_{A}$.  Thus, for any feedback edge $(y_b,u_a)\in A^{\Delta}_j$, the $Forest(\mN^{\Delta}_j,(y_b,u_a))=\emptyset$. Hence the edge set $F(\mN^{\Delta}_j,(y_b,u_a))=\emptyset$ and $c(F(\mN^{\Delta}_i,(y_a,u_b)))= 0$. Thus we need to find the minimum cost to cover the node $\mN^{\Delta}_j$ only.  Therefore, the minimum cost edge set $Z(\mN^{\Delta}_j)$ to cover the node $\mN^{\Delta}_j$ is given by $Z(\mN^{\Delta}_j) = \arg\min_{(y_b,u_a) \in A^{\Delta}_j} P_{ab}$. Thus, for each node $\mN^{\Delta}_j$ in the lowest layer $L_\Delta$, Algorithm~\ref{algo:hierarchical} selects a minimum cost feedback edge in $A^{\Delta}_j$ for each $\mN^{\Delta}_j \in L_\Delta$. As a consequence of Corollary~\ref{cor:spcase2-cor1}, the algorithm finds a minimum cost feedback edge to cover each node in $L_\Delta$ independently. This completes the base step.

\noindent\textbf{Induction Step:}
For the induction step, we assume that the algorithm gives an optimal feedback edge set to cover all the nodes in layers $L_{k+1},\ldots,L_{{\Delta}}$, i.e., the cost to cover each subtree rooted at nodes in $\cup_{s=k+1}^{{\Delta}}L_{s}$. Then the collection $\{Z(\mN^s_j)$: $s\in 1,\ldots,k+1$ and $j\in 1,\ldots,|L_s|\}$ is the collection of optimal edge sets to cover all the subtrees whose root nodes are nodes present in layer $L_{k+1}$ and below it. Now, we will prove that $Z(\mN^k_j)$ is an optimal set of feedback edges to cover subtree $Tree(\mN^k_j)$ for each node $\mN^k_j$ in layer $L_k$, i.e., the algorithm gives the optimal cost to cover the subtrees rooted at the nodes in layer $L_k$.  Note that $A^k_j$ consists of all the feedback edges which can cover $\mN^k_j$. Since $\mN^k_j$ lies in $Tree(\mN^k_j)$, an edge $(y_b,u_a)\in A^k_j$ is essential to cover $Tree(\mN^k_j)$. Then, cost to cover $Tree(\mN^k_j)$ using some feedback edge $(y_b,u_a) \in A^k_j$ is given by $c(F(\mN^k_j,(y_b,u_a))) + P_{ab}$, where $c(F(\mN^k_j,(y_b,u_a)))$ is the optimal cost to cover $Forest(\mN^k_j,(y_b,u_a))$. As a consequence of Corollary~\ref{cor:spcase2-2}, the optimal cost of covering $Forest(\mN^k_j,(y_b,u_a))$ is the sum of the optimal costs of covering the subtrees present in the forest independently and since the optimal costs to cover these subtrees are already found (induction step assumption), we have the optimal cost to cover $Forest(\mN^k_j,(y_b,u_a))$. Therefore, the optimal cost to cover $Tree(\mN^k_j)$ using a particular feedback edge $(y_b,u_a) \in A^k_j$ is given as $c(F(\mN^k_j,(y_b,u_a))) + P_{ab}$. Since we perform the minimization of the cost over all the feedback edges in $A^k_j$, we obtain the optimal cost to cover $Tree(\mN^k_j)$. Further, $Z(\mN^k_j)$ is the union of the feedback edge $(y_b,u_a)$ selected in the minimization step and the edge set $F(\mN^k_j,(y_b,u_a))$. Thus $Z(\mN^k_j)$ is an optimal feedback edge set to cover $Tree(\mN^k_j)$. After the final iteration for the top layer $L_1$, we obtain an optimal edge set $Z(\mN^1_1)$ to cover $Tree(\mN^1_1)$, which in fact is the hierarchical network. This completes the proof of Theorem~\ref{theorem:spcase2}.
\qed
\begin{theorem}\label{theorem:algo_hierar_complex}
Consider a structured system $(\bA,\bB=\mathbb{I}_m,\bC=\mathbb{I}_p)$ and the feedback cost matrix $P$. Then, Algorithm~\ref{algo:hierarchical} which takes as input the hierarchical network corresponding to $(\bA,\bB=\mathbb{I}_m,\bC=\mathbb{I}_p)$ and feedback cost matrix $P$ and outputs an optimal cost feedback edge set to solve Problem~\ref{prob:one} has complexity of $O(n^3)$, where $n$ denotes the system dimension.
\end{theorem}

\begin{proof}
The number of subtrees possible in the hierarchical network is equal to the number of SCCs in $\D(\bA)$ which is of the order of $n$. The minimization step in Algorithm~\ref{algo:hierarchical} is performed for all the feedback edges which cover a node $\mN^f_i \in V_A$, which is of the order of $|E_K|$. Therefore, the complexity of Algorithm~\ref{algo:hierarchical} is $O(n|E_K|)$. Since $m=O(n)$ and $p=O(n)$, the number of feedback edges in the system is $O(n^2)$. Thus the complexity of Algorithm~\ref{algo:hierarchical} is $O(n^3)$.
\end{proof}            
\vspace*{-2 mm}
\subsubsection*{Illustrative example for hierarchical network}
 \begin{figure}[t]
\hspace*{-6 mm} \begin{minipage}{.28\textwidth}
\begin{tikzpicture}[scale=0.45, ->,>=stealth',shorten >=1pt,auto,node distance=1.65cm, main node/.style={circle,draw,font=\scriptsize\bfseries}]
\definecolor{myblue}{RGB}{80,80,160}
\definecolor{almond}{rgb}{0.94, 0.87, 0.8}
\definecolor{bubblegum}{rgb}{0.99, 0.76, 0.8}
\definecolor{columbiablue}{rgb}{0.61, 0.87, 1.0}


\fill[almond] (0,0) circle (15.0 pt);  
  \node at (0,0) {\scriptsize $\mN^1_1$};

  \fill[almond] (-2.5,-2) circle (15.0 pt);
  \fill[almond] (2.5,-2) circle (15.0 pt);
  
  \node at (-2.5,-2) {\scriptsize $\mN^2_1$};  
  \node at (2.5,-2) {\scriptsize $\mN^2_3$};
   
  \fill[almond] (-3.5,-4) circle (15.0 pt);
  \fill[almond] (-1.5,-4) circle (15.0 pt);
   \node at (-3.5,-4) {\scriptsize $\mN^3_1$};
   \node at (-1.5,-4) {\scriptsize $\mN^3_2$};
   

  \fill[almond] (2.5,-4) circle (15.0 pt);
   \node at (2.5,-4) {\scriptsize $\mN^3_5$}; 

\fill[bubblegum] (-1.5,0) circle (9.0 pt);
   \node at (-1.5,0) {\scriptsize $u_1$};
   \fill[columbiablue] (0,-1.5) circle (9.0 pt);
   \node at (0,-1.5) {\scriptsize $y_1$};
   \fill[bubblegum] (2.5,-0.5) circle (9.0 pt);
   \node at (2.5,-0.5) {\scriptsize $u_2$};
   \fill[bubblegum] (-2.5,-0.5) circle (9.0 pt);
   \node at (-2.5,-0.5) {\scriptsize $u_3$};
   \fill[columbiablue] (4,-2) circle (9.0 pt);
   \node at (4,-2) {\scriptsize $y_2$};
   \fill[columbiablue] (-2.5,-3.5) circle (9.0 pt);
   \node at (-2.5,-3.5) {\scriptsize $y_3$};
   
   \fill[bubblegum] (-3.5,-2.5) circle (9.0 pt);
   \node at (-3.5,-2.5) {\scriptsize $u_4$};
   \fill[columbiablue] (-3.5,-5.5) circle (9.0 pt);
   \node at (-3.5,-5.5) {\scriptsize $y_4$};
   \fill[bubblegum] (-1.5,-2.5) circle (9.0 pt);
   \node at (-1.5,-2.5) {\scriptsize $u_5$};
   \fill[columbiablue] (-1.5,-5.5) circle (9.0 pt);
   \node at (-1.5,-5.5) {\scriptsize $y_5$};
   \fill[bubblegum] (1,-4) circle (9.0 pt);
   \node at (1,-4) {\scriptsize $u_6$};
   \fill[columbiablue] (4,-4) circle (9.0 pt);
   \node at (4,-4) {\scriptsize $y_6$};

   \draw (-1.2,0)->(-0.4,0);
   \draw (0,-0.5)->(0,-1.2);
   \draw (-2.5,-0.8)->(-2.5,-1.6);
   \draw (-2.5,-2.5)->(-2.5,-3.2);
   \draw (2.5,-0.8)->(2.5,-1.6);
   \draw (2.9,-2)->(3.7,-2);
   
   \draw (-3.5,-2.8)->(-3.5,-3.6);
   \draw (-3.5,-4.5)->(-3.5,-5.2);
   \draw (-1.5,-2.8)->(-1.5,-3.6);
   \draw (-1.5,-4.5)->(-1.5,-5.2);
   \draw (1.3,-4)->(2.1,-4);
   \draw (2.9,-4)->(3.7,-4);

	\draw (0,-0.5)  ->   (-2.5,-1.5);
	\draw (0,-0.5)  ->   (2.5,-1.5);
	
	\draw (-2.5,-2.5)  ->   (-3.5,-3.5);
	\draw (-2.5,-2.5)  ->   (-1.5,-3.5);
	
	
	\draw (2.5,-2.5)  ->   (2.5,-3.5);

;
\path[every node/.style={font=\sffamily\small}]

(-3.8,-5.5)  edge[red, bend left = 92] node [left] {} (-1.7,0.2)
(-1.2,-5.5) edge[red, bend right = 92] node [left] {} (-1.5,-2.2)
(4.3,-4)  edge[red, bend right = 95] node [left] {} (2.5,-0.2);
\end{tikzpicture}
\end{minipage}\hspace{0.25 cm}
\begin{minipage}{.21\textwidth}
\caption{\small Illustrative example of a structured system with a hierarchical network topology to  demonstrate Algorithm~\ref{algo:hierarchical}. The set of optimal edges $\{(y_4,u_1),(y_5,u_5),(y_6,u_2)\}$ obtained by Algorithm~\ref{algo:hierarchical} are shown in red.}
\label{fig:example}
\end{minipage}
\vspace*{-8 mm}
\end{figure}
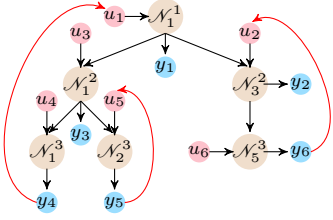

In this section, we describe Algorithm~\ref{algo:hierarchical} using the example illustrated in Figure~\ref{fig:example}. In the hierarchical network, there are three layers, $\{L_1,L_2,L_3\}$, and six SCCs, $\{\mN^1_1,\mN^2_1,\mN^2_2,\mN^3_1,\mN^3_2,\mN^3_3\}$. Corresponding to the six input and output nodes, let the feedback cost matrix be 
{\scalefont{0.9}{
$$ P=\left[
\begin{smallmatrix}
1      & 10     & 10     & 2      & 10     & 10 \\
\infty & 3      & \infty & \infty & \infty & 2 \\
\infty & \infty & 1      & 10     & 10     & \infty \\
\infty & \infty & \infty & 1      & \infty & \infty \\
\infty & \infty & \infty & \infty & 1      & \infty \\
\infty & \infty & \infty & \infty & \infty & 1 
\end{smallmatrix}
\right]
$$
}}
We need to select an optimal set of feedback edges such that the six SCCs in this network satisfies condition~(a) in Proposition~\ref{prop:one} optimally. For each SCC $\mN^f_k$, the corresponding set of feedback edges $A^f_k$ covering $\mN^f_k$ are as follows:
{\scalefont{0.9}{
\begin{eqnarray*}
A^1_1&\hspace*{-2 mm}=&\hspace*{-2 mm}\{(y_1,u_1),(y_2,u_1),(y_3,u_1),(y_4,u_1),(y_5,u_1),(y_6,u_1)\}\\ 
A^2_1&\hspace*{-2 mm}=&\hspace*{-2 mm}\{(y_3,u_1),(y_4,u_1),(y_5,u_1),(y_3,u_3)\}\\
A^2_2&\hspace*{-2 mm}=&\hspace*{-2 mm}\{(y_2,u_1),(y_6,u_1),(y_6,u_2),(y_2,u_2)\}\\
A^3_1&\hspace*{-2 mm}=&\hspace*{-2 mm}\{(y_4,u_1),(y_4,u_3),(y_4,u_4)\}\\ 
A^3_2&\hspace*{-2 mm}=&\hspace*{-2 mm}\{(y_5,u_1),(y_5,u_3),(y_5,u_5)\}\\ 
A^3_3&\hspace*{-2 mm}=&\hspace*{-2 mm}\{(y_6,u_1),(y_6,u_2),(y_6,u_6)\}
\end{eqnarray*}
}}
In the first iteration ($f=3$) we select the layer $L_3$. Our aim is to cover each subtree rooted at some node in layer $L_3$, i.e., subtrees rooted at each SCC $\mN^3_k \in L_3$.\\
For $Tree(\mN^3_1)$, $c(Z(\mN^3_1))=$  
{\scalefont{0.9}{
$$\min
\begin{Bmatrix}
P_{14}+ c(F(\mN^3_1,(y_4,u_1)))\\
P_{34}+ c(F(\mN^3_2,(y_4,u_3)))\\
P_{44}+ c(F(\mN^3_3,(y_4,u_4)))
\end{Bmatrix}
=\min
\begin{Bmatrix}
2+0\\
10+0\\
\bf{1+0}
\end{Bmatrix}
=1
$$
}}
and $Z(\mN^3_1)=(y_4,u_4)$. For $Tree(\mN^3_2)$, $c(Z(\mN^3_2))=$ 
{\scalefont{0.9}{
$$\min
\begin{Bmatrix}
P_{15}+ c(F(\mN^3_2,(y_5,u_1)))\\
P_{35}+ c(F(\mN^3_2,(y_5,u_3)))\\
P_{55}+ c(F(\mN^3_2,(y_5,u_5)))
\end{Bmatrix}
=\min
\begin{Bmatrix}
10+0\\
10+0\\
\bf{1+0}
\end{Bmatrix}
=1
$$
}}
and $Z(\mN^3_2)=(y_5,u_5)$. For $Tree(\mN^3_3)$, $c(Z(\mN^3_3))=$
{\scalefont{0.9}{
$$\min
\begin{Bmatrix}
P_{16}+ c(F(\mN^3_3,(y_6,u_1)))\\
P_{26}+ c(F(\mN^3_3,(y_6,u_2)))\\
P_{66}+ c(F(\mN^3_3,(y_6,u_6)))
\end{Bmatrix}
=\min
\begin{Bmatrix}
10+0\\
2+0\\
\bf{1+0}
\end{Bmatrix}
=1
$$
}}
and $Z(\mN^3_3)=(y_6,u_6)$.
In the next iteration ($f=2$), our aim is to cover each subtree rooted at some node  in layer $L_2$. 
For $Tree(\mN^2_1)$, $c(Z(\mN^2_1)){=}$ 
{\scalefont{0.9}{
$$\min
\begin{Bmatrix}
P_{14} + c(F(\mN^2_1,(y_4,u_1)))\\
P_{34} + c(F(\mN^2_1,(y_4,u_3)))\\
P_{15} + c(F(\mN^2_1,(y_5,u_1)))\\
P_{35} + c(F(\mN^2_1,(y_5,u_3)))\\
P_{33} + c(F(\mN^2_1,(y_3,u_3)))\\
P_{13} + c(F(\mN^2_1,(y_3,u_1))) 
\end{Bmatrix}   
{=}\min
\begin{Bmatrix}
\bf{2{+}1}\\
10{+}1\\
10{+}1\\
10{+}1\\
1{+}1{+}1\\
10{+}1{+}1
\end{Bmatrix} 
{=}3
$$
}}
and $Z(\mN^2_1)$=$\{(y_4,u_1)\}\cup F(\mN^2_1,(y_4,u_1))$
{=}$\{(y_4,u_1),(y_5,u_5)\}$.
For $Tree(\mN^2_2)$, $c(Z(\mN^2_1))=$
{\scalefont{0.9}{$$\min
\begin{Bmatrix}
P_{16}+c(F(\mN^2_2,(y_6,u_1)))\\
P_{26}+c(F(\mN^2_2,(y_6,u_2)))\\
P_{12}+c(F(\mN^2_2,(y_2,u_1)))\\
P_{22}+c(F(\mN^2_2,(y_2,u_2)))
\end{Bmatrix}
 = \min
\begin{Bmatrix} 
 10+0\\
 \bf{2+0}\\
 10+1\\
 3+1
 \end{Bmatrix} 
 = 2
 $$
 }}
and $Z(\mN^2_2)$ = $\{(y_6,u_2)\} \cup F(\mN^2_2,(y_6,u_2))$ = $\{(y_6,u_2)\}$.
In the final iteration ($f=1$), our aim is to cover each subtree rooted at some node in layer $L_1$, i.e., $Tree(\mN^1_1)$ which is the entire hierarchical network.
For $Tree(\mN^1_1)$, $c(Z(\mN^1_1)=$
{\scalefont{0.9}{
$$\min
\begin{Bmatrix}
P_{16} + c(F(\mN^1_1,(y_6,u_1)))\\
P_{15} + c(F(\mN^1_1,(y_5,u_1)))\\
P_{14} + c(F(\mN^1_1,(y_4,u_1)))\\
P_{13} + c(F(\mN^1_1,(y_3,u_1)))\\
P_{12} + c(F(\mN^1_1,(y_2,u_1)))\\
P_{11} + c(F(\mN^1_1,(y_1,u_1)))
\end{Bmatrix}
= \min
\begin{Bmatrix}
10{+}3\\
10{+}1{+}2\\
\bf{2{+}1{+}2}\\
10{+}1{+}1{+}2\\
10{+}1{+}3\\
1{+}3{+}2
\end{Bmatrix}
{=}5
$$
}}
and $Z(\mN^1_1)$ = $\{(y_4,u_1)\}\cup F(\mN^1_1,(y_4,u_1))$ = $\{(y_4,u_1),(y_5,u_5),(y_6,u_2)\}$.
Thus $Z(\mN^1_1)$ is an optimal feedback edge set to cover all the nodes in the digraph using a feedback edge and the optimal solution to Problem~\ref{prob:one} is given by 
{\scalefont{0.9}{
$$ \bK^\*=\left[
\begin{smallmatrix}
0&0&0& \* & 0 & 0 \\
0&0&0& 0  & 0 & \* \\
0&0&0& 0  & 0 & 0 \\
0&0&0& 0  & 0 & 0 \\
0&0&0& 0  & \*& 0 \\
0&0&0& 0  & 0 & 0
\end{smallmatrix}
\right].
$$
}}
\begin{rem}\label{rem:hierar}
Consider a structured system $(\bA,\bB,\bC)$ and feedback cost matrix $P$ such that the DAG of SCCs of the system consists of multiple hierarchical networks with distinct root nodes and disjoint node sets. Then all the analysis and results discussed in Subsection~\ref{subsec:hierar} still hold. In such a case, Algorithm~\ref{algo:hierarchical} is implemented separately on each of the hierarchical networks and by combining the solutions obtained gives an optimal solution to Problem~\ref{prob:one}. This gives a generalization of the structured systems considered in Subsection~\ref{subsec:hierar}.
\end{rem}
\section{Conclusion}\label{sec:conclu}
This paper addressed the following optimization problem: given a structured system with {\em dedicated} inputs and outputs and a feedback cost matrix, where each entry denotes the cost of the individual feedback connection, the objective is to obtain an optimal set of feedback edges that guarantees arbitrary pole-placement of the closed-loop structured system. This problem is referred as the optimal feedback selection problem with dedicated inputs and outputs. We proved the NP-hardness of this problem using a reduction from a known NP-hard problem, the weighted set cover problem (Theorem~\ref{theorem:one}). Later it is also shown that the problem is inapproximable to a multiplicative factor of ${\rm log\,}n$, where $n$ denotes the number of states in the system (Theorem~\ref{theorem:two}). We then proposed an algorithm that incorporates a {\em greedy scheme} with a {\em potential function} to solve this problem (Algorithm~\ref{algo:four}). This algorithm is shown to attain a solution with guaranteed approximation ratio in pseudo-polynomial time (Theorem~\ref{theorem:four}). The proposed algorithm has limitations regarding the  pseudo-polynomial time complexity. We then considered two special cases, namely structured systems with a {\em back-edge} feedback structure and structured systems satisfying a {\em hierarchical network} topology. These topologies find application in many real time networks like power networks, water distribution networks and social organization networks.  For the first class of systems, we show that Problem~\ref{prob:one} is NP-hard and also inapproximable to multiplicative factor of $\log\,n$ (Corollary~\ref{cor:spcase1}). We then provide a $(\log\,n)$-optimal approximation algorithm for this class of systems (Algorithm~\ref{algo:five} and~Theorem~\ref{th:spec_1}). For hierarchical networks, a polynomial time algorithm based on dynamic programming is proposed (Algorithm~\ref{algo:hierarchical}) and the optimality of the solution is proved (Theorem~\ref{theorem:spcase2}). Investigating other network topologies of practical importance and developing computationally efficient algorithms is part of future work. 
\bibliographystyle{myIEEEtran}  
\bibliography{Cycle}

\begin{thebibliography}{10}
\providecommand{\url}[1]{#1}
\csname url@rmstyle\endcsname
\providecommand{\newblock}{\relax}
\providecommand{\bibinfo}[2]{#2}
\providecommand\BIBentrySTDinterwordspacing{\spaceskip=0pt\relax}
\providecommand\BIBentryALTinterwordstretchfactor{4}
\providecommand\BIBentryALTinterwordspacing{\spaceskip=\fontdimen2\font plus
\BIBentryALTinterwordstretchfactor\fontdimen3\font minus
  \fontdimen4\font\relax}
\providecommand\BIBforeignlanguage[2]{{%
\expandafter\ifx\csname l@#1\endcsname\relax
\typeout{** WARNING: IEEEtran.bst: No hyphenation pattern has been}%
\typeout{** loaded for the language `#1'. Using the pattern for}%
\typeout{** the default language instead.}%
\else
\language=\csname l@#1\endcsname
\fi
#2}}

\bibitem{ComDio:15}
C.~Commault and J.-M. Dion, ``The single-input minimal controllability problem
  for structured systems,'' \emph{Systems \& Control Letters}, vol.~80, pp.
  50--55, 2015.

\bibitem{ComDioWou:02}
C.~Commault, J.-M. Dion, and J.~W. van~der Woude, ``Characterization of generic
  properties of linear structured systems for efficient computations,''
  \emph{Kybernetika}, vol.~38, no.~5, pp. 503--520, 2002.

\bibitem{KalBelSiv:13}
R.~K. Kalaimani, M.~N. Belur, and S.~Sivasubramanian, ``Generic pole
  assignability, structurally constrained controllers and unimodular
  completion,'' \emph{Linear Algebra and its Applications}, vol. 439, no.~12,
  pp. 4003--4022, 2013.

\bibitem{Ols:14}
A.~Olshevsky, ``Minimal controllability problems,'' \emph{IEEE Transactions on
  Control of Network Systems}, vol.~1, no.~3, pp. 249--258, 2014.

\bibitem{PeqKarAgu_2:16}
S.~Pequito, S.~Kar, and A.~P. Aguiar, ``A framework for structural input/output
  and control configuration selection in large-scale systems,'' \emph{IEEE
  Transactions on Automatic Control}, vol.~61, no.~2, pp. 303--318, 2016.

\bibitem{Rei:88}
K.~J. Reinschke, \emph{Multivariable Control: a Graph Theoretic
  Approach}.\hskip 1em plus 0.5em minus 0.4em\relax Springer-Verlag, 1988.

\bibitem{LiuBar:16}
Y.-Y. Liu and A.-L. Barab{\'a}si, ``Control principles of complex systems,''
  \emph{Reviews of Modern Physics}, vol.~88, no.~3, pp. 035\,006:1--58, 2016.

\bibitem{UnySez:89}
K.~{\"U}nyelio${\breve{g}}$lu and M.~E. Sezer, ``Optimum feedback patterns in
  multivariable control systems,'' \emph{International Journal of Control},
  vol.~49, no.~3, pp. 791--808, 1989.

\bibitem{PeqKarAgu:16}
S.~Pequito, S.~Kar, and A.~P. Aguiar, ``Minimum cost input/output design for
  large-scale linear structural systems,'' \emph{Automatica}, vol.~68, pp.
  384--391, 2016.

\bibitem{MooChaBel:17_Automatica}
\BIBentryALTinterwordspacing
S.~{Moothedath}, P.~{Chaporkar}, and M.~N. {Belur}, ``{Approximating
  Constrained Minimum Cost Input-Output Selection for Generic Arbitrary Pole
  Placement in Structured Systems},'' \emph{ArXiv e-prints}, May 2017.
  [Online]. Available: \url{http://adsabs.harvard.edu/abs/2017arXiv170509600M}
\BIBentrySTDinterwordspacing

\bibitem{MooChaBel:17_Erratum}
\BIBentryALTinterwordspacing
S.~{Moothedath}, P.~{Chaporkar}, and M.~N. {Belur}, ``{Optimal Feedback
  Selection for Structurally Cyclic Systems with Dedicated Actuators and
  Sensors},'' \emph{Conditionally accepted in IEEE Transactions on Automatic
  Control}, 2018. [Online]. Available: \url{https://arxiv.org/abs/1706.07928}
\BIBentrySTDinterwordspacing

\bibitem{MooChaBel:17_Line}
S.~Moothedath, P.~Chaporkar, and M.~N. Belur, ``Minimum cost feedback selection
  for arbitrary pole placement in structured systems,'' \emph{IEEE Transactions
  on Automatic Control}, 2018.

\bibitem{CarPeqAguKarPap:15}
J.~F. Carvalho, S.~Pequito, A.~P. Aguiar, S.~Kar, and G.~J. Pappas, ``Static
  output feedback: on essential feasible information patterns,'' in
  \emph{Proceedings of the IEEE Conference on Decision and Control (CDC)},
  Osaka, Japan, 2015, pp. 3989--3994.

\bibitem{PicSezSil:84}
V.~Pichai, M.~Sezer, and D.~{\v{S}}iljak, ``A graph-theoretic characterization
  of structurally fixed modes,'' \emph{Automatica}, vol.~20, no.~2, pp.
  247--250, 1984.

\bibitem{CorLeiRivSte:01}
T.~H. Cormen, C.~E. Leiserson, R.~L. Rivest, and C.~Stein, \emph{Introduction
  to Algorithms}.\hskip 1em plus 0.5em minus 0.4em\relax MIT press: Cambridge,
  2001.

\bibitem{Die:00}
R.~Diestel, \emph{Graph Theory}.\hskip 1em plus 0.5em minus 0.4em\relax
  Springer: New York, 2000.

\bibitem{Chv:79}
V.~Chvatal, ``A greedy heuristic for the set-covering problem,''
  \emph{Mathematics of Operations Research}, vol.~4, no.~3, pp. 233--235, 1979.

\bibitem{Fei:98}
U.~Feige, ``A threshold of ln $n$ for approximating set cover,'' \emph{Journal
  of the ACM}, vol.~45, no.~4, pp. 634--652, 1998.

\bibitem{ChaMes:13}
A.~Chapman and M.~Mesbahi, ``On strong structural controllability of networked
  systems: A constrained matching approach,'' in \emph{Proceedings of the IEEE
  American Control Conference (ACC)}, Washington DC, USA, 2013, pp. 6126--6131.

\bibitem{Joh:75}
D.~B. Johnson, ``Finding all the elementary circuits of a directed graph,''
  \emph{SIAM Journal on Computing}, vol.~4, no.~1, pp. 77--84, 1975.

\bibitem{LiuYanSlo:12}
Y.-Y. Liu, J.-J. Slotine, and A.-L. Barab{\'a}si, ``Control centrality and
  hierarchical structure in complex networks,'' \emph{Plos one}, vol.~7, no.~9,
  pp. e44\,459:1--7, 2012.

\bibitem{FegPer:77}
P.~F. Kenneth A.~Fegley, ``Hierarchical control of a multiarea power grid,''
  \emph{IEEE Transactions on Systems, Man, and Cybernetics}, vol.~7, no.~7, pp.
  545--551, July 1977.

\bibitem{Mar:07}
M.~D. Ilic, ``From hierarchical to open access electric power systems,''
  \emph{Proceedings of the IEEE}, vol.~95, pp. 1060--1084, 2007.

\bibitem{MarBarPui:12}
C.~Ocampo-Martinez, D.~Barcelli, V.~Puig, and A.~Bemporad, ``Hierarchical and
  decentralised model predictive control of drinking water networks:
  Application to barcelona case study,'' \emph{IET Control Theory
  Applications}, vol.~6, no.~1, pp. 62--71, January 2012.

\bibitem{VraSchSan:09}
J.~Vrancken, J.~H. van Schuppen, M.~dos Santos~Soares, and F.~Ottenhof, ``A
  hierarchical network model for road traffic control,'' in \emph{Proceedings
  of the IEEE International Conference on Networking, Sensing and Control},
  Okayama, Japan, 2009, pp. 340--344.

\end{thebibliography}
\vspace*{-15 mm}
\begin{IEEEbiographynophoto}
{Aishwary Joshi} is pursuing his Dual Degree (B.Tech. + M.Tech.) in Electrical Engineering with specialisation in Communication and Signal Processing from Indian Institute of Technology Bombay, India. His research interests include graph theory, optimization, algorithms and computational complexity.
\end{IEEEbiographynophoto}
\vspace*{-15 mm}
\begin{IEEEbiographynophoto}
{Shana Moothedath} obtained her B.Tech. and M.Tech. in Electrical and Electronics Engineering from Kerala
University, India in 2011 and 2014 respectively. Currently she is pursuing Ph.D. in the
Department of Electrical Engineering, Indian Institute of Technology Bombay. Her research interests include matching or allocation problem, structural analysis of control systems, combinatorial optimization and applications of graph theory.
\end{IEEEbiographynophoto}
\vspace*{-15 mm}
\begin{IEEEbiographynophoto}
{Prasanna Chaporkar} received his M.S. in Faculty of Engineering from Indian Institute of Science, Bangalore, India in 2000, and Ph.D. from University of Pennsylvania, Philadelphia, PA in 2006. He was an ERCIM post-doctoral fellow at ENS, Paris, France and NTNU, Trondheim, Norway. Currently, he is an Associate Professor at Indian Institute of Technology Bombay. His research interests are in resource allocation, stochastic control, queueing theory, and distributed systems and algorithms.
\end{IEEEbiographynophoto}
\end{document}